\documentclass[a4paper,11pt, reqno]{amsart}

\usepackage[margin=2cm]{geometry}
\usepackage[utf8]{inputenc}
\usepackage[T1]{fontenc}
\usepackage{lmodern}
\usepackage[french,english]{babel}
\usepackage{amsthm, amssymb, amsmath, amsfonts, mathrsfs, dsfont, esint, textcomp, bm}
\usepackage[colorlinks=true, pdfstartview=FitV, linkcolor=blue, citecolor=blue, urlcolor=blue,pagebackref=false]{hyperref}
\usepackage{mathtools}
\usepackage{tikz}
\usepackage{enumitem, imakeidx}
\usetikzlibrary{patterns}

\usepackage{microtype}
\definecolor{darkblue}{rgb}{0,0,0.7} 
\definecolor{darkred}{rgb}{0.9,0.1,0.1}
\definecolor{darkgreen}{rgb}{0,0.5,0}

\newtheorem{thm}{Theorem}[section]

\newtheorem{lem}[thm]{Lemma}

\theoremstyle{remark}
\newtheorem{rem}[thm]{Remark}
\theoremstyle{definition}

\newcommand{\blue}{\color{blue}}

\renewcommand{\leq}{\leqslant}
\renewcommand{\geq}{\geqslant}

\renewcommand{\subset}{\subseteq}

\newcommand{\I}{\mathcal{I}}

\newcommand{\uh}{u_{\mathsf{h}}}
\newcommand{\ph}{p_{\mathsf{h}}}
\newcommand{\F}{\mathcal{F}}

\newcommand{\N}{\mathbb{N}}

\newcommand{\1}{\mathbf{1}}
\newcommand{\R}{\mathbb{R}}

\newcommand{\Z}{\mathbb{Z}}

\renewcommand{\P}{\mathbb{P}}

\newcommand{\eps}{\varepsilon}

\newcommand{\red}{\color{red}}

\newcommand{\aeps}{\eps^\frac{d}{d-2}}

\newcommand{\wto}{\rightharpoonup}
\newcommand{\Rd}{\R^d}

\newcommand{\rr}{\mathcal{R}}
\newcommand{\km}{k_{max}}

\DeclareMathOperator{\dist}{dist}
\DeclareMathOperator{\dv}{div}

\DeclareMathOperator{\supp}{supp}

\numberwithin{equation}{section}

\mathtoolsset{showonlyrefs}

\title[]{Homogenization for the Stokes equations in randomly perforated domains under almost minimal assumptions on the size of the holes}
\author{Arianna Giunti, Richard M. H\"ofer}

\begin{document}

\begin{abstract}

We prove the homogenization to the Brinkman equations for the incompressible Stokes equations in a bounded domain which is perforated by a random
collection of small spherical holes. The fluid satisfies a no-slip boundary condition at the holes. The balls generating the holes have centres distributed according to a
Poisson point process and i.i.d. unbounded radii satisfying a suitable moment condition. We stress that our assumption on the distribution of the radii does not
exclude that, with overwhelming probability, the holes contain clusters made by many overlapping balls. We show that the formation of these clusters has no effect on
the limit Brinkman equations. Due to the incompressiblility condition and the lack of a maximum principle for the Stokes equations, our proof requires a very careful study of the geometry of the random holes generated by the class of probability measures considered.
\end{abstract}
\maketitle

\tableofcontents

\section{Introduction}
\label{sec:intro}
In this paper we consider the steady incompressible Stokes equations
\begin{align}
	\label{Stokes}
\begin{cases}
-\Delta u_\eps + \nabla p_\eps = f \ \ \ &\text{in $D^\eps$}\\
\nabla \cdot u_\eps = 0 \ \ \ &\text{in $D^\eps$}\\
u_\eps = 0 \ \ \ &\text{on $\partial D^\eps$}
\end{cases}
\end{align}
in a domain $D^\eps$, that is obtained by removing from a bounded set $D \subset \Rd$, $d > 2$, a random number of small balls having random centres and 
radii. More precisely, for $\eps > 0$, we define
\begin{align}
	\label{perforation}
D^\eps= D \backslash H^\eps, \ \ \ \ \ H^\eps:= \bigcup_{z_i \in \Phi \cap \frac{1}{\eps}D} B_{\aeps \rho_i} (\eps z_i),
\end{align}
where $\Phi$ is a Poisson point process on $\Rd$ with homogeneous intensity rate $\lambda > 0$, and the radii $\{\rho_i \}_{z_i \in \Phi} \subset \R_+$ are identically and
independently distributed unbounded random variables.  We comment on the exact assumptions on the distribution of each $\rho_i$ later in this introduction. 
Our main result states that, for almost every realization of $H^\eps$ in \eqref{perforation}, the solution  $u_\eps$ to \eqref{Stokes} weakly converges in $H^1_0(D)$ to the
solution $\uh$ of the Brinkman equations
\begin{align}
	\label{Brinkman}
\begin{cases}
-\Delta \uh + \mu \uh + \nabla p_h  = f \ \ \ &\text{in $D$}\\
\nabla \cdot \uh = 0 \ \ \ &\text{in $D$}\\
\uh= 0 \ \ \ &\text{on $\partial D$}.
\end{cases}
\end{align}
The constant matrix $\mu$ appearing in the equations above satisfies
\begin{align}\label{strange.term.Stokes}
\mu= \mu_0 \mathrm{I}, \ \ \ \ \ \mu_0 = C_d \lambda \langle \rho^{d-2} \rangle,
\end{align}
where $\langle \cdot \rangle$ denotes the expectation under the probability measure on the radii $\rho_i$, and the constant $C_d > 0$ depends only on the dimension $d$. 
In the case $d=3$, we have $C_d =6\pi$.

\smallskip

From a physical point of view, the equations in \eqref{Stokes} represent the motion of an incompressible viscous fluid among many small obstacles; the additional term 
$\mu \uh$ appearing in \eqref{Brinkman} corresponds to the effective friction force of the obstacles acting on the fluid.  In the physical literature, the term $\mu$ is usually referred to as the ``Stokes resistance''; in this paper, we mostly adopt for $\mu$ the term ``Stokes capacity density'' to emphasize
the analogy with the harmonic capacity density which appears in the analogue homogenization problem for the Poisson equation \cite{ CioranescuMurat, GHV1}. More precisely, for a smooth and bounded set $E \subset \Rd$, let us define its \textit{Stokes capacity} as the symmetric and positive-definite matrix given by
\begin{align}\label{Stokes.capacity}
 \xi^t \cdot M \xi = \inf_{w \in E_\xi} \int_{\Rd \backslash E} |\nabla w |^2, \ \ \  \text{for all $\xi \in \Rd$.}
\end{align}
Here,
$$
E_\xi= \bigl\{ w \in H^1_{loc}(\Rd; \Rd) \, \colon \, \nabla \cdot w =0, \ w=\xi \ \text{in $E$}, \ \ w \to 0 \ \text{for $|x| \uparrow +\infty$} \bigr\}.
$$
Then, in the case $E= B_r$, we obtain $M = C_d r^{d-2} \mathrm{I}$ (see e.g. \cite{AllaireARMA1990a}). The definition \eqref{strange.term.Stokes} of $\mu$ is thus an 
averaged version of the previous formula where we take into account the intensity rate of the Process $\Phi$ according to which the balls of $H^\eps$ are generated.

\bigskip

This work is an adaptation to the Stokes equations of the homogenization result obtained in \cite{GHV1} for the Poisson equation. In particular, the class of random holes
considered in the current paper is included in the class studied in \cite{GHV1}. In the latter, it is assumed that the identically distributed radii $\rho_i$ in \eqref{perforation} satisfy
\begin{align}\label{GHV.condition}
\langle \rho^{d-2} \rangle < + \infty.
\end{align}
In the current paper, we require the slightly stronger condition
\begin{align}
	\label{power.law}
 \langle \rho^{(d-2)+ \beta} \rangle < +\infty, \ \ \  \text{for some $\beta > 0$.}
\end{align}
Before further commenting on \eqref{power.law} in the next paragraph, we recall that in the case of the Poisson problem, the analogue of the term $\mu$ appearing in the
homogenized equation \eqref{Brinkman} is the asymptotic harmonic capacity density generated by the holes $H^\eps$. Assumption \eqref{GHV.condition} is minimal in order to
have that this quantity is finite in average, but does not exclude that with overwhelming probability some balls generating $H^\eps$ overlap. For further comments on this,
we refer to the introduction in \cite{GHV1}.

\smallskip

The main challenge in proving the results of this paper is related to the regions of $H^\eps$ where there are clustering effects. More precisely, the main goal is to 
estimate their contribution to the Stokes capacity density, and thus to the limit term $\mu$ appearing in \eqref{Brinkman}. In the case of the Poisson equation in \cite{GHV1}, the analogue is done by 
relying on the sub-additivity of the harmonic capacity, together with \eqref{GHV.condition} and a Strong Law of Large Numbers. In the case of the Stokes capacity 
\eqref{strange.term.Stokes}, though, sub-additivity fails due to the incompressibility of the fluid (i.e. the divergence-free condition). We thus need to cook up a different method to deal with the balls in $H^\eps$ which overlap or are too close.  Heuristically speaking, the main challenge is that the incompressibility condition yields that big velocities are needed to squeeze a fixed volume of fluid through a possible narrow opening. The main reason for the strengthened assumption \eqref{power.law} is that it allows us to obtain a certain degree of information on the geometry of the clusters of $H^\eps$. In particular, \eqref{power.law} rules out the occurrence of clusters made of too many holes of \emph{similar size}. We emphasize, however, that it neither prevents the balls generating $H^\eps$ from overlapping, nor it implies a uniform upper bound on the number of balls of very different size which combine into a cluster (see Section \ref{s.probability}). The main technical effort of this paper goes into developing a strategy to deal with these geometric considerations and succeed in controlling the term in \eqref{Brinkman}. We refer to Subsection \ref{sec:ideas} for a more detailed discussion on our strategy.

\smallskip

We also mention that, to avoid further technicalities, we only treat the case where the centres of the balls in \eqref{perforation} are distributed according to a homogeneous Poisson point
process. It is easy to check that our result applies both to the case of periodic centres and to any (short-range)
correlated point process for which the results contained in Appendix \ref{sec:SLLN} hold. 

\bigskip

After Brinkman proposed the equations \eqref{Brinkman} in \cite{Brinkman1947} for the fluid flow in porous media, an extensive literature has been developed to obtain a
rigorous derivation of  \eqref{Brinkman} from \eqref{Stokes} in the case of periodic configuration of holes \cite{Brillard1986-1987, LEVY198311, SanchezP82,  MarKhr64}. 
We take inspiration in particular from \cite{AllaireARMA1990a}, where the method used in \cite{CioranescuMurat} for the Poisson equations is adapted to treat the case of the
Stokes equations in domains with periodic holes of arbitrary and identical shape. In \cite{AllaireARMA1990a}, by a compactness argument, the same techniques used for
the Stokes equations also provide the analogous result in the case of the stationary Navier-Stokes equations. The same is true also in our setting (see Remark \ref{N-S}
in Section \ref{setting}).

\smallskip

In \cite{Desvillettes2008}, with methods similar to \cite{AllaireARMA1990a} and \cite{CioranescuMurat}, the homogenization of stationary Stokes and Navier-Stokes equations 
has been extended also to the case of spherical holes where different and constant Dirichlet boundary conditions are prescribed  at the boundary of each ball.
This corresponds to the \emph{quasi-static regime} of holes slowly moving in a fluid, and gives rise in \eqref{Brinkman} to an additional source term 
$\mu j$, with $j$ being the limit flux of the holes. In \cite{Desvillettes2008}, the holes have all the same radius, are not necessarily 
periodic, but satisfy a uniform minimal distance condition of the same order of $\eps$ as in the periodic setting. In \cite{Hillairet2018}, this last condition has been weakened but
not completely removed. In particular it is still assumed that, asymptotically for $\eps \downarrow 0$, the radius of each hole is much smaller than its distance to any other hole. 

\smallskip

In \cite{Hillairet2017}, the quasi-static Stokes equations are considered in perforated domains with holes of different shapes which are both translating and rotating.
Due to the shapes of the holes, the problem becomes non-isotropic, i.e. the matrix $\mu$ in \eqref{Brinkman} is not a multiple of the 
identity. Moreover, since also the rotations of the holes are included into the model, a more complicated source term $\bar{\mathbb{F}}$ arises on the right hand side
of the limit problem. The result in \cite{Hillairet2017} is proved under the same uniform minimal distance assumption as in \cite{Desvillettes2008}.

\smallskip

Finally, we also mention that the homogenization in the Brinkman regime for evolutionary Navier-Stokes in a bounded domain of $\R^3$ has been considered in
\cite{Feireisl2016}. In this paper, the holes are assumed to be disjoint, have arbitrary shape and uniformly bounded diameter. A condition on the minimal distance between 
the holes is substituted by a weaker assumption implying that, for $\eps$ small enough, the diameter of the holes is much smaller than the distance between them.

\medskip

There are fewer results in the literature concerning the case of randomly distributed holes: In \cite{Rubinstein1986}, the case of $N$ randomly distributed spherical
holes of size $N^{-1}$ in $\R^3$ is considered. Starting from the Brinkman equation \eqref{Brinkman} with the term $\mu$ sufficiently large, it is shown that in the
limit $N \to \infty$ an additional zero-order term appears in the limit equation. This result has been recently generalized in \cite{Carrapatoso2018} to the case of the Stokes equations in the
quasi-static regime.

\medskip

The derivation of the Brinkman equations can be viewed as a very first step in deriving the so-called Vlasov(-Navier)-Stokes equations, a model for the
coupled dynamics of particles suspended in a fluid. A rigorous derivation of these equations for the full problem is completely open.
Homogenization results for such dynamic problems have only been achieved in the case when the inertia of the particles is neglected.
In that case, an external constant gravitation field is considered, and the friction caused by the particles is only related to gravity.
For inertialess particles, \cite{Jabin2004} identified the regime that is so dilute that  particles effectively do not interact.
In \cite{Hofer18}, the homogenization result for the inertialess problem has been obtained under a uniform minimal distance assumption.
A related result has been obtained in \cite{Mecherbet2018} where convergence to the same limit equation is proven also when rotations of the particles are taken into account. The assumptions on the initial particle distributions in \cite{Mecherbet2018} do not contain the uniform minimal distance assumption from \cite{Hofer18}, but they are similar to those in \cite{Hillairet2018}. However, the convergence is only proved for small times and for
initial particle distributions that are sufficiently dilute.

\medskip

We emphasize that the main novelty of our paper is that we consider spherical holes whose radii are not uniformly bounded and only satisfy
\eqref{power.law}. As already mentioned above, for small $\beta$ in \eqref{power.law}, with probability tending to one as $\eps \to 0$, 
the perforated domain $D^\eps$ in \eqref{perforation} contains many holes that overlap. In all the deterministic results listed above, overlapping balls are either excluded 
or asymptotically ruled out for $\eps \downarrow 0$. Similarly, in the random settings of \cite{Rubinstein1986} and
\cite{Carrapatoso2018}, the overlapping are negligible in probability: Since the radii of the holes are chosen to be identically $N^{-1}$, it is shown that, with probability tending to one as
$N \to \infty$, the minimal distance between them is bounded below by $N^{-\alpha}$ for $\alpha < 1$ . 

\smallskip

{
We finally mention that our main result does not provide any convergence result for the pressures $\{ p_\eps\}_{\eps > 0}$. However, it is possible to upgrade our
techniques to obtain a partial control on the pressure terms. We address this issue in the paper in preparation \cite{GH2}, and refer to Remark \ref{r.pressure} of the 
current paper for a detailed statement.}

\bigskip

This paper is organized as follows: In Section \ref{setting} we state the main theorem on the convergence of the fluid velocity $u_\eps$. In Subsection \ref{proof.main} we 
formulate Lemma \ref{reduction.operator} which provides a rich class of test-functions for \eqref{Stokes} 
and characterizes their behaviour in the limit $\eps \to 0$. We then show how the convergence of  $u_\eps$ follows from this result. In Section \ref{geometry}, we give some
geometric properties for the realization of the holes $H^\eps$ that are needed in order to prove Lemma \ref{reduction.operator}. These properties are split into two lemmas.
The first one is analogous to the corresponding lemma in 
\cite{GHV1}, the other one gives more detailed informations on the geometry of the clusters of $H^\eps$ and is the result which requires the strengthened version
\eqref{power.law} of \eqref{GHV.condition}. In subsection \ref{s.geometry.proof}, we prove the results stated in Section \ref{geometry}. In Section \ref{sec:testfunction}, 
we prove Lemma \ref{reduction.operator}. In Section \ref{s.probability}, 
we prove some probabilistic result on the number of comparable balls which may combine into a cluster of $H^\eps$. These are the key ingredients used in subsection
\ref{s.geometry.proof} to show the geometric results of Section \ref{geometry}. Finally, the appendix is divided into three parts: In Appendix \ref{sec:Proof.remark}, we show how to extend the convergence result from the Stokes equations to the Stationary Navier-Stokes equations.  In Appendix \ref{sec:estimates.stokes}, we give some standard estimates for the solutions of the Stokes equations in annuli and exterior domains.
In Appendix \ref{sec:SLLN}, we recall some results concerning the Strong Law of Large Numbers, which have been proved in detail in \cite{GHV1} and which are used also 
throughout this paper.

\section{Setting and main result}\label{setting}

Let $D \subset \R^d$, $d >2$, be an open and bounded set that is star-shaped with respect to the origin. For $\varepsilon > 0$, we denote by $D^\varepsilon \subset D$ the 
domain obtained as in \eqref{perforation}, namely by setting $D^\eps = D \backslash H^\eps$ with 
\begin{equation}\label{holes}
H^\varepsilon := \bigcup_{z_j \in \Phi \cap \frac 1 \eps D} B_{\varepsilon^{\frac{d}{d-2}} \rho_j} (\varepsilon z_j).
\end{equation}
Here, $\Phi \subset \Rd$ is a homogeneous Poisson point process having intensity $\lambda > 0$ and the radii
$\rr :=\{ \rho_i \}_{z_i \in \Phi}$ are i.i.d. random variables which satisfy condition \eqref{power.law} for a fixed $\beta >0$. Since assumption \eqref{power.law} with 
$\beta_1> 0$ implies \eqref{power.law} for every other $0 < \beta \leq \beta_1$, with no loss of generality we assume that $\beta \leq 1$.

\smallskip

Throughout the paper we denote by $(\Omega, \F , \P)$ the probability space associated to the marked point process $(\Phi, \rr)$, i.e. the joint process of the centres and radii
distributed as above. We refer to \cite{GHV1} for a detailed introduction of marked point processes as the one introduced in this paper. 

\medskip 

\subsection{Notation} \label{sec:notation}
 For a point process $\Phi$ on $\R^d$ and any bounded set $E \subset \R^d$, we define the random variables
\begin{equation} 	\label{Number.function}
\begin{aligned}
	\Phi(E)&:= \Phi \cap E, && \Phi^\eps(E):= \Phi \cap \left(\frac 1 \eps E \right), \\
  N(E) &:= \# (\Phi (E)), && N^{\eps}(E) := \# (\Phi^\eps(E)).
\end{aligned}
\end{equation}
For $\eta > 0$, we denote by  $\Phi_\eta$ a thinning for the process $\Phi$ obtained as 
\begin{align}\label{thinned.process}
\Phi_\eta(\omega):= \{ x \in \Phi(\omega) \,  \colon \, \min_{y \in \Phi(\omega), \atop y \neq x} | x- y| \geq \eta \},
\end{align}
i.e. the points of $\Phi(\omega)$ whose minimal distance from the other points is at least $\eta$. Given the process $\Phi_\eta$, we set $\Phi_\eta(E)$, 
$ \Phi_\eta^\eps(E)$, $N_\eta(E)$ and $N_\eta^\eps(E)$ for the analogues for  $\Phi_\eta$ of the random variables defined in \eqref{Number.function}.

\smallskip

 For a bounded and measurable set $E \subset \Rd$ and any $1\leq p < + \infty$, we denote 
 \begin{align}
 L^p_0(E):= \{ f \in L^p(E)  \ \colon \ \int_E f = 0 \}.
 \end{align}
As in \cite{GHV1}, we identify $v \in H^1_0(D^\eps)$ with the function $\bar v \in H^1_0(D)$ obtained by extending $v$ to zero in $H^\eps$.

\smallskip

Throughout the proofs in this paper, we write $a \lesssim b$ whenever $a \leq C b$ for a constant $C=C(d,\beta)$ depending only on the dimension $d$ and $\beta$ from assumption
\eqref{power.law}. Moreover, when no ambiguity occurs, we use a scalar notation also for vector fields and 
vector-valued function spaces, i.e. we write for instance $C^\infty_0(D), H^1(\Rd), L^p(\Rd)$ instead of $C^\infty_0(D; \Rd), H^1(\Rd ;\Rd), L^p(\Rd;\Rd)$.
 
\subsection{Main result} Let $(\Phi, \rr)$ be a marked point process as above, and let $H^\varepsilon$ be defined as in \eqref{holes}.

\begin{thm}\label{t.main}
 For $f \in H^{-1}(D ; \Rd)$ and $\varepsilon > 0$, let
$(u_\varepsilon, p_\eps) =( u_\varepsilon (\omega, \cdot), p_\eps(\omega, \cdot)) \in H^1_0(D^\varepsilon; \R^d)\times L^2_0(D^\eps; \R)$  be the solution of
\begin{equation}\label{P.epsilon}
\begin{cases}
-\Delta u_\varepsilon + \nabla p_\eps = f  \ \ \ \ &\text{in $D^\varepsilon$}\\
\, \nabla \cdot u_\eps = 0\ \ \ \ &\text{in $D^\varepsilon$}\\
\ u_\eps = 0 \ \ \ \ &\text{on $\partial D^\eps$.}
\end{cases}
\end{equation}

Then, for $\P$-almost every $\omega \in \Omega$ and for $\varepsilon \downarrow 0^+$
\begin{align}\label{convergence}
u_\varepsilon(\omega, \cdot) \rightharpoonup \uh,
\end{align} 
where $(\uh, \ph) \in H^1_0(D; \Rd) \times L^2_0(D; \R)$ is the solution of 
\begin{equation}\label{P.hom}
\begin{cases}
-\Delta \uh + \nabla \ph + {C_d} \lambda \langle \rho^{d-2}\rangle \uh = f  \ \ \ \ &\text{in $D$}\\
\, \nabla \cdot \uh = 0\ \ \ \ &\text{in $D$}\\
\ \uh = 0 \ \ \ \ &\text{on $\partial D$.}
\end{cases}
\end{equation}
Here, the constant $C_d$ is as in \eqref{strange.term.Stokes}.
\end{thm}

\begin{rem}[Stationary Navier-Stokes equations]\label{N-S}
As in the case of periodic holes \cite{AllaireARMA1990a}, we remark that the same result of Theorem \ref{t.main} holds in dimension $d=3,4$ for the solutions $u_\eps$ 
to the stationary Navier-Stokes system
\begin{equation}\label{NS.epsilon}
\begin{cases}
u_\eps \cdot \nabla u_\eps - \Delta u_\varepsilon + \nabla p_\eps = f  \ \ \ \ &\text{in $D^\varepsilon$}\\
\, \nabla \cdot u_\eps = 0\ \ \ \ &\text{in $D^\varepsilon$,}\\
u_\eps = 0 \ \ \ \ &\text{on $\partial D^\eps$}
\end{cases}
\end{equation}
with homogenized equations
\begin{equation}\label{NS.hom}
\begin{cases}
u_h \cdot \nabla u_h - \Delta u_h + C_d \lambda \langle \rho^{d-2}\rangle u_h + \nabla p_h = f  \ \ \ \ &\text{in $D$}\\
\, \nabla \cdot u_h = 0\ \ \ \ &\text{in $D$}\\
u_h= 0 \ \ \ \ &\text{on $\partial D$,}
\end{cases}
\end{equation}

We argue in the appendix how the same argument that we give in the next section for Theorem \ref{t.main} allows also to treat the non-linear term in \eqref{NS.epsilon}.
\end{rem}

\begin{rem}[Convergence of the pressure terms]\label{r.pressure}
In most of the literature concerned with the homogenization of equations \eqref{P.epsilon} the convergence of the pressure is not considered. The only exception 
is \cite{AllaireARMA1990a} where it is shown that for a suitable extension $P_\eps (p_\eps)$ for $p_\eps$ on the whole domain $D$, the functions $P_\eps (p_\eps)$ converge
to $p_h$ weakly in $L^2(D)$.  The main difficulty in our case is again given by the presence of the clusters of $H^\eps$ that prevents us from finding suitable bounds for $p_\eps$ close to those regions.
Nonetheless, we anticipate here the following result, which we prove {in \cite{GH2}}. It states that $p_\eps$ converges to $p_h$, as long as
we remove from $D$ an exceptional set $E^\eps$ containing $H^\eps$. This set almost coincides with $H^\eps$ in the sense that the difference $E^\eps \backslash H^\eps$
has harmonic capacity $\operatorname{Cap}(E^\eps \backslash H^\eps)$ vanishing in the limit $\eps \downarrow 0^+$.
\end{rem}
\begin{thm}[\cite{GH2}]
	\label{t.pressure}
	For almost every $\omega \in \Omega$, there exists a set $E^\eps \subset \Rd$ such that $E^\eps \supset H^\eps$ and for $\eps \downarrow 0^+$
	\begin{align}
		\label{Cap.E^eps}
	 \operatorname{Cap}(E^\eps \backslash H^\eps) \to 0.
	\end{align}
	 Moreover, for  every compact set $K \Subset D$, the modification of the pressure 
	\begin{align}\label{p.tilde}
	 \tilde p_\eps = \begin{cases}
	  p_\eps - \fint_{K \backslash E^\eps} p_\eps \ \ \ \ &\text{ in $K \backslash E^\eps$}\\
	  0 \ \ \ \ &\text{ in $D\backslash K \cup E^\eps$}
	  \end{cases}
	\end{align}
satisfies $\tilde p_\eps \rightharpoonup p_h$ in $L^{q}_0(K ; \R)$, for all $q < \frac{d}{d-1}$.
\end{thm}

\medskip

{
\subsection{Main ideas in proving Theorem \ref{t.main}}
\label{sec:ideas}
As already mentioned above, the structure and many arguments of this paper are an adaptation of \cite{GHV1} to the case of the Stokes equations.
In this subsection, we point out the main differences and the challenges that we encountered along the process.

\smallskip

In contrast with \cite{GHV1}, we prove the convergence of the fluid velocities $u_\eps$ by using an implicit version of the method of oscillating test-functions, which is 
similar to the one of \cite{Desvillettes2008}: We construct an operator $R_\eps$ which acts on divergence-free test-functions
$v$ such that $R_\eps v \in H^1_0(D^\eps)$ is an admissible test function for \eqref{P.epsilon}, $R_\eps v  \to v$ in $H^1_0(D)$ and  $\nabla\cdot R_\eps v = 0$ in $D$. 
This last condition in particular implies that we may test the equation \eqref{P.epsilon} with $R_\eps v$ and do not need any bounds on the pressure $p_\eps$. 
 
\smallskip
 
As in \cite{GHV1} with the construction of the oscillating test-functions $w_\eps$, the construction of the operator $R_\eps$ relies on
a lemma dealing with the geometric properties of the set of holes $H^\eps$ which perforate $D$ in \eqref{perforation}. This lemma allows us to split the set $H^\eps$ into a
``good'' set $H^\eps_g$, which contains holes which are small and well-separated, and a ``bad'' set $H^\eps_b$, which contains big and overlapping holes. On the one hand, we
construct $R_\eps v$ such that it vanishes on $H^\eps_g$ by closely following the ideas in \cite{AllaireARMA1990a} and \cite{Desvillettes2008}. On the other hand, to define
 $R_\eps v$ in such a way that it vanishes also on $H^\eps_b$, we need to improve the arguments used in \cite{GHV1}. In fact, as pointed out in the introduction, 
 in contrast with \cite{GHV1}, by the 
 incompressibility condition it is not enough to prove that the harmonic capacity of $H^\eps_b$ vanishes in the limit $\eps \downarrow 0^+$.

\smallskip

In order to overcome this problem, we use the following strategy to construct $R_\eps v$ such that, for any divergence-free $v \in C^\infty_0(D, \Rd)$, the function
$R_\eps v$ vanishes on the ``bad'' set $H^\eps_b$, remains divergence-free in $D$ and converges to $v$ in $H^1_0(D; \Rd)$. We recall that in the set $H^\eps_b$ the balls may
overlap; the challenge is therefore to find a suitable truncation for $v$ on this set, which preserves the divergence-free condition and which remains bounded
in an $H^1$-sense.  A first approach to construct $R_\eps v$ would then be to solve the Stokes problem in a large enough neighbourhood $D^\eps_b$ of $H^\eps_b$
 \begin{align}\label{stokes.1}
 \begin{cases}
 -\Delta w_\eps +  \nabla \pi_\eps  = \Delta v \ \ \ &\text{in $D^\eps_b \setminus \overline H^\eps_b$}\\
 \nabla \cdot w = 0 \ \ \ &\text{ in $D^\eps \setminus \overline H^\eps_b$}\\
 w = 0 \ \ \ &\text{on $\partial H^\eps_b$}\\
 w(x) = v \ \ \ &\text{on $\partial D^\eps_b$.}
 \end{cases}
 \end{align}
The connection with the concept of ``Stokes capacity'' generated by the set $H^\eps_b$ thus becomes apparent; namely, at least in the case of sets $E$ regular enough, the minimizer in \eqref{Stokes.capacity} solves
 \begin{align}
 \begin{cases}
 -\Delta w +  \nabla \pi  = 0 \ \ \ &\text{in $\Rd \setminus \overline E$}\\
 \nabla \cdot w = 0 \ \ \ &\text{ in $\Rd\setminus \overline E$}\\
 w = \xi \ \ \ &\text{on $\partial E$}\\
 w(x) \to 0 \ \ \ &\text{as $|x| \to \infty$.}
 \end{cases}
 \end{align}
However, getting  $H^1$-estimates on the solution $w^\eps$ of \eqref{stokes.1} which depend explicitly on $\eps$, requires more informations than we have on the geometry of 
the set $H^\eps_b$. In fact, condition \eqref{power.law} does not prevent the balls from overlapping nor provides an upper bound on the number of balls in each of the
clusters (cf. Lemma \ref{l.chains}). The approach that we adopt to construct $R_\eps v$ is therefore different and is based on finding a suitable covering $\bar H^\eps_b$ 
of the set  $H^\eps_b$. The set $\bar H^\eps_b$ is obtained by selecting some of the balls that constitute $H^\eps_b$ and dilating them by a uniformly bounded factor 
$\lambda_\eps \leq \Lambda$. The main, crucial, feature of this covering is that 
it allows us to construct $R_\eps v$ vanishing on $H^\eps_b \subset \bar H^\eps_b$ by solving different Stokes problems in disjoint annuli of the form
$B_{\theta \lambda_\eps\aeps \rho_i}(\eps z_i) \backslash B_{\lambda_\eps \aeps \rho_i}(\eps z_i)$, $\theta > 1$, and iterating
this procedure a finite number of steps.  The advantage in this is that we construct $R_\eps v$ iteratively and obtain bounds by applying a finite number of times some
standard and rescaled estimates for solutions to Stokes equations in the annulus $B_\theta \backslash B_1$.

\smallskip

More precisely, $\bar H^\eps_b$ is chosen to satisfy the following properties:
\begin{itemize}
 \item[$(a)$] $\bar H^\eps_b$ is the union of $M < +\infty$ families of balls such that, inside the same family, the balls $B_{\lambda_\eps \aeps\rho_i}(\eps z_i)$ are disjoint even if dilated by a further
 factor $\theta^2 >0$, i.e. by considering $B_{\theta^2\lambda_\eps \aeps\rho_i}(\eps z_i)$;
\end{itemize}
By this property, if we want to construct $R_\eps v$ vanishing only in the holes of the same family, it suffices to solve \eqref{stokes.1} in the disjoint annuli 
$B_{\theta \lambda_\eps \aeps\rho_i}(\eps z_i) \backslash B_{\lambda_\eps \aeps\rho_i}(\eps z_i)$ and stitch the solutions together. This suffices to construct $R_\eps v$
vanishing on the balls $B_{\lambda_\eps \aeps\rho_i}(\eps z_i)$ of the same family, and thus on the subset of $H^\eps_b$ covered by 
them. In order to obtain $R_\eps v$ vanishing on the whole set $H^\eps_b$, one may try to iterate the previous procedure: Let the families of balls constituting 
$\bar H^\eps_b$ be ordered with an index $k = 1, \cdots, M$. Then:
\begin{itemize}
 \item[$\bullet$] We construct a first solution $v^1_\eps$ which solves \eqref{stokes.1} in all the (disjoint) annuli generated by the first family;
 \item[$\bullet$] We construct $v^2_\eps$ solving \eqref{stokes.1} with $v$ substituted by $v^1_\eps$ in the (disjoint) annuli of the second family;
 \item[$\bullet$] We iterate the procedure up to the $M$-th family and set $R_\eps v = v^M_\eps$. 
\end{itemize}

 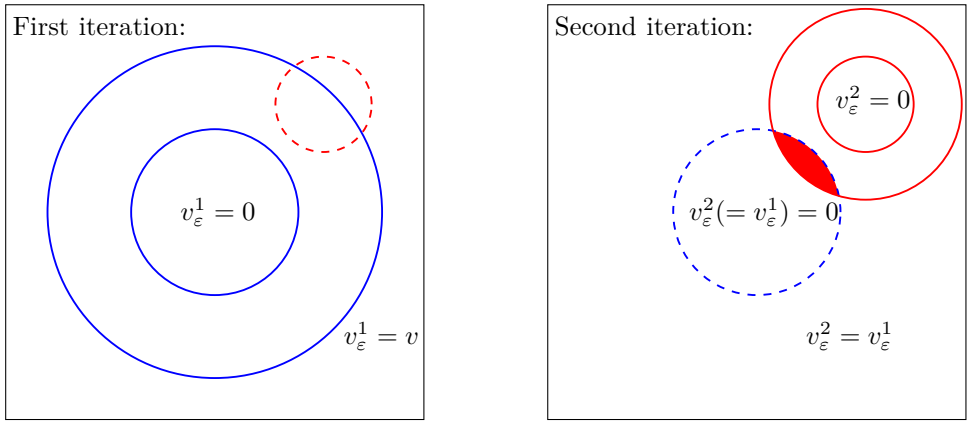
\begin{figure}
 \begin{minipage}[c]{7cm}
 \begin{tikzpicture}[scale=0.55]
\node[text width=3cm] at (-0.1,6.5) {\small  First iteration:};
\draw[blue, line width=0.25mm](2,2) circle (2);
\draw[blue, line width=0.25mm](2,2) circle (4);
\draw[red,dashed, line width=0.25mm](4.6,4.6) circle (1.15);
\node[text width=2cm] at (3,2) {\small $v^1_\eps=0$};
\node[text width=2cm] at (6.9,-1) {\small $v^1_\eps=v$};
\draw (-3,7) rectangle (7,-3);
;\end{tikzpicture}
 \end{minipage}
  \begin{minipage}[c]{7cm}
 \begin{tikzpicture}[scale=0.55]
\node[text width=3cm] at (-0.1,6.5) {\small Second iteration:};
\def\firstcircle{(2,2) circle (2)}
\def\secondcircle{(4.6,4.6) circle (2.3)}
\fill[white] \firstcircle \secondcircle;
    \begin{scope}
        \clip \firstcircle;
        \fill[red] \secondcircle;
    \end{scope}
\draw[blue, dashed, line width=0.25mm](2,2) circle (2);
\draw[red, line width=0.25mm](4.6,4.6) circle (1.15);
\draw[red, line width=0.25mm](4.6,4.6) circle (2.3);
\node[text width=2cm] at (2.2,2) {\small $v^2_\eps(= v^1_\eps)= 0$};
\node[text width=2cm] at (5,-1) {\small $v^2_\eps=v^1_\eps$};
\node[text width=2cm] at (5.7,4.7) {\small $v^2_\eps=0$};
\draw (-3,7) rectangle (7,-3);
;\end{tikzpicture}
 \end{minipage}
 \caption{{\small This is an example of a configuration which satisfies only $(a)$ for which the algorithm to construct $R_\eps v$ may not give a function vanishing on all the holes. The first picture on the left represents the first iteration step: The blue, full-lined, ball is the hole belonging to the first family generating $\bar H^\eps_b$. We solve a Stokes problem in the blue annulus, with zero boundary conditions in the inner ball. The dashed, red ball represents a hole generated by another family of $\bar H^\eps_b$, which is neglected in this step. The second picture represents the second iteration step: Given the solution $v^1_\eps$ obtained in the first step, we solve another Stokes problem in the red, smaller, annulus with zero boundary conditions in the inner hole. Since this new annulus intersects the hole of the previous step, $v^2_\eps$ may not vanish in the intersection in red.}}  \label{BC.destroyed}
 \end{figure}

However, property $(a)$ alone does not ensure that the final solution constructed in this fashion vanishes on $H^\eps_b$: Since annuli generated by different families
may still intersect, at each step the zero-boundary conditions of the previous steps may be destroyed (as an example, see Figure \ref{BC.destroyed}). This is the reason why we need that the covering $\bar H^\eps_b$
satisfies an additional property. This property should ensure that, if at step $k$ the function $v^k$ vanishes on a certain subset of $H^\eps_b$, then also $v^{k+1}$ vanishes on that same
subset. We thus construct $\bar H^\eps_b$ in such a
way that
 \begin{itemize}
 \item[$(b)$] all the balls  $B_{\theta\lambda_\eps \aeps\rho_i}(\eps z_i)$ belonging to the $k$-th family do not intersect the balls of $H^\eps_b$ contained in the 
 previous families (cf. property \eqref{small.dont.intersect.big} of the Lemma \ref{l.geometric.v2}).\footnote{Strictly speaking, this is a simplification of the statement of Lemma \ref{l.geometric.v2} (cf. Remark \ref{real.property} in Section \ref{geometry}).}
 \end{itemize}
The construction of $\bar H^\eps_b$ satisfying $(a)$-$(b)$ is given in Lemma \ref{l.geometric.v2} of Section 4 and constitutes the most technically challenging part of
this paper.
}

\subsection{Lemma \ref{reduction.operator} and proof of Theorem \ref{t.main}}\label{proof.main}
The proof of Theorem \ref{t.main} relies on the following lemma: 

\begin{lem}\label{reduction.operator} For almost every $\omega \in \Omega$ and for all $\eps \leq \eps_0(\omega)$ there exists a linear map
$$
R_\eps \colon \{ v \in C^\infty_0(D) \, \colon \, \nabla \cdot v = 0 \}  \to H^1(D)
$$
with the following properties:
\begin{enumerate}[label=(\roman*)]
 \item $R_\eps v =0$ in $H^\eps$ and,  for $\eps$ small enough, also $R_\eps v \in H^1_0(D)$; \label{pro.0.inside.holes}
 \item $\nabla \cdot R_\eps v =0$ in $\Rd$; \label{pro.divergence.free}
 \item $R_\eps v \rightharpoonup v$ in $H^1_0(D)$; \label{pro.convergence.H^1}
  \item $R_\eps v \to v$ in $L^p(D)$ for all $1 \leq p < \infty$; \label{pro.convergence.L^p}
 \item For all $u_\eps \in H^1_0(D^\eps )$ such that $\nabla \cdot u_\eps =0$ in $D$ and $u_\eps \rightharpoonup u$ in $H^1_0(D)$, we have
 \begin{align}
  \int \nabla R_\eps v : \nabla u_\eps \to \int \nabla v : \nabla u + C_d \lambda \langle \rho^{d-2} \rangle \int v \cdot  u,
 \end{align}
 with $C_d$ as in Theorem \ref{t.main}. \label{pro.capacity}
\end{enumerate}
\end{lem}

\medskip

\begin{proof}[Proof of Theorem \ref{t.main}]
Let us fix $\omega \in \Omega$ such that the operator $R_\eps$ of Lemma \ref{reduction.operator} exists and satisfies all the properties \ref{pro.0.inside.holes} - \ref{pro.capacity}. We trivially  extend $u^\eps$ to the whole set $D$. Since by the standard energy estimate we have $\| u_\eps \|_{H^1_0(D)} \leq \| f \|_{H^{-1}(D)}$, then up to a subsequence $\eps_j$, we have
$u_\eps \rightharpoonup u^*$ in $H^1_0(D)$. Note that also $\nabla \cdot u^*=0$ in $D$. We show that $u^*$ solves \eqref{P.hom} and, by uniqueness, that $u^*= \uh$ in $H^1_0(D)$.
We thus may extend the convergences above to the whole limit $\eps \downarrow 0^+$. 

\smallskip

For any divergence-free $v \in C^\infty_0(D )$, we consider $\eps$ small enough such that the divergence-free vector field $R_\eps v$ obtained by means of Lemma \ref{reduction.operator} is in  $H^1_0(D)$. By testing \eqref{P.epsilon} with this vector field, we obtain
\begin{align}
\int \nabla R_\eps v :\nabla u_\eps = \langle R_\eps v , f\rangle_{H^1,H^{-1}}.
\end{align}
We now apply \ref{pro.convergence.H^1} and \ref{pro.capacity} of Lemma \ref{reduction.operator} to the left- and right-hand side of the above identity, respectively, and conclude that  $u^*$ satisfies  
\begin{align}
\int \nabla v : \nabla u^* + C_d \lambda \langle \rho^{d-2} \rangle \int v  \cdot u^* = \langle v , f\rangle_{H^1,H^{-1}}.
\end{align}
Since $v \in C^\infty_0(D)$ is an arbitrary divergence-free test function, we conclude that $u^*$ is the solution $u_h$ of \eqref{P.hom}.
\end{proof}
\section{Geometric properties of the holes}\label{geometry} 
This section is the core of the argument of Theorem \ref{t.main} and provides some almost sure geometrical properties on $H^\eps$. 
These allow us to construct the operator of Lemma \ref{reduction.operator}. 

The results contained in this section rely on assumption
\eqref{power.law} and may be considered as an upgrade of Section 4 of \cite{GHV1}. Since \eqref{power.law} is stronger than the one assumed in \cite{GHV1} (see 
\eqref{GHV.condition}), the marked point process $(\Phi, \rr)$  considered in this work is included in the class of processes studied in \cite{GHV1}. Therefore, all the 
results for $H^\eps$ contained in Section 4 of \cite{GHV1} hold also in our case. Bearing this in mind, we introduce the first main result of this section: This is almost 
a rephrasing of Lemma 4.2 of \cite{GHV1}, where, thanks to \eqref{power.law}, we are allowed to choose the sequence $r_\eps$ appearing in the statement of Lemma 4.2 in
\cite{GHV1} as a power law $r_\eps = \eps^\delta$, for $\delta = \delta(d,\beta) > 0$.

\begin{lem}\label{l.geometry}
There exists a $\delta = \delta(d,\beta) > 0$ such that for almost every $\omega \in \Omega$ and all $\eps \leq \eps_0=\eps_0(\omega)$, 
there exists a partition $H^\eps= H^\eps_g \cup H^\eps_b$ and a set $D^\eps_b \subset \Rd$ such that
$H^\eps_b \subset D^\eps_b$ and
\begin{align}
	\label{safety.layer}
\dist( H^\eps_g ; D^\eps_b) > \eps^{1 + \delta},  \ \ \ |D^\eps_b| \downarrow 0^+.
\end{align}
Furthermore, $H_g^\eps$ is a union of disjoint balls centred in $n^\eps \subset \Phi^\eps(D)$, namely

\begin{equation}\label{good.set}
\begin{aligned}
 H^\eps_g = \bigcup_{z_i \in n^\eps} B_{\aeps \rho_i}(\eps z_i), \ \ \ \ \eps^d \#n^\eps \to \lambda \, |D|, \\
\min_{z_i \neq z_j \in n^\eps} \eps |z_i - z_j| \geq 2 \eps^{1+\frac \delta 2}, \quad \aeps \rho_i \leq  \eps^{1+2\delta}.
\end{aligned}
\end{equation}

Finally, if for $\eta > 0$ the process $\Phi^\eps_{2\eta}$ is defined as in \eqref{thinned.process}, then
\begin{align}\label{small.distance.bad}
	\lim_{\eps \downarrow 0}\eps^d \#(\{z_i \in \Phi^{\eps}_{2\eta}(D) \colon \dist(\eps z_i, D_b^\eps) \leq \eta \eps \}) = 0.
\end{align}
\end{lem}

\medskip

The next result upgrades the previous lemma and is the key result on which relies the
construction of the operator $R_\eps$ of Lemma \ref{reduction.operator}. We introduce the following notation: We set $\mathcal I^\eps := \Phi^\eps(D) \backslash n^\eps$, so 
that, by the previous lemma, we may write
\begin{align}
	\label{H^b}
H^\eps_b:= \bigcup_{z_i \in \mathcal \I^\eps} B_{\aeps \rho_i}(\eps z_i).
\end{align}
As already discussed in Subsection 2.1, the main aim of the next result is to show that there exists a suitable covering for $H^\eps_b$, which is of the form
\begin{align}
\bar H^\eps_b := \bigcup_{z_j \in J^\eps} B_{\lambda_j^\eps \aeps \rho_j}( \eps z_j), \ \ \ J^\eps \subset \I^\eps, \ \sup_{z_j \in J^\eps} \lambda_j^\eps \leq \Lambda
\end{align}
and which satisfies (a) and (b) of Subsection 2.1.
More precisely, we have:
\begin{lem} \label{l.geometric.v2}
 Let $\theta > 1$ be fixed. Then for almost every $\omega \in \Omega$ and $\eps \leq \eps_0(\omega, \beta, d, \theta)$ we may choose $H^\eps_g, H^\eps_b$ of Lemma \ref{l.geometry}
 in such a way that have the following:
 \begin{itemize}
 \item There exist $ \Lambda(d, \beta)> 0$, a sub-collection $J^\eps \subset \mathcal I^\eps$ and constants $\{ \lambda_l^\eps \}_{z_l\in J^\eps} \subset [1, \Lambda]$ such that
	\begin{align}
		\label{bar.H^b}
	 H_b^\eps \subset \bar H^\eps_b := \bigcup_{z_j \in J^\eps} B_{\lambda_j^\eps \aeps \rho_j}( \eps z_j), \ \ \ \lambda_j^\eps \aeps \rho_j \leq \Lambda \eps^{2d\delta}.
	\end{align}
 \item There exists $k_{max}= k_{max}(\beta, d)>0$ such that we may partition 
 	$$
	\mathcal I^\eps= \bigcup_{k=-3}^{k_{max}} \mathcal I_k^\eps, \ \ \ J^\eps= \bigcup_{i=-3}^{k_{max}} J_k^\eps,
	$$
 with $\I^\eps_k \subset J^\eps_k$ for all $k= 1, \cdots, \km$ and
{ \begin{align}\label{inclusion.step.by.step}
	\bigcup_{z_i \in \mathcal I_k^\eps} B_{\aeps \rho_i}( \eps z_i) \subset \bigcup_{z_j \in J_k^\eps} B_{\lambda_j^\eps \aeps \rho_j}( \eps z_j);
\end{align}}
 \item  For all $k=-3, \cdots, k_{max}$ and every $z_i, z_j \in J_k^\eps$, $z_i \neq z_j$
\begin{align}\label{similar.size.apart}
B_{\theta^2 \lambda_i^\eps \aeps \rho_i}(\eps z_i) \cap B_{\theta^2 \lambda_j^\eps \aeps \rho_j}(\eps z_j) = \emptyset;
\end{align}
\item For each $k=-3, \cdots, k_{max}$ and $z_i \in \mathcal I_k^\eps$ and for all $ z_j \in \bigcup_{l=-3}^{k-1} J_l^\eps$ we have
\begin{align}
\label{small.dont.intersect.big}
B_{\aeps \rho_i}(\eps z_i) \cap B_{\theta \lambda_j^\eps \aeps \rho_j}(\eps z_j) = \emptyset.
\end{align}
 \end{itemize}
 \smallskip
 Finally, the set $D^\eps_b$ of Lemma \ref{l.geometry} may be chosen as
\begin{align}
\label{D_b}
& D^\eps_b = \bigcup_{z_i \in J^\eps} B_{\theta \aeps \lambda_i^\eps\rho_i}(\eps z_i).
\end{align}
\end{lem}

\begin{rem}\label{real.property}
 As explained in Subsection \ref{sec:ideas}, property \eqref{small.dont.intersect.big} is crucial for the construction of the operator $R_\eps$ of Lemma \ref{reduction.operator}.
 However, it slightly differs from property (b) stated in that section. Namely, the balls $B_{\aeps \theta \lambda_j^\eps \rho_j}(\eps z_j)$, $z_j \in J^\eps_l$ might intersect with some of the balls in $H_b^\eps$ that are contained in $B_{\aeps\lambda_i^\eps \rho_i}(\eps z_i)$ for $z_i \in J^\eps_k$, $k > l$.
 This is why the additional index sets $\mathcal I_k^\eps$ are introduced. In these index sets, the balls are not ordered by size, but in such a way that \eqref{small.dont.intersect.big} holds. More precisely, if a ball in $H_b^\eps$ is contained in several of the dilated balls in $J^\eps$,
 we will put it into the index set $\mathcal I_k$ with $k$ minimal such that it is contained in a dilated ball in $J_k^\eps$.
\end{rem}

\subsection{Structure and main ideas in the proof of Lemma \ref{l.geometry} and Lemma \ref{l.geometric.v2}.}
Since the proof of Lemma \ref{l.geometric.v2} requires different steps and technical constructions, we give a sketch of the ideas behind it. It is clear that Lemma \ref{l.geometry}
 follows immediately from Lemma \ref{l.geometric.v2}; we thus only need to focus on the proof of this last result.

\smallskip 

To this end we introduce the following notation, which we will also use
throughout the rigorous proof of Lemma \ref{l.geometric.v2} in Section 5: Let
\begin{align}\label{def.kappa}
\delta :=  \frac{\beta}{2(d-2)(d-2+\beta)} \wedge \frac{\beta}{2d}
\end{align}
and
	\begin{align}
		\label{def.I_k}
		I^\eps_k := 
		\begin{cases}
		\{ z_i \in \Phi^\eps(D) \ \colon \ \eps^{1 - \delta k} \leq \aeps \rho_i < \eps^{1 - \delta (k+1)} \} \ \ \ &k \geq -2\\
		\{ z_i \in \Phi^\eps(D) \ \colon \aeps \rho_i < \eps^{1 +2\delta} \} \ \ \ \ &k= -3.
		\end{cases}
	\end{align}
	Note that $\Phi^\eps(D) = \bigcup_{k \geq -3} I^\eps_k$. We remark that the sets $I^\eps_k$ correspond to $I^\eps_{\delta, k}$ in \eqref{def.I_k_delta} of Section \ref{s.probability} with $\delta$ as in \eqref{def.kappa}.
Since we chose $\delta$ above such that $\delta < \frac{\beta}{2d}$, we may apply Lemma \ref{l.chains} with this choice of $\delta$ and infer that there exists
$\km \in \N$ such that $I^\eps_k = \emptyset$ for all $k > \km$. From now on, we assume that $\km$ is chosen in this way and thus that 
\begin{align}
 \Phi^\eps(D) = \bigcup_{k = -3}^{\km} I^\eps_k.
\end{align}
In addition, since we may bound
\begin{align}
 \aeps \max_{\Phi^\eps(D)} \rho_i \leq \eps^{\frac{d}{d-2} - \frac{d}{d-2+\beta}} \bigl(\eps^d \sum_{z_i\in \Phi^\eps(D)}\rho_i^{d-2 +\beta} \bigr)^{\frac{1}{d-2+\beta}},
\end{align}
we use \eqref{power.law} and the Strong Law of Large Numbers, to infer that almost surely and for $\eps$ small enough
\begin{align}
 \aeps\max_{\Phi^\eps(D)} \rho_i \lesssim \eps^{\frac{d}{d-2} - \frac{d}{d-2+\beta}}\langle \rho^{d-2+\beta} \rangle^{\frac{1}{d-2+\beta}}.
\end{align}
This implies by \eqref{def.kappa} that
\begin{align}
	\label{delta.max}
	\max_{z_i \in \Phi^\eps(D)} \aeps \rho_i \lesssim \eps^{2d\delta}.
\end{align}

\medskip

\noindent {\bf Step 1: Combining clusters of holes of similar size: }  We begin obtaining a first covering of $H^\eps$ made by a union of balls which, if of comparable size, are disjoint even if dilated by a constant factor $\alpha > 1$. Roughly speaking, we do this by merging the balls of $H^\eps$ generated each family $I^\eps_k \cup I^\eps_{k-1}$, in holes of similar size which which are also disjoint. More precisely, we prove:

\smallskip

\noindent {\bf Claim: } Let $\alpha > 1$. Then, there exists $\tilde \Lambda= \tilde \Lambda(d,\beta,\alpha) > 0$ such that for $\P$-almost every 
	$\omega \in \Omega$ and all $\eps < \eps_0(\omega)$ and all $ - 3 \leq k \leq k_{\max}$
	there are $\tilde I_k^\eps \subset I_k^\eps$ and $\{\tilde \lambda_j^\eps\}_{z_j \in \tilde I_{k}} \subset [1, \tilde\Lambda]$
	 with the following properties:
	 \begin{align}
	 	\label{covering.comparable}\\
	 	\forall z_i \in I_k^\eps \ \exists \, z_j \in \bigcup_{l \geq k} \tilde I_l^\eps\ \colon \ B_{ \aeps \rho_i}(\eps z_i) \subset B_{ \aeps\tilde \lambda_j^\eps \rho_j}(\eps z_j) .
	 \end{align}
	  For each $ - 3 \leq k \leq k_{\max}$ the balls
	 \begin{align}
	 	\label{no.intersection.comparable}
	 	\bigg\{ B_{ \aeps \alpha \tilde \lambda_i^\eps\rho_i}(\eps z_i) \bigg\}_{z_i \in \tilde I^\eps_k \cup \tilde I^\eps_{k-1}}\quad \text{are pairwise disjoint}.
	 \end{align}

Note that ``most'' of the balls generated by the points in $I^\eps_{-2} \cup I^\eps_{-3}$ already satisfy \eqref{no.intersection.comparable} with $\lambda^\eps_i =1$. Hence, $\tilde I_{-3}^\eps$ contains most of the points of $I^\eps_{-3}$. The only elements of $I^\eps_{-2} \cup I^\eps_{-3}$ which might violate this conditions are the ones which are too close to each other. We will show that, since the  collection $I^\eps_{-2} \cup I^\eps_{-3}$ is generated by a Poisson point process, these exceptional points are few for small values of $\eps > 0$.

\smallskip

 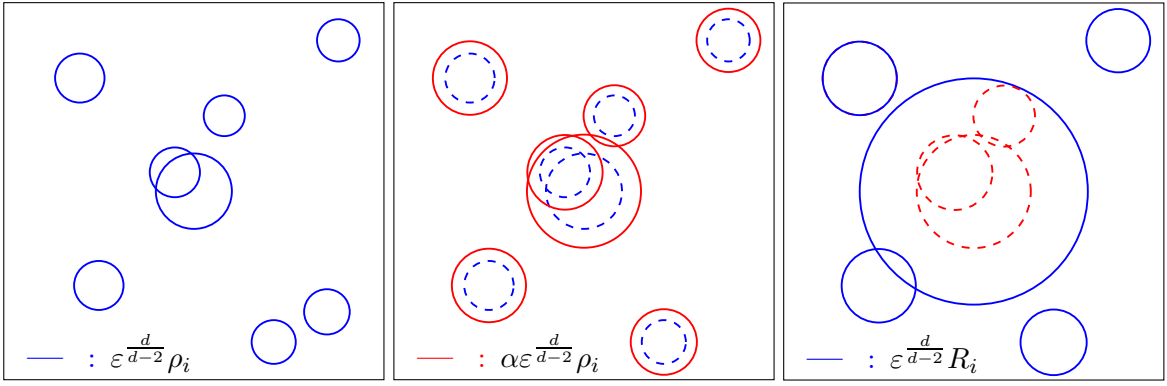
\begin{figure}
 \begin{minipage}[r]{5cm}
 \begin{tikzpicture}[scale=0.5]
\draw[blue, line width=0.25mm](2,2) circle (1);
\draw[blue, line width=0.25mm](1.5,2.5) circle (0.66);
\draw[blue, line width=0.25mm](2.8, 4)circle(0.54);
\draw[blue, line width=0.25mm](-0.5,-0.5) circle (0.65);
\draw[blue, line width=0.25mm](-1,5) circle (0.65);
\draw (-3,7) rectangle (7,-3);

\draw[blue, line width=0.25mm](4.1,-2) circle (0.58);
\draw[blue, line width=0.25mm](5.5, -1.2)circle(0.6);
\draw[blue, line width=0.25mm](5.8, 6)circle(0.56);
\node[text width=2.5cm] at (0.4,-2.3) {{\blue {\bf \large\textemdash} \, :} $\aeps\rho_i$}
;\end{tikzpicture}
 \end{minipage}
  \begin{minipage}[r]{5cm}
 \begin{tikzpicture}[scale=0.5]
\draw[blue, dashed, line width=0.25mm](2,2) circle (1);
\draw[blue, dashed, line width=0.25mm](1.5,2.5) circle (0.66);
\draw[blue, dashed, line width=0.25mm](2.8, 4)circle(0.54);
\draw[blue, dashed, line width=0.25mm](-0.5,-0.5)  circle (0.65);
\draw[blue, dashed, line width=0.25mm](-1,5) circle (0.65);

\draw[blue, dashed, line width=0.25mm](4.1,-2) circle (0.58);

\draw[blue, dashed, line width=0.25mm](5.8, 6)circle(0.56);
\draw[red, line width=0.25mm](2,2) circle (1.5);
\draw[red, line width=0.25mm](1.5,2.5) circle (0.99);
\draw[red, line width=0.25mm](2.8, 4)circle(0.81);
\draw[red, line width=0.25mm](-0.5,-0.5)  circle (0.975);
\draw[red, line width=0.25mm](-1,5) circle (0.975);
\draw[red, line width=0.25mm](4.1,-2) circle (0.87);
\draw[red, line width=0.25mm](5.8, 6)circle(0.84);
\draw (-3,7) rectangle (7,-3);
\node[text width=2.5cm] at (0.4,-2.3) {{\red {\bf \large \textemdash} \, :} $\alpha\aeps \rho_i$}
;\end{tikzpicture}
 \end{minipage}
  \begin{minipage}[r]{5cm}
 \begin{tikzpicture}[scale=0.5]
\draw[blue, line width=0.25mm](2,2) circle (3);
\draw[red,dashed, line width=0.25mm](2,2) circle (1.5);
\draw[red, dashed, line width=0.25mm](1.5,2.5) circle (0.99);
\draw[red,dashed, line width=0.25mm](2.8, 4)circle(0.81);
\draw[red, dashed, line width=0.25mm](-1,5) circle (0.975);
\draw[blue, line width=0.25mm](-0.5,-0.5) circle (0.975);
\draw[blue, line width=0.25mm](-1,5) circle (0.975);
\draw[blue, line width=0.25mm](4.1,-2) circle (0.87);
\draw[blue, line width=0.25mm](5.8, 6)circle(0.84);
\draw (-3,7) rectangle (7,-3);
\node[text width=2.5cm] at (0.4,-2.3) {{\blue {\bf \large \textemdash} \, :} $\aeps R_i$}
;\end{tikzpicture}
 \end{minipage}
 \caption{{\small This sequence of pictures shows how to implement the algorithm of Step 1. From left to right: We begin with an initial configuration of comparable balls generated by centres in $I^\eps_{-3} \cup I^\eps_{-2}$ and with associated radii $\aeps \rho_i$. In the picture in the middle, the full line represents a dilation by a factor $\alpha = 1.5$ of this initial configuration (here drawn with a dashed line). In the last picture, the full line represents the new configuration obtained with the modified radii $R_i$ which covers all the dilated balls of the previous figure (here drawn with a dashed line). }}  \label{strategy.step1}
 \end{figure}

To construct the sets $\tilde I_k$ above we adopt the following strategy (see Figure \ref{strategy.step1} for a sketch): 
 \begin{itemize}
 
 \item Let $\alpha > 1$ and $-2 \leq k \leq \km$ be fixed. We multiply each one of the radii $\{\rho_i \}_{z_i \in I_k^\eps \cup I_{k-1}^\eps}$ by $\alpha$ and consider the set of balls
\begin{align}\label{collection.inflated}
 \bigg\{ B_{\alpha \aeps  \rho_i}(\eps z_i) \bigg\}_{z_i \in I^\eps_k \cup  I^\eps_{k-1}}.
 \end{align}
 For each point $z_i \in I^\eps_k \cup  I^\eps_{k-1}$ we now define a new radius $R_i^\eps$ in the following way:
 For each disjoint ball in the previous collection we set $R_i^\eps:= \rho_i$. We now consider the balls which are 
 not disjoint: For each connected component $C_k^\eps$ of \eqref{collection.inflated}, we pick on of the largest balls belonging to $C_k^\eps$, say $B_{\alpha \aeps \rho_l}(\eps z_l)$,
 and set $R_l^\eps$ as the minimal one such that $C_k^\eps \subset B_{\aeps R_l^\eps}(\eps z_l)$. We set $R_i^\eps=0$ for all the $z_i \neq z_l$ generating the balls contained in $C_k^\eps$. 
We thus have a new collection of radii $\{ R_i^\eps\}_{z_i \in I_k^\eps \cup I_{k-1}^\eps}$.

\medskip

 \item We multiply each $R_i^\eps$ above by the same factor $\alpha$ of the previous step and repeat the construction sketched above with $\rho_i$ substituted by $R_i^\eps$.
 
 \medskip

 \item We show that, almost surely, after a number $M= M(d, \beta)< +\infty$ of iterations of the previous two steps, all the radii $R_i^\eps$ obtained at the $M^{th}$-step do not change
 any further. This means that the balls $B_{\aeps R_i^\eps}(\eps z_i)$, for $R_i^\eps \neq 0$, satisfy \eqref{covering.comparable} and \eqref{no.intersection.comparable}. Moreover, we may easily bound each ratio $\frac{R_i^\eps}{\rho_i} =: \tilde\lambda_i^\eps  \leq \tilde\Lambda$.
 
 \smallskip
 
 The key idea to prove the existence of the threshold $M$ is that the configurations $\omega\in \Omega$ for which the radii $R_i$'s obtained after $M$ iterations continue to
 change is related to events of the form 
$$
\text{ \textit{``There exist $M+1$ balls in $I^\eps_k \cup  I^\eps_{k-1}$ which are connected when dilated by $C(\alpha,M)$''}}.
$$
By Lemma \ref{l.chains}, this event has zero probability for $\eps$ sufficiently small.  

\medskip
 
\item The construction above can be expressed by a dynamical system (cf. \eqref{dyn.system}).

\medskip

\item We iterate this process for $I_{k}^\eps \cap I_{k-1}^\eps$, $-2 \leq k \leq \km$ starting from $k=-2$, each time working with the dilated radii 
that we got from the previous step.
 \end{itemize}
  
  \medskip
 
 {\bf Step 2: Construction of the sets $\I^\eps$ and $J^\eps$ :}
 Let us set $\theta = \alpha^{\frac 1 4} \geq 1$, with $\alpha \geq 1$ as in Step 1 (see \eqref{no.intersection.comparable}).  In the previous step we extracted from each family $I^\eps_k$ generating the whole $\Phi^\eps(D)$ a sub-collection $\tilde I^\eps_k$. These sub-collections provide a covering for the whole set $H^\eps$ and satisfy \eqref{no.intersection.comparable}. The aim of this step is to use the previous result to find a way to extract from $\Phi^\eps(D)$ the subset $\I^\eps$ generating the bad holes and to construct the covering $\bar H^\eps_b$.
 
 \medskip
 
 We remark that, if we set $\lambda_i = \theta^2 \tilde\lambda_i$,  the covering
 \begin{align}
  \bigcup_{k=-3}^{\km}\bigcup_{z_j \in \tilde I^\eps_k} B_{\aeps \tilde\lambda^\eps_j \rho_j}(\eps z_j) \supseteq H^\eps
 \end{align}
 satisfies  \eqref{similar.size.apart} thanks to \eqref{no.intersection.comparable}.
  
 \smallskip
 
 The construction of this step is based on the following simple geometric fact: Let $z_1 \in \tilde I_{k_1}^\eps$ and $z_2 \in \tilde I_{k_2}^\eps$ with $ k_1 < k_2-1$. Since by construction 
 we had $\tilde I_k^\eps \subset \I_{k}^\eps$, this means by definition \eqref{def.I_k} of the sets $I_k^\eps$ that $\aeps \rho_1 \leq \eps^\delta \aeps \rho_2$ and thus that the
 ball $B_{\aeps \rho_1}(\eps z_1)$ is much smaller than $B_{\aeps \rho_2}(\eps z_2)$.
 Therefore, for $\eps \leq \eps_0(d, \beta, \theta)$ we have that 
  \begin{align}\label{basic.fact.1}
 B_{\aeps \theta^3 \tilde\lambda_1^\eps \rho_1 }(\eps z_1) \cap B_{\aeps \tilde \lambda_2^\eps \rho_2} (\eps z_2) \neq \emptyset \
 \Rightarrow \  B_{\aeps \theta \tilde \lambda_1^\eps \rho_1 }(\eps z_1) \subseteq B_{\aeps \theta^2 \tilde \lambda_2^\eps \rho_2} (\eps z_2).
 \end{align}
Indeed,  if the inequality on the left-hand side above is true, for all $z \in B_{\aeps \theta \tilde \lambda_1^\eps \rho_1 }(\eps z_1)$ we have
$$
\eps | z - z_2| \leq \eps |z- z_1 | + \eps |z_1 - z_2 | \leq \aeps \theta \tilde \lambda_1^\eps \rho_1 + \aeps \theta^3 \tilde \lambda_1^\eps \rho_1 + \aeps \tilde \lambda_2^\eps \rho_2.
$$
Since $\aeps \rho_1 \leq \eps^\delta \aeps \rho_2$ and all $1 \leq \tilde \lambda^\eps_i \leq \tilde \Lambda$, we may choose $\eps^\delta < \frac{\theta^2 -1}{\theta\tilde \Lambda
(1+\theta^2)}$ and obtain that
$$
\eps | z - z_2| \leq \aeps \theta^2 \tilde\lambda_2^\eps \rho_2,
$$
i.e. the right-hand side in \eqref{basic.fact.1}. 

\smallskip

By relying on \eqref{basic.fact.1}, we construct the covering $J^\eps$ in the following way:
 \begin{itemize}
 \item We start with $k_{max}$ and set $J_{\km}^\eps= \tilde I_{\km}^\eps$ and $J_{\km -1}^\eps = \tilde I_{\km-1}^\eps$. We know that all the balls of the form 
 $B_{\aeps \tilde\lambda_i^\eps \rho_i}(\eps z_i)$ generated by $z_i \in \tilde I_{\km}^\eps \cup \tilde I_{\km-1}^\eps$ are disjoint in the sense of \eqref{no.intersection.comparable} 
 (recall that $\theta^4 = \alpha$).
 The same holds for the balls $B_{\aeps \tilde\lambda_j^\eps \rho_j}(\eps z_j)$ generated by the centres in  $\tilde I_{\km-2}^\eps \cup \tilde I_{\km-1}^\eps$. 
 We thus focus on the intersections between the balls generated by $\tilde I_{\km-2}^\eps$ and $\tilde I_{\km}^\eps$.
 
 \medskip
 
 \item We show how to obtain the set $J_{\km-2}^\eps$ from $\tilde I_{\km-2}^\eps$ in such a way that \eqref{small.dont.intersect.big} is satisfied by this family. 
 We begin by dilating the balls generated by the centres in $J_{\km}^\eps$ of a factor $\theta^2$ and thus obtain the set
 $$
 E_{\km}^\eps= \bigcup_{z_j \in J_{\km}^\eps}B_{\aeps \lambda_j^\eps \rho_j}(\eps z_j)
 $$ 
 (we recall that $\lambda_j^\eps = \theta^2 \tilde \lambda_j^\eps$). We define 
 $$
 J_{\km -2}^\eps := \{ z_i \in \tilde \I_{\km-2}^\eps \, \colon \, B_{\aeps \theta \tilde \lambda_i^\eps \rho_i}(\eps z_i) \nsubseteq E_{\km}^\eps\}.
 $$
 Note that with this definition, for all $z_j \in  J_{\km-2}^\eps$ and every $z_i \in J_{\km}^\eps$ we have that
 $$
 B_{\aeps \theta \tilde \lambda_i^\eps \rho_i}(\eps z_i) \nsubseteq B_{\aeps \lambda_j^\eps \rho_j}(\eps z_j)
 $$
 and thus by property \eqref{basic.fact.1} (with $z_i = z_1$ and $z_j= z_2$) that 
 \begin{align}
B_{\aeps \theta \lambda_i^\eps \rho_i}(\eps z_i) \cap B_{\aeps \tilde \lambda_j^\eps \rho_j}(\eps z_j) = \emptyset.
 \end{align}
Since $\tilde \lambda^\eps_j \geq 1$, the previous equality implies that the collection $J_{\km-2}^\eps$ satisfies condition \eqref{small.dont.intersect.big}.

\medskip

 \item We now iterate the previous construction: We define 
 $$
 E_{\km-1}^\eps= E_{\km}^\eps \cup  \bigcup_{z_i \in J_{\km-1}^\eps}B_{\aeps \lambda_i^\eps \rho_i}(\eps z_i)
 $$
 and
 $$
 E_{\km -2}^\eps= ( E_{\km-1}^\eps \backslash
 \bigcup_{z_i \in J_{\km-2}^\eps}B_{\aeps \theta \lambda_i^\eps \rho_i}(\eps z_i) ) \cup \bigl(\bigcup_{z_i \in J_{\km-2}^\eps}B_{\aeps \lambda_i^\eps \rho_i}(\eps z_i) \bigr).
 $$
 Note that in the definition of this last set we need to remove the annuli 
 $$
 B_{\aeps \theta \lambda_i^\eps \rho_i}(\eps z_i) \backslash B_{\aeps \lambda_i^\eps \rho_i}(\eps z_i)
 $$
  in order to be able to iterate the argument of the previous step (see Figure \ref{setE} for an illustration of the construction of the set $E_{\km - 2}$).

 \begin{figure}
 \begin{minipage}[r]{5cm}
 \begin{tikzpicture}[scale=0.5]
\node[text width=2cm] at (-0.8,6.5) {\tiny  Set $E_{k_{max} -1}$:};
\draw[blue, line width=0.25mm](0,0) circle (2.5);
\draw[blue, line width=0.25mm](4,3.8) circle (2.3);
\path[draw, blue, pattern=horizontal lines](4,3.8) circle (2.3);
\path[draw, blue, pattern=horizontal lines](0,0) circle (2.5);
\draw[red,dashed, line width=0.25mm](4.6,4.6) circle (0.6);
\draw[red,dashed, line width=0.25mm](4.6,4.6) circle (1.2);
\draw[red,dashed, line width=0.25mm](-1.8,3) circle (0.6);
\draw[red,dashed, line width=0.25mm](-1.8,3) circle (1.2);
\draw[red,dashed, line width=0.25mm](2,2) circle (0.7);
\draw[red,dashed, line width=0.25mm](2,2) circle (1.4);
\draw[red,dashed, line width=0.25mm](5,-1) circle (0.65);
\draw[red,dashed, line width=0.25mm](5,-1) circle (1.3);
\draw (-3,7) rectangle (7,-3);
;\end{tikzpicture}
 \end{minipage}
  \begin{minipage}[r]{5cm}
 \begin{tikzpicture}[scale=0.5]
\node[text width=4cm] at (1.2,6.5) {\tiny  Find the centres in $J_{k_{max} -2}$:};
\draw[blue, dashed,line width=0.25mm](0,0) circle (2.5);
\draw[blue,dashed, line width=0.25mm](4,3.8) circle (2.3);
\draw[red, line width=0.25mm](-1.8,3) circle (0.6);
\draw[red,dashed, line width=0.25mm](-1.8,3) circle (1.2);
\draw[red, line width=0.25mm](2,2) circle (0.7);
\draw[red,dashed, line width=0.25mm](2,2) circle (1.4);
\draw[red, line width=0.25mm](5,-1) circle (0.65);
\draw[red,dashed, line width=0.25mm](5,-1) circle (1.3);
\draw (-3,7) rectangle (7,-3);
;\end{tikzpicture}
 \end{minipage}
  \begin{minipage}[r]{5cm}
 \begin{tikzpicture}[scale=0.5]
\node[text width=2cm] at (-0.8,6.5) {\tiny  Set $E_{k_{max} -2}$:};
\fill[blue](0,0) circle (2.5);
\fill[blue](4,3.8) circle (2.3);
\fill[white](-1.8,3) circle (1.2);
\fill[white](2,2) circle (1.4);
\fill[white](5,-1) circle (1.3);

\fill[blue](-1.8,3) circle (0.6);

\fill[blue](2,2) circle (0.7);

\fill[blue](5,-1) circle (0.65);

\draw (-3,7) rectangle (7,-3);
;\end{tikzpicture}
 \end{minipage}
 \caption{{\small This sequence of pictures shows how to construct $E_{\km-2}$ from $E_{\km-1}$: In the first picture on the left, the set $E_{\km-1}$ is the one filled with horizontal lines. Note that the balls are all disjoint and well-separated. The dashed annuli are the balls generated by centres in $\tilde I_{\km-2}$ and dilated by the factor $\theta$. The circles with the full line in the second picture represent the balls whose centres are in the set $J_{\km-2}$. The third picture shows the set $E_{\km -2}$. }}  \label{setE}
 \end{figure}
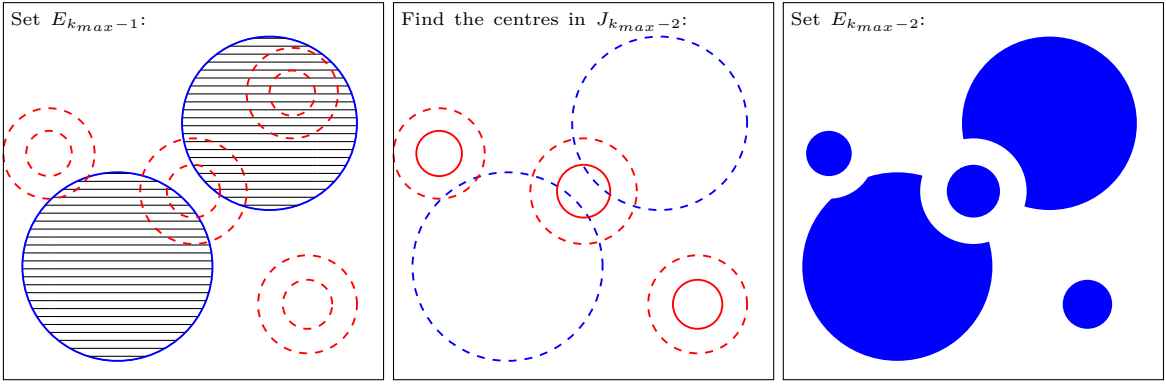

 \medskip
 
 \item We iterate the previous procedure and construct the sets $J_{k}^\eps$, up to $- 2 \leq k \leq \km$. In the last step $k=-3$, we define $J^\eps_{-3}$ as the set of those elements which either intersect $E_{-2}^\eps$ or that are too close to each other. Thanks to this construction, some elements of $\tilde I^\eps_{-3}$, i.e. the holes which are small and well-separated from the clusters and from each others, do not belong to any of the sets $J^\eps_k$ nor are covered by any of the dilated balls generated by these centres. We then show that the remaining elements in $\tilde I^\eps_{-3}$ constitute the set $n^\eps$ generating the holes $H^\eps_g$.  
  
 \medskip
 
\item We finally define and partition the set $\I^\eps$ generating the holes of $H^\eps_b$ by using the sets $\{J_{k}^\eps \}_{-3 \leq k \leq \km}$: We insert in each $\I_k^\eps$ the centres of the balls of $H^\eps$ such that  $k$ is the smallest integer for which $J_k^\eps$ provides a covering. 
\end{itemize}

\medskip

{\bf Step 3. Conclusion.} We show that with these definitions of $J^\eps, \I^\eps_k$ and $\lambda^\eps_j$, the covering obtained in the previous step satisfies all 
the properties of Lemma \ref{l.geometry} and Lemma 
\ref{l.geometric.v2}.

\subsection{Proof of Lemma \ref{l.geometry} and Lemma \ref{l.geometric.v2}.}\label{s.geometry.proof}

\begin{proof}[Proof of Lemma \ref{l.geometric.v2}] 
In the sake of a leaner notation, when no ambiguity occurs we drop the index $\eps$ in the sets of points (e.g. $I^\eps_k, J^\eps_k, \cdots$) and holes which are generated by them. 

\smallskip

\textbf{Proof of Step 1.} 
We start by fixing a (total) ordering $\leq$ of the points in $\Phi^\eps(D)$ such that 
$$
z_i \leq z_j \Rightarrow \rho_i \leq \rho_j,
$$
 with $\rho_i$ and $\rho_j$ the radii of the balls in $H^\eps(D)$ centred in $z_i$ and $z_j$, respectively. We fix $\alpha > 1$ and set $C_0(\alpha, M) = (2\alpha M)^{M(k_{max} + 3)} < +\infty $, where $M= M(\beta, d) \in \N$ is as in Lemma \ref{l.chains}. We only consider $\omega \in \Omega$ belonging to the full-probability subset of $\Omega$ satisfying  Lemma \ref{l.chains} with $\alpha=C_0$ and $\delta$ as in \eqref{def.kappa}.

\smallskip

We introduce some more notation which is needed to implement the construction sketched in Step 1: Let $\Psi^\eps \subset \Phi^\eps(D)$ be any sub-collection of centres and let $\rr^\eps = \{R_i\}_{z_i \in \Psi_\eps} \subset \R_+^{\#\Psi^\eps}$ be their associated radii. Throughout this proof, unless there is danger of ambiguity, we forget about the dependence of both $\Psi$ and $\rr$ on $\eps$. For any two centres $z_i, z_j \in \Psi$ with radii $R_i$ and $R_j$, respectively, we write
  \begin{align}\label{cluster.connection}
   z_i \stackrel{\alpha}{-} z_j \Leftrightarrow \ \ B_{\alpha \aeps R_j}(\eps z_j) \cap B_{\alpha \aeps R_i}(\eps z_i) \neq \emptyset.
  \end{align}
 We define a notion of connection between points and associated radii in the following way: We say that $(z_i, R_i)$ and $(z_j, R_j)$ are connected, and we write that $z_i \sim_{(\Psi, \rr),\alpha} z_j$ whenever 
  \begin{align}\label{equivalence.relation}
 \text{ $\exists \ z_1, \cdots z_m \in \Psi$ s.t. } \ z_i \stackrel{\alpha}{-} z_1 \stackrel{\alpha}{-}\cdots \stackrel{\alpha}{-} z_m \stackrel{\alpha}{-} z_j.
\end{align}
This equivalence relation depends on $\eps$, but we forget about it in the notation. We use the notation  $[z_i](\Psi, \rr,\alpha)$ for  each equivalence class with respect to the previous equivalence relation $\sim_{(\Psi, \rr) \alpha}$. Each equivalence class constitutes a cluster of balls in the sense of \eqref{cluster.connection}.

\smallskip

By using this notation we may reformulate the result of Lemma \ref{l.chains}: 
For almost every $\omega \in \Omega$, every $\eps \leq \eps_0(\omega, d, \beta)$ and any $k \geq -2$, if we choose  $\Psi = I_k \cup I_{k-1}$, and $\rr = \{ \rho_i\}_{z_i \in \Psi}$, we have
\begin{align}\label{reformulation.equivalence.class}
 \sup_{z \in \Psi}\bigr( \# [ z ]( \Psi, {\rr}, C_0) \bigr) \leq M,
\end{align}
i.e. every equivalence class contains at most $M$ elements of $\Psi$. From now on, we thus fix $\omega \in \Omega$ and $\eps \leq \eps_0(\omega,d, \beta)$ satisfying this bound.

\smallskip

Given $\Psi \subset \Phi^\eps(D)$, we introduce the map $T^{\Psi, \alpha} : \R_+^{\#\Psi} \rightarrow \R_+^{\#\Psi}$ which acts on $\rr= \{ R_i\}_{z_i \in \Psi}$ as
 \begin{align}\label{dyn.system}
 (T^{\Psi, \alpha} ( \rr))_j := \begin{cases}
 0 &\text{ if $\max\{ z_i \in [ z_j]_{\Psi_\rr,\alpha} \} \neq z_j$ }\\
 \max_{z_i \in [ z_j ]_{\Psi_\rr,\alpha} }(\eps^{1-\frac{d}{d-2}}|z_j - z_i| + R_i) &\text{ if $\max\{ z_i \in [ z_j ]_{\Psi_\rr,\alpha} \} = z_j$}
 \end{cases}
 \end{align}
We recall that the maximum above is taken with respect to the ordering $\leq$ between centres of $\Psi^\eps(D)$.  
We observe that \eqref{dyn.system} implies that, if  $[ z_j ]({\Psi, \rr,\alpha}) = \{ z_j \}$, then
\begin{align}
	T^{\Psi, \alpha} ( \rr))_j = R_j.
\end{align}

 \medskip
 
By relying on \eqref{reformulation.equivalence.class}, we use an iteration of the previous map to implement the construction sketched at Step 1. We begin by considering $k=-2$ and setting $\Psi = I_{-2} \cup I_{-3}$ and  $\rr = \{ \rho_i\}_{z_i \in \Psi}$.  We define the dynamical system
 \begin{align}\label{definition.comparable.radii}
 \begin{cases}
 \rr(n)= T^{\Psi, \alpha}( \rr(n-1) ) \ \ \ n \in \N\\
 \rr(0)=  \rr \end{cases}
 \end{align}
  and claim that
  \begin{align}\label{stability}
   \rr(n)&=  \rr(M)\ \ \ \forall n \geq M \\
  (\rr(n))_j &\leq (2 \alpha M)^n \rho_j  \quad  \forall z_j \in \Psi, \ \ \forall n \leq M.
  \label{bound.R_j}
  \end{align}

\smallskip

We start with \eqref{bound.R_j} and prove it by induction over $n \leq M$. By definition (cf. \eqref{definition.comparable.radii}), the inequality trivially holds for $n=0$. Let us now assume that \eqref{bound.R_j} holds for some $0 \leq n < M$. We claim that at step $n+1$, each equivalence class $[ z_i ](\Psi, {\rr(n)}, \alpha )$ contains at most $M$ elements: If otherwise, by the inductive hypothesis \eqref{bound.R_j} for $n$ and the choice of the constant $ C_0(M, \alpha)$, also the equivalence class $[ z_i ](\Psi, {\rr(0)}, C_0)$ contains more than $M$ elements. Since we chose $\rr(0)= \{ \rho_i\}_{z_i \in \Psi}$,by our choice of $\omega \in \Omega$ and $\eps \leq \eps(\omega,C_0)$,  property \eqref{reformulation.equivalence.class} is contradicted. Thus, each equivalence class $[ z_i ](\Psi, {\rr(n)}, \alpha)$ contains at most $M$ elements. This allows us to bound
 \begin{align}
	(\rr(n+1))_j \stackrel{\eqref{definition.comparable.radii}}{\leq} 2\alpha\sum_{z_i \in [ z_j ]({\Psi, {\rr (n)},\alpha} )} R(n)_i.\stackrel{\eqref{bound.R_j}}{\leq}
	(2\alpha)^{n+1} M^n \sum_{z_i \in [ z_j ]({\Psi, {\rr (n)},\alpha} )} \rho_i
	 \end{align}
We now observe that by construction \eqref{definition.comparable.radii} and definition \eqref{dyn.system}, either $\rr(n+1)_j =0$, and thus the bound \eqref{bound.R_j} holds trivially, or $\rho_j \geq \rho_i$ for all $z_i \in [ z_j ]({\Psi, {\rr (n)},\alpha} )$. Thus, the previous inequality implies that
 \begin{align}\label{dyn.sys.1}
	(\rr(n+1))_j {\leq}
	(2\alpha M)^{n+1} \rho_j,
	 \end{align}	
i.e. inequality \eqref{bound.R_j} for $n+1$. The induction proof for \eqref{bound.R_j} is complete. 

\smallskip

We now show \eqref{stability}: We begin by remarking that, by construction, if we have $\rr(M) \neq \rr (M+1)$, then there exist  $z_1, \cdots, z_{M+1}$ such that
\begin{align}
\bigcup_{k=1}^{M+1}B_{\aeps \rho_k}(\eps z_k) \subset B_{\aeps \rr(M+1)_1}(\eps z_1).
\end{align}
This, together with estimate \eqref{bound.R_j} for $n=M$, implies that the equivalence class
$[ z_i ]({\Psi, {\rr(0)}, C_0})$ contains more than $M$ elements. As above, this contradicts our choice of the realization $\omega \in \Omega$ and $\eps$. We established \eqref{stability}.

Equipped with properties \eqref{bound.R_j} and \eqref{stability} we may set for every $z_i \in \Phi^\eps(D)$
\begin{align}\label{updated.radii}
\rr^{(-2)}_j := \begin{cases}
\rr(M) \ \ \ \ &\text{if $z_i \in I_{-2}\cup I_{-3}$}\\
 \rho_i\ \ \  &\text{otherwise}
\end{cases}
\end{align}
and define
\begin{align}\label{updated.centres}
\tilde I_{-3}:= \{ z_i \in I_{-3} \ \colon \ \rr^{(-2)}_i > 0 \}.
\end{align}
Note that this definition of $\rr^{(-2)}$ implies that the balls
\begin{align}
\{ B_{\alpha \aeps \rr^{(-2)}_i}(\eps z_i ) \}_{z_i \in I_{-2} \cup \tilde I_{-3} }
\end{align}
are pairwise disjoint.

\medskip

We now iterate the previous step up to $k=\km$: For each $-1 \leq k \leq \km$ we define recursively
\begin{align}\label{updated.radii.1}
\rr^{(k)}_j := \begin{cases}
\rr(M) \ \ \ \ &\text{if $z_i \in I_{k}\cup I_{k-1}$}\\
\rr^{(k-1)}\ \ \  &\text{otherwise,}
\end{cases}
\end{align}
where $\rr(M)$ is obtained by solving \eqref{dyn.system} with $\Psi = I_{k} \cup I_{k-1}$ and $\rr(0) = \rr^{(k-1)}$.
We note that for a general $-1 \leq k \leq \km$, \eqref{bound.R_j} turns into
 \begin{align}
  (\rr^{(k)}(n))_j &\leq (2 \alpha M)^{(k+2)M + n} \rho_j \quad \quad \forall z_j \in \Psi, \,  ~  \forall n \leq M.
  \label{bound.R_j.k}
 \end{align}
In fact, since for $n \leq M$ we have $(2 \alpha M)^{(k+2)M + n} \leq C_0 $, property \eqref{stability} follows by this inequality exactly as in the case $k=-2$ shown above.
We emphasize that, by definition \eqref{updated.radii.1}, at each step $k$ we have
that the balls
\begin{align}\label{disjoint.at.each.step}
\{ B_{\alpha \aeps \rr^{(k)}_i}(\eps z_i ) \}_{z_i \in I_{k} \cup \tilde I_{k-1}, \rr^{(k)}_i >0 }
\end{align}
are pairwise disjoint.

\medskip

From the previous construction we construct the sets $\tilde I_k$ and the parameters $\{\tilde \lambda_i\}_{z_i \in \bigcup_{k=-3}^{\km} \tilde I_k}$ of Step 1: 
For every $-3 \leq k \leq k_{max}$, let
\begin{align}\label{iteration.I.tilde}
	\tilde I_{k} &:= \{ z_i \in  \I_k \colon (\rr^{(k+1)}(M))_i > 0 \},\\
	\tilde \lambda_i &= \frac{(\rr^{(k+1)}(M))_i}{\rho_i} \quad \text{for } z_i \in \tilde I_k.
\end{align}

\medskip

By \eqref{bound.R_j.k} and the definition of the sets $\tilde I_k$, we immediately have that each $\tilde \lambda_i \geq 1$ and is bounded by
$\tilde \Lambda := (2\alpha M)^{(\km+3)M}$. It remains to argue that $\tilde I^k$ satisfy  \eqref{covering.comparable} and \eqref{no.intersection.comparable}: Property 
\eqref{covering.comparable} follows immediately from the construction and the definition of  the operator $T^{\Psi, \alpha}$. To prove \eqref{no.intersection.comparable},
we claim that is enough to show that for every  $k= -2, \cdots, \km$ and $z_i \in \tilde \I_{k}$,
\begin{align}\label{dicotomy}
\tilde \lambda_i = \frac{\rr^{(k)}_i}{\rho_i}.
\end{align}
Indeed, if this is true, then  \eqref{no.intersection.comparable} follows immediately from \eqref{disjoint.at.each.step}.

\smallskip

Let $-2 \leq k \leq \km$ be fixed. By \eqref{updated.radii.1}, to show \eqref{dicotomy} it enough to prove that  
\begin{align}
\rr^{(k)}_i = \rr^{(k+1)}_i, \ \ \ \ \text{for all $z_i \in \tilde I_k$.} 
\end{align}
Since by \eqref{updated.radii.1} we have for all  $z_i \in \tilde I_k$ that $\rr^{(k+1)}_i = \rr(M)_i$, with $\rr(M)$ solving
 \begin{align*}
 \begin{cases}
 \rr(n)= T^{\Psi, \alpha}( \rr(n-1) ) \ \ \ n \in \N\\
 \rr(0)=  \rr^{(k)}, \end{cases}
 \end{align*}
we need to make sure that $\rr(n)_i = \rr^{(k)}_i$ for each $1 \leq n \leq M$. By induction we show that for $z_i \in I_k$ we have 
\begin{align}\label{iteration.disjoint}
 \rr(n)_i \neq \rr^{(k)}_i \Rightarrow \rr(n+1)_i = \rr^{(k+1)}=0
\end{align}
This implies \eqref{dicotomy} by definition \eqref{iteration.I.tilde}.

For $n =1$,  property \eqref{iteration.disjoint} is an easy consequence of \eqref{disjoint.at.each.step} for the balls generated by points $z_i \in I_k$.
Let us assume that \eqref{iteration.I.tilde} holds at step $n$. Then, again by \eqref{iteration.I.tilde}, we have that for $z_i \in I_k$ either $\rr(n)_i=0$, or 
$\rr(n)_i = \rr^{(k)}_i$. Thus, if $\rr(n+1)_i \neq \rr(n)_i$, we necessarily have again by \eqref{disjoint.at.each.step} that there exists $z_j \in I_{k+1}$ such that
\begin{equation}
B_{\alpha \aeps\rr^{(n-1)}_j}(\eps z_j) \cap B_{\alpha \aeps \rr(n-1)_i}(\eps z_i) \neq \emptyset.
\end{equation}
This implies that $\rho_j \geq \rho_i$ and in turn that $z_j \geq z_i$. By definition of the map $T^{\Psi,\alpha}$, this yields necessarily that $\rr(n+1)_i=0$. The proof of
\eqref{iteration.disjoint} is complete. This establishes \eqref{dicotomy} and concludes the proof of \eqref{no.intersection.comparable}.

\medskip

{
We conclude this step with the following remark: Let $\Phi_{2\eps^{\delta/2}}^\eps(D)$ be the thinned process (see \eqref{thinned.process}) with $\delta$ fixed as in \eqref{def.kappa}. Moreover, let $S^\eps:= \Phi^\eps(D) \backslash \Phi^\eps_{2\eps^{\delta/2}}(D)$ and
\begin{align}
	\label{I_-3^g,b}
	I_{-3}^g = I_{-3} \cap \Phi^\eps_{2\eps^{\delta/2}}(D), \qquad I_{-3}^b = I_{-3} \setminus I_{-3}^g =  I_{-3} \cap S^\eps.
\end{align}
We claim that, up to taking $\eps_0= \eps_0(d,\beta)$ smaller than above, we have
\begin{align}
	\label{lambda=1.good.particles}
	I_{-3}^g \subset \tilde I_{-3}, \qquad \tilde \lambda_i = 1 \quad \text {for all } z_i \in I_{-3}^g.
\end{align}
As will be shown in the next step, the set  $I_{-3}^g$ contains the set $n^\eps$ generating $H^\eps_g$. 

\smallskip

To show \eqref{lambda=1.good.particles}, we observe that whenever $z_i, z_j \in I_{-3}^g \cup I_{-2}$ with $z_i \neq z_j$, then we may choose $\eps$ small enough to infer that
\begin{align}
	\label{good.particles.separated}
	B_{\alpha \aeps \rho_i}(\eps z_i) \cap B_{\alpha \tilde \Lambda \aeps \rho_j}(\eps z_j) = \emptyset.
\end{align}
Indeed, for  $\eps^{\frac \delta 2 } \leq (\alpha\tilde\Lambda)^{-1}$, we bound
\begin{align}
	\eps |z_i - z_j| \stackrel{\eqref{I_-3^g,b}}{\geq} 2\eps^{1 + \frac \delta 2} \geq 2 \alpha \tilde \Lambda \eps^{1 + \delta}
	\stackrel{\eqref{def.I_k}}{\geq} \aeps (\alpha \rho_i +\tilde \Lambda \rho_j).
\end{align}
This implies that after $M$ iterations of the dynamical system \eqref{dyn.sys.1}, we have $\rr(M)= \rho_i$ for all $z_i \in I^{g}_{-3}$. Thanks to \eqref{iteration.I.tilde}
we obtain \eqref{lambda=1.good.particles}.

}

\bigskip

\textbf{Proof of Step 2.} In this step we rigorously implement the method sketched in Step 2 and construct the sets $J^\eps_k$ as subsets of  $\tilde I^\eps_k$, 
$-3 \leq k \leq \km$. We define $\lambda_j = \theta^2 \tilde \lambda_j$, with $\tilde \lambda_j \in [1, \tilde\Lambda]$ constructed in Claim 1 of Step 1,
and $\theta^4=\alpha$. Clearly, we may choose the upper bound $\Lambda$ in the statement of Lemma \ref{l.geometric.v2} as $\Lambda := \theta \tilde\Lambda$.
We start by setting
\begin{align}\label{def.J.max}
	J_{k_{max}} &:= \tilde I^\eps_{k_{max}}, \\
	E_{k_{max}} &:= \bigcup_{z_j \in J_{k_{max}}}  B_{\lambda_j \aeps \rho_j} (\eps z_j),
\end{align}
and inductively define for $-1 \leq l \leq k_{max}$
\begin{align}
	\label{def.J_l}
	J_{l-1} &:= \biggl\{z_j \in \tilde I_{l-1} \colon B_{\theta \tilde \lambda_j \aeps \rho_j} (\eps z_j) \not \subset E_l\biggr \}, \\
	\label{def.E_l}
	E_{l-1} &:= \biggl( E_l \backslash \bigcup_{z_j \in J_{l-1}}  B_{\theta \lambda_j \aeps \rho_j} (\eps z_j) \biggr) \cup  \bigcup_{z_j \in J_{l-1}} B_{\lambda_j \aeps \rho_j} (\eps z_j).
\end{align}
To construct the remaining sets $J_{-3}$ and $E_{-3}$, we need an additional step: We recall the definition of $S^\eps$ and $I_{-3}^g$ from
\eqref{thinned.process} and \eqref{I_-3^g,b}, respectively. We first set
\begin{align}\label{E.tilde}
	\tilde J_{-3} &:= \biggl\{z_j \in \tilde I_{-3} \cap S^\eps \ \colon \ B_{\theta \lambda_j \aeps \rho_j} (\eps z_j) \not \subset E_{-2} \biggr \}, \\	
	\tilde E_{-3} &:= \biggl( E_{-2} \backslash \bigcup_{z_j \in \tilde J_{-3}}  B_{\theta \lambda_j \aeps \rho_j} (\eps z_j) \biggr)
	 \cup  \bigcup_{z_j \in \tilde J_{-3}} B_{\lambda_j \aeps \rho_j} (\eps z_j).
\end{align}
Finally, for $z_i \in \Phi^\eps(D)$ we define the set  
\begin{align}\label{def.K.eps}
	K^\eps :=\biggl\{ z_j \in   I^g_{-3}  \  \colon \  B_{2 \eps^{1+\delta}}(\eps z_j) \cap  \bigcup_{z_i \in \cup_{k=-2}^{\km} J_k \cup \tilde J_{-3}} B_{\theta \lambda_i \aeps \rho_i}(\eps z_i) \neq \emptyset \biggr\},
\end{align} 
and finally consider
\begin{align}\label{def.J_-3}
	J_{-3} &:=  \tilde J_{-3} \cup \biggl\{ z_j \in K^\eps \colon  B_{\theta \lambda_j \aeps \rho_j} (\eps z_j)  \not \subset \tilde E_{-3} \biggr\}, \\
	\label{def.E_-3}
	\tilde E_{-3} &:= \biggl( E_{-2} \backslash \bigcup_{z_j \in  J_{-3}}  B_{\theta \lambda_j \aeps \rho_j} (\eps z_j) \biggr)
	 \cup  \bigcup_{z_j \in J_{-3}} B_{\lambda_j \aeps \rho_j} (\eps z_j).
\end{align}

We remark that in the definitions of $E_l$, the annuli $  B_{\theta\lambda_j \aeps \rho_j} (\eps z_j) \backslash B_{\lambda_j \aeps \rho_j} (\eps z_j)$ are cut out in order to satisfy  \eqref{small.dont.intersect.big}. 
Moreover, we observe that each connected component of the set $E_{k}$ is a subset of $B_{\lambda_{j} \aeps \rho_{j}}(\eps z_{j})$ for some $z_{j} \in J_l$, for $k \geq l$. This follows from the the definition of $E_k$ and \eqref{no.intersection.comparable}.

\smallskip

We finally denote
\begin{align}\label{def.J}
	J := \bigcup_{k=-3}^{k_{max}}  J_{k}.
\end{align}
and define the set $\I$ of the centres generating $H^\eps_b$ as
\begin{align}\label{def.bad.set.0}
	\mathcal I &:= \Bigl\{ z_i \in \Phi^\eps(D) \ \colon \ B_{\aeps \rho_i}(\eps z_i) \subset B_{\lambda_j \aeps \rho_j}(\eps z_j) \text{ for some } z_j \in J \Bigr\}, \\
	\label{def.bad.set}
	\mathcal I_k &:=  \Bigl\{ z_i \in \mathcal I \colon k \text{ is minimal such that } B_{\aeps \rho_i}(\eps z_i) \subset B_{\lambda_j \aeps \rho_j}(\eps z_j) 
	\text{ for } z_j \in J_k \Bigr\}.
\end{align}
Equipped with the previous definition, we construct  $H_b^\eps$, $\bar H_b^\eps$ and $D^\eps_b$ as shown in \eqref{H^b}, \eqref{bar.H^b}, and \eqref{D_b}.

\bigskip

{\bf Proof of Step 3.} We first argue that the sets $H_b^\eps, \bar H^\eps_b$ and $D^\eps_b$ constructed in the previous step satisfy the conditions of Lemma
\ref{l.geometry}. 

\smallskip

We begin by claiming that 
\begin{align}
	\label{char.n_eps}
	n_\eps = I^g_{-3} \backslash K^\eps,
\end{align}
with $K^\eps$ defined in \eqref{def.K.eps}. Since, by construction we set $H^\eps_g = H^\eps \backslash H^\eps_b$, by \eqref{H^b} this also reads as  
\begin{align}\label{char.n.eps.1}
 \Phi^\eps(D) \backslash \I =  I^g_{-3} \backslash K^\eps.
\end{align}
The $\supseteq$-inclusion is a consequence of the fact that by \eqref{lambda=1.good.particles} we have by construction $I^g_{-3} \cap \tilde J_{-3} = \emptyset$ (see \eqref{E.tilde}, 
\eqref{thinned.process}). This yields that in the definition \eqref{def.J_-3} of $J_{-3}$ the only elements of $I^g_{-3}$ in $J_{-3}$ are the ones contained in $K^\eps$.
By
\eqref{def.J_l} and \eqref{def.J}, this yields that $(I^g_{-3} \backslash K) \cap J = \emptyset$. We now use \eqref{def.bad.set} to infer that also
$(I^g_{-3} \backslash K^\eps) \cap \I = \emptyset$, i.e. the $\supset$-inclusion in \eqref{char.n.eps.1}.

\smallskip

For the $\subset$ inclusion we argue the complementary statement which, by \eqref{I_-3^g,b}, also reads as
\begin{align}\label{first.inclusion}
 K^\eps \cup \bigcup_{k \geq -2}I^\eps_{k} \cup I^b_{-3} \subset \I.
\end{align}
We show how to argue that $I_k \subset \I$, for some $k \geq -2$. The argument for the other sets is analogous.

\smallskip

Let $z_i \in I_k$. Then, by \eqref{covering.comparable}, 
there exists $l \geq k$, $z_{j_1} \in \tilde I_l$ such that 
$$
B_{\aeps \rho_i}(\eps z_i) \subset B_{\aeps \tilde \lambda_{j_1} \rho_{j_1}}(\eps z_{j_1}).
$$
By definition \eqref{def.J_l}, this yields that either $z_{j_1} \subset J_l$ or 
$$
B_{\aeps \theta \tilde \lambda_{j_1} \rho_{j_1}}(\eps z_{j_1}) \subset E_{l+1}.
$$
In the first case, it is immediate that $z_i \in \I$ (see \eqref{def.bad.set.0}); in the second case, since each connected component of the set $E_{l+1}$ is a subset of 
a ball $B_{\lambda_{j_2} \aeps \rho_{j_2}}(\eps z_{j_2})$ for some $z_{j_2} \in J_{l_2}$ with $l_2 > l_1$,  
it follows that 
$$
B_{\aeps \rho_i}(\eps z_i) \subset B_{\lambda_{j_2} \aeps \rho_{j_2}}(\eps z_{j_2}).
$$
Hence, also in this case $z_i \in \mathcal I$. We established $I_k \subset \I$. This concludes the proof of \eqref{first.inclusion} and thus also of 
\eqref{char.n.eps.1} and \eqref{char.n_eps}.

\smallskip

From identity \eqref{char.n_eps}, the second line of \eqref{good.set} immediately follows by \eqref{I_-3^g,b} and definition \eqref{def.I_k} for the set $I_{-3}$. 
In addition, since $K^\eps$ is not contained in $n^\eps$, also the first inequality in \eqref{safety.layer} holds. The remaining claims in \eqref{safety.layer}, \eqref{good.set}, and \eqref{small.distance.bad} may be obtained from \eqref{first.inclusion} similarly to \cite{GHV1}[Lemma 4.2], thanks to the very explicit definition of the sets $\bar H^\eps_b$ and $D^\eps_b$. 

\medskip

In the sake of completeness we give these arguments explicitly: We claim
\begin{align}
	\label{I.small}
	\lim_{\eps \downarrow 0} \eps^{d} \#(\I) = 0.
\end{align}
By taking the complement with respect to $\Phi^\eps(D)$ in \eqref{char.n.eps.1}, we have
\begin{align}
	\I = \bigcup_{k=-2}^{\km} I_k \cup I_{-3}^b \cup K^\eps.
\end{align}
We estimate the limit for $\eps \downarrow 0^+$ for the first sets on the right-hand side by appealing to Lemma \ref{SLLN.general.pp} and \eqref{def.kappa} (we recall that we assumed $\beta \leq 1$): Indeed, we have
\begin{align}
	\limsup_{\eps \downarrow 0} \eps^{d} \#(\bigcup_{k=-2}^{\km} I_k) &= \limsup_{\eps \downarrow 0}  \eps^{d} \# \{ z_i \in \Phi^\eps(D) \colon \aeps \rho_i \geq \eps^{1+2\delta} \} \\
	& \leq \limsup_{\eps \downarrow 0}   \eps^{d -(d-2)(1+2\delta)} \eps^{d}  \sum_{z_i \in \Phi^\eps(D)} \rho_i^{d-2} \to 0 \\
	& \lesssim \limsup_{\eps \downarrow 0} \eps^{2(1-(d-2)\delta)} = 0.
\end{align}

We now turn to $I_{-3}^b$: Let $\{\delta_k \}_{k \in \N}$ be any sequence such that $\delta_k \downarrow 0^+$. Since $2\eps^{\delta/2} \to 0$, we estimate for any $\delta_k>0$
\begin{align}
	 \limsup_{\eps \downarrow 0^+} \eps^d \#(I_{-3}^b)\stackrel{\eqref{I_-3^g,b}}{\leq} 
	 \limsup_{\eps \downarrow 0^+} \eps^d \bigl(N^\eps(D) - N^\eps_{2\eps^{\delta/2}}(D)\bigr) 
	 \stackrel{\eqref{thinned.process}}{\leq} \lim_{\eps \downarrow 0^+} \eps^d \bigl(N^\eps(D) - N^\eps_{\delta_k}(D)\bigr).
\end{align}
 We now apply Lemma \ref{SLLN.general.pp} to $\Phi$ and each $\Phi_{\delta_k}$, $k \in \N$, to deduce that almost surely and for every $\delta_k >0$
\begin{align}
	\limsup_{\eps \downarrow 0^+} \eps^d \#(I_{-3}^b) \leq \lambda |D| - \langle N_{\delta_k}(D) \rangle.
\end{align}
By sending $\delta_k \downarrow 0^+$, we use once more Lemma \ref{SLLN.general.pp} on the last term on the right-hand side above and obtain 
\begin{align}
	\label{number.K_b}
	\lim_{\eps \downarrow 0^+} \eps^d \#(I_{-3}^b) = 0.
\end{align}

\smallskip

To conclude the proof of \eqref{I.small}, it thus remains to show that almost surely also
\begin{align}
	\label{number.tildeI_b}
 \eps^d \#(K^\eps) \rightarrow 0 \ \ \ \ \ \eps \downarrow 0^+.
\end{align}
We have for all $z_i \in  K^\eps \subset I^g_{-3}$
\begin{align}\label{minimal.dist}
 \min_{z_j \in \Phi^\eps(D) \backslash\{ z_i\}}\eps| z_j - z_i| \geq 2\eps^{1+\delta/2}, \qquad  \aeps\rho_i < \eps^{1+2\delta}.
\end{align}
In particular, by the first inequality above, the balls $\{ B_{\eps^{1+2\delta} }( \eps z_i ) \}_{z_i \in K^\eps}$ are all disjoint, and therefore
\begin{align}
	\label{K^eps.1}
\eps^d \#(K^\eps) &\lesssim \eps^d \sum_{z_i \in K^\eps} \eps^{-d(1+2\delta)} |B_{\eps^{1+2\delta}}(\eps z_i)| 
=  \eps^{-2d \delta} \sum_{z_i \in \tilde I^\eps_b} |B_{\eps^{1+2\delta} }(\eps z_i)| .
\end{align}
In addition, we observe that by definition of $K^\eps$, for any $z_i \in K^\eps$ there exists $z_j \in \cup_{k=-2}^{\km} J_k$ such that
\begin{align}\label{intersection.1}
	B_{2\eps^{1+\delta} }(\eps z_i) \cap B_{\theta \lambda_j \aeps \rho_j}(\eps z_j) \neq \emptyset.
\end{align}
Here we used $K_\eps \subset \tilde I_{-3}$ and \eqref{no.intersection.comparable} to rule out that $z_j \in J_{-3} \subset \tilde I_{-3}$.
In particular, \eqref{minimal.dist} and \eqref{intersection.1} imply
\begin{align}
	2\eps^{1+\delta/2} \leq \eps |z_i - z_j| \leq 2\eps^{1+\delta} + \theta \lambda_j \aeps \rho_j,
\end{align}
we obtain that $\theta \lambda_j \aeps \rho_j \geq 2 \eps^{1+\delta}$. We combine this inequality with condition \eqref{intersection.1} to infer that
\begin{align}
	B_{\eps^{1+2\delta} }(\eps z_i) \subset B_{2 \theta \lambda_j \aeps \rho_j}(\eps z_j)
\end{align}
and, by \eqref{K^eps.1}, to estimate
\begin{align}
	\eps^d \#(K^\eps) & \lesssim  \eps^{-2d\delta} \sum_{z_j \in \cup_{k=-2}^{\km} J_k} |B_{2 \theta \lambda_j \aeps \rho_j}(\eps z_j)| \\
	&\lesssim \eps^{-2d \delta} \biggl(\eps^{\frac{d}{d-2}} \max_{z_j \in \Phi^\eps(D)} \aeps \rho_j \biggr)^2 
	\sum_{z_j \in \cup_{k=-2}^{\km} J_k}  (\aeps \rho_j)^{d-2}\\
	&\stackrel{\eqref{delta.max}}{\lesssim}\eps^{2\delta d} \sum_{z_j \in \Phi^\eps(D)}  (\aeps \rho_j)^{d-2}.	
\end{align}
Thanks to Lemma \ref{SLLN.general.pp}, the right-hand side vanishes almost surely in the limit $\eps \downarrow 0^+$. This concludes the proof of \eqref{I.small}.

\smallskip

The limit in the first line of \eqref{good.set} is a direct consequence of \eqref{I.small}. Moreover, the second inequality in \eqref{safety.layer}
follows from \eqref{I.small} and Lemma \ref{small.index.set}.

\medskip

 To show \eqref{small.distance.bad}, we resort to the definition of $D^\eps_b$ to estimate
  \begin{align*}
 \bigl\{  z_i \in \Phi^\eps_{2\eta}(D)(\omega)  &\colon  \mbox{dist}( z_i , D^\eps_b ) \leq \eta \eps \bigr\}\\
 & \begin{aligned}
 	\subset \I & \cup \Bigl\{  z_i \in n^\eps(\omega)  \colon  \dist\Bigl( z_i , \bigcup_{z_j \in \cup_{k=-2}^{\km} J_k} B_{\Lambda \aeps \rho_j}(\eps z_j)  \Bigr) \leq \eta \eps \Bigr\} \\
  &\cup  \Bigl\{  z_i \in n^\eps(\omega) \cap \Phi^\eps_{2\eta}(D)(\omega) \colon  \dist\Bigl( z_i , \bigcup_{z_j \in J_{-3}} B_{\Lambda \aeps \rho_j}(\eps z_j)  \Bigr) \leq \eta \eps \Bigr\}
  \end{aligned} \\
 &:= I_b^\eps \cup F^\eps \cup C^\eps .
\end{align*}
We already know $ \eps^d \# (I_b^\eps) \to 0$.
Next, we argue that 
\begin{align}\label{estimate.set.1}
\eps^d \#( F^\eps ) \rightarrow 0.
\end{align}
This follows by an argument similar to the one for \eqref{number.tildeI_b}: 
We may choose $\eps_0=\eps_0(d)$ such that for all $\eps \leq \eps_0$,  $\eps^{\delta/2} \leq \eta$. By definition of $J_k$ and of $F^\eps$ above, we infer that for such $\eps \leq \eps_0$, for all $z_j \in F^\eps$ there exists $-2 \leq k \leq \km $ and $z_i \in J_k$ such that
\begin{align}\label{inclusion.E}
	B_{\eps^{1+\delta/2}}(\eps z_j) \subset B_{2 \eta \eps + \Lambda \aeps \rho_i}(\eps z_i) \subset B_{2 \Lambda \eta \eps^{-2\delta} \aeps \rho_i}(\eps z_i),
\end{align}
where in the second inequality we use that $\eps^{-2 \delta} \eta \geq  1$ and $\aeps \rho_i \geq \eps^{1+2\delta}$.  We note that by \eqref{minimal.dist} the balls $B_{\eps^{1+\delta/2}}(\eps z_j) $ with $z_j \in n^\eps$ are all disjoint. Hence,
\begin{align*}
\eps^d \#( F^\eps )& \stackrel{\eqref{inclusion.E}}{\lesssim} \eps^{-d\delta} 
\biggl | \bigcup_{z_i \in \cup_{k=-2}^{\km} J_k} B_{2 \Lambda \eta \eps^{-2\delta} \aeps \rho_i}(\eps z_i) \biggr| \\
&{ \lesssim} \eta^d \eps^{-d(\delta+2\delta)}  \biggl(\max_{z_j \in \Phi^\eps(D)} \aeps \rho_j \biggr)^2 
\sum_{z_j \in \Phi^\eps(D)}  (\aeps \rho_j)^{d-2}\\
&\stackrel{\eqref{delta.max}}{\lesssim} \eta^d \eps^{d\delta}\sum_{z_j \in \Phi^\eps(D)}  (\aeps \rho_j)^{d-2}.
\end{align*}
The right-hand side vanishes almost surely in the limit $\eps \downarrow 0^+$ thanks to \eqref{power.law} and Lemma \ref{SLLN.general.pp}.

\medskip 

We conclude the argument for  \eqref{small.distance.bad} by showing that the set $C^\eps$ is empty when $\eps$ is small: In fact, by construction, if $z_i \in n_\eps$ satisfies
\begin{align}
	\dist\biggl( \eps z_i , \bigcup_{z_j \in J_{-3}} B_{\Lambda \aeps \rho_j}(\eps z_j)  \biggr) \leq \eta\eps,
\end{align}
then there exists a $z_j \in J_{-3} \subset I_{-3}$ such that for $\eps \leq \eps_0$ with $\Lambda \eps^{2 \delta} \leq \eta$
\begin{align*}
\eps |z_i - z_j| \leq \mbox{dist}\Bigl( \eps z_i, B_{\Lambda \aeps \rho_j}(\eps z_j) \Bigr) + \Lambda \eps^{1+2\delta} \leq 2\eta\eps.
\end{align*}
 This yields $C^\eps \subset \Phi^\eps(D) \backslash \Phi^\eps_{2 \eta}(D)$ and thus that it is empty since by definition we also have $C^\eps \subset \Phi^\eps_{2\eta}(D)$. 
This finishes the proof of \eqref{small.distance.bad}.

\smallskip

We hence have shown that $H^\eps_b, \bar H^\eps_b$ and $D^\eps_b$ in Lemma \ref{l.geometry} may be chosen as in Step 2 (see \eqref{H^b}, \eqref{bar.H^b}, and \eqref{D_b}).
We also remark that it immediately follows by  \eqref{delta.max} and the bounds on $\lambda^\eps_i \leq \Lambda$ obtained at the beginning of Step 2, that the radii $\lambda^\eps_i \aeps \rho_i$ 
generating the balls of $\bar H^\eps_b$ satisfy the second inequality in \eqref{bar.H^b}.

\medskip

It remains to argue \eqref{similar.size.apart} and \eqref{small.dont.intersect.big}. The first property follows directly from \eqref{no.intersection.comparable} for
$J_k \subset \tilde \I_k$ and the choice of the parameters $\lambda_i = \theta \tilde\lambda_i$ and $\theta^4= \alpha$. 

\smallskip

We now turn to \eqref{small.dont.intersect.big} and begin by showing that it suffices to prove the following:\\

\textbf{Claim:} For all $-3 \leq k < l \leq k_{max}$ and every $z_k \in J_k$, $z_l \in \tilde I_l$ we have
\begin{align}
	\label{small.dont.intersect.big.3}
	 B_{\tilde \lambda_l \aeps \rho_l}(\eps z_l) \cap B_{\theta \lambda_k \aeps \rho_k}(\eps z_k) = \emptyset.
\end{align}

We first prove \eqref{small.dont.intersect.big.2} provided this claim holds. To do so, for any $k < l$ and $z_j \in J_l$ we begin by denoting by $E^{z_j}_k$  the
set
\begin{equation}\label{char.E_k^z_j}
E^{z_j}_{k} := B_{ \aeps \lambda_j \rho_j}(\eps z_j) \setminus \bigcup_{m = k}^{l-1} \bigcup_{z_i \in J_{m}}  B_{\theta \lambda_i \aeps \rho_i}(\eps z_i)
\end{equation}
and arguing that
\begin{align}\label{char.E_k^z_j.1}
&B_{\aeps \tilde\lambda_j \rho_j}(\eps z_j) \subset E^{z_j}_k \subset E_k,\\
	\label{E_k.disjoint.union}
&E_k = \dot\bigcup_{l \geq k} \dot \bigcup_{z_j \in J_l}  E^{z_j}_k,
\end{align}
where each union above is between disjoint sets. 

\smallskip

By \eqref{def.E_l} for $E_{l-1}$ and \eqref{def.J_l} for $J_l$, we clearly have that 
$$
B_{\aeps \lambda_j \rho_j}(\eps z_j ) \subset E_{l-1}.
$$
Note that, by construction, this ball is a connected component of the set $E_{l-1}$. From the previous inclusion, the second inclusion in \eqref{char.E_k^z_j.1} is an easy
application of the recursive definition \eqref{def.E_l} of $E_k$. Similarly, \eqref{E_k.disjoint.union} 
is an easy consequence of the definition \eqref{def.E_l} of the sets $E_k$. Furthermore, since each 
$J_m \subset \tilde I_m$, we apply claim \eqref{small.dont.intersect.big.3} to $z_j$ and all $z_k \in J_m$ with $m \leq l-1$, and conclude also the first inclusion in 
\eqref{char.E_k^z_j.1}. We conclude that definition \eqref{char.E_k^z_j} immediately yields the monotonicity property $E^{z_j}_{k-1} \subset E^{z_j}_k$ for all $z_j\in J_l$
 and $-3 \leq k \leq l$.

\smallskip

Equipped with \eqref{char.E_k^z_j.1}-\eqref{E_k.disjoint.union}, we now turn to \eqref{small.dont.intersect.big}: Let $z_0 \in \mathcal I_{k_0}$ for some $-2 \leq k_0 \leq k_{max}$. By definition 
\eqref{def.bad.set}, there exists $z_1 \in J_{k_0}$ such that 
\begin{align}
	\label{ball.z_1}
	B_{\aeps \rho_0}(\eps z_0) \subset B_{\lambda_{1} \aeps \rho_1} (\eps z_{1}).
\end{align}
By this, property \eqref{small.dont.intersect.big} follows immediately if we prove that for any $l < k_0$ and all $z_3 \in J_l$ we have
\begin{align}
	\label{small.dont.intersect.big.2}
		B_{\aeps \rho_0}(\eps z_0) \cap B_{\theta \lambda_3 \aeps \rho_3}(\eps z_3) = \emptyset.
\end{align}

\smallskip

Let $-3 \leq k_2 \leq \km$ be minimal such that there exists $z_{2} \in \tilde I^\eps_{k_{2}}$ with the property that
\begin{align}
	\label{ball.z_2}
	B_{\aeps \rho_0}(\eps z_0) \subset B_{\tilde \lambda_{2} \aeps \rho_{2}}(\eps z_{2}).
\end{align}
Note that, by \eqref{covering.comparable}, we may always find such $k_2$. If $k_0 \leq k_2$, we use the above claim \eqref{small.dont.intersect.big.3} on 
$z_2 \in \tilde I_{k_2}$ and $z_3 \in J_l$ with $l < k_2$ and conclude \eqref{small.dont.intersect.big.2}. Let us now assume that $k_0 > k_2$: Since
 $z_0 \in \I_{k_0}$, by definition \eqref{def.bad.set} we have that $z_2 \not \in J_{k_2}$. This implies by \eqref{def.J_l} that 
$$
B_{\theta \tilde \lambda_{2} \aeps \rho_{2}}(\eps z_{2}) \subset E_{k_2 +1}.
$$
In particular, by \eqref{ball.z_2} and \eqref{char.E_k^z_j} there exists a $\tilde k_0 > k_2$ and $\tilde z_1 \in J_{\tilde k_0}$ such that
\begin{align}
	\label{small.dont.intersect.big.4}
	  B_{\aeps \rho_0}(\eps z_0) \subset B_{\theta \tilde \lambda_{2} \aeps \rho_{2}}(\eps z_{2}) \subset E^{\tilde z_1}_{k_2 +1}.
\end{align}
Moreover, by \eqref{char.E_k^z_j} and the assumption $k_2 < k_0$, we also have 
\begin{align}
	 B_{\aeps \rho_0}(\eps z_0) \subset	E^{\tilde z_1}_{k_2 +1}   \subset E^{\tilde z_1}_{k_0}.
\end{align}
On the other hand, by \eqref{ball.z_1} also 
\begin{align}
B_{\aeps \rho_0}(\eps z_0) \subset  B_{\lambda_{1} \aeps \rho_1} (\eps z_{1}) = E^{z_1}_{k_0}.
\end{align}
By combining the previous two inequalities and using that the sets $E^{z_i}_k, E^{z_j}_k$ are whenever $z_i \neq z_j \in J$, we conclude that $\tilde z_1 = z_1$. 
Thus, since $z_1 \in J_{k_0}$, definition \eqref{char.E_k^z_j} applied to $E^{z_1}_{k_2+1}$ yields that for all $k_2 < l < k_0$ we have for all $z_i \in J_l$
\begin{align}
 E^{z_1}_{k_2+1} \cap B_{\theta \aeps \lambda_i \rho_i}(\eps z_i) = \emptyset.
\end{align}
By using \eqref{small.dont.intersect.big.4}, the above inequality implies \eqref{small.dont.intersect.big.2} with $z_i=z_3$ and for all $k_2 < l < k_0$. To extend 
\eqref{small.dont.intersect.big.2} also to the indices $l \leq k_2$ it suffices to observe that for $l < k_2$ we may argue as above in the case $k_0 \leq k_2$. Finally,
 if $l= k_2$, we obtain \eqref{small.dont.intersect.big.2} by applying \eqref{ball.z_2} and \eqref{no.intersection.comparable} to $z_2 \in \tilde I_{k_2}$ and 
 $z_3 \in J_{k_2} \subset \tilde I_{k_2}$.

\bigskip

It remains to prove claim \eqref{small.dont.intersect.big.3}. Let $z_l \in \tilde I^\eps_l$, $-2 \leq l \leq \km$. We begin by arguing that
\begin{align}
	\label{ball.j.subset.E_k-1}
	B_{\theta \tilde \lambda_{l} \aeps \rho_{l}} (\eps z_{l}) \subset E_{l}.
\end{align}
Indeed, if $z_l \in J_{l}$, this follows immediately from the definition of $E_{l}$. If $z_l \not \in J_{l}$, then by \eqref{def.J_l} we have
$B_{\lambda_{l} \aeps \rho_{l}} (\eps z_{l}) \subset E_{l+1}$. We now use \eqref{no.intersection.comparable} on the family $J_l$ and definition \eqref{def.E_l} of $E_l$ 
 to conclude \eqref{ball.j.subset.E_k-1}. 
From \eqref{ball.j.subset.E_k-1} we may use again \eqref{no.intersection.comparable} to the families $J_{l}, J_{l-1}$ and also obtain that
\begin{align}
	\label{ball.j.subset.E_k-2}
	B_{\theta \tilde \lambda_{l} \aeps \rho_{l}} (\eps z_{l}) \subset E_{l-1}.
\end{align}

We are now ready to argue \eqref{small.dont.intersect.big.3} by contradiction: Let us assume that there exists a $k< l$ and $z_k \in J_k$ such that 
\eqref{small.dont.intersect.big.3} fails, i.e.
\begin{align}\label{small.dont.intersect.big.5}
  B_{\tilde \lambda_l \aeps \rho_l}(\eps z_l) \cap B_{\theta \lambda_k \aeps \rho_k}(\eps z_k) \neq \emptyset.
\end{align}
Then, again by \eqref{no.intersection.comparable} applied to $J_{l}$ and $J_{l-1}$, we necessarily have $k \leq l-2$.
Let us now assume that $z_k \in J_{l-2}$: Then by \eqref{def.J_l} we have 
\begin{align}\label{iteration.boundary.1}
B_{\aeps \theta \tilde\lambda_k \rho_k }(\eps z_k) \subsetneq E_{l-1}.
\end{align}
This, together with \eqref{ball.j.subset.E_k-2} for $z_l$ and \eqref{small.dont.intersect.big.5} yields
\begin{align}\label{intersection.boundary}
	B_{\theta \tilde \lambda_k \aeps \rho_k}(\eps z_k) \cap \partial B_{\theta \tilde \lambda_l \aeps \rho_l}(\eps z_l) \neq \emptyset.
\end{align}

{
For a general $k < l-2$, we claim that we may iterate the previous argument and obtain that
\eqref{small.dont.intersect.big.5} implies the existence of an integer $m \leq 1 + \lceil \frac {k_{max}} 2 \rceil$ and a collection $k_0, \cdots, k_m \leq l-2$, 
such that $k=k_0$ and for all $0 \leq n \leq m -1$ we have $k_n \leq k_{n+1}-2$ and there exist $z_{k_n} \in J_{k_n}$ and $z_m \in J_{k_m}$ satisfying
(see Figure \ref{fig.Step3.induction})
\begin{equation}\label{chain.to.boundary}
 \begin{aligned}
		B_{\theta \tilde \lambda_{k_m} \aeps \rho_{k_m}}(\eps z_{k_m}) \cap \partial B_{\theta \tilde \lambda_l \aeps \rho_l}(\eps z_l) &\neq \emptyset, \\
		B_{\theta \tilde \lambda_{k_{n}} \aeps \rho_{k_{n}}}(\eps z_{k_{n}}) \cap B_{\theta \lambda_{k_{n+1}} \aeps \rho_{k_{n+1}}}(\eps z_{k_{n+1}})  &\neq \emptyset.
\end{aligned}
\end{equation}

\begin{figure}
\centering
\begin{tikzpicture}[scale=0.6]
\draw[red, dashed, line width=0.25mm](2,2) circle (4.5);
\draw[red,  line width=0.5mm](2,2) circle (2);
\draw[black,  line width=0.25mm](5.2,4.9) circle (0.8);
\draw[black, dashed, line width=0.25mm](5.2,4.9) circle (1.2);
\draw[black, dashed,  line width=0.25mm](4.3,4) circle (0.4);
\draw[black,  line width=0.25mm](4.3,4) circle (0.3);
\draw[black,  line width=0.25mm](3.8,3.8) circle (0.2);
\node[text width=2.5cm] at (2.5,2.3) {$B_l$};
\node[text width=2.5cm] at (6.4,0.7) {$B_{k_0}$};
\node[text width=2.5cm] at (7.45,3.2) {$B_{k_2}$};
\node[text width=2.5cm] at (6.4,3) {$B_{k_1}$};
\node[text width=2.5cm] at (-0.1,2.3) {$B_{\theta l}$};
\draw[->](4.5,1.1)--(3.9, 3.5);
\draw (-3,7) rectangle (7,-3);
;\end{tikzpicture}
\caption{ The thick ball $B_{l}$ in the centre represents $B_{\theta \tilde \lambda_l \aeps \rho_l}(\eps z_l)$, while the nested dashed ball $B_{\theta l}$ is its 
dilation by $\theta> 1$. The balls $B_{k_0}$, $B_{k_1}$ and $B_{k_2}$ correspond to $B_{\theta \tilde \lambda_{k_{0}} \aeps \rho_{k_{0}}}(\eps z_{k_{0}})$,
$B_{\theta \tilde \lambda_{k_{1}} \aeps \rho_{k_{1}}}(\eps z_{k_{1}})$ and $B_{\theta \tilde \lambda_{k_{2}} \aeps \rho_{k_{2}}}(\eps z_{k_{2}})$, respectively. 
The nested, dashed balls around $B_{k_0}$, $B_{k_1}$ and $B_{k_2}$ are the dilations by the factor $\theta^2$.} \label{fig.Step3.induction}
\end{figure}
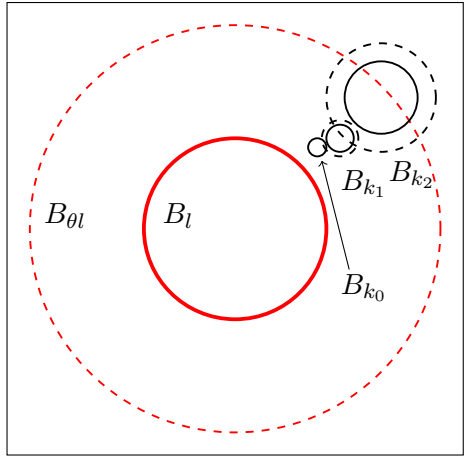

Indeed, for $z_k \in  J_k$ with $k < l-2$, we know that by \eqref{def.J_l} 
\begin{align}\label{iteration.boundary.2}
B_{\aeps \theta\tilde \lambda_k \rho_k}(\eps z_k) \not\subset E_{k+1}.
\end{align}
If also  \eqref{intersection.boundary} is true, then we obtain \eqref{chain.to.boundary} with 
$k_0= k_m= k$. Let us assume, instead, that \eqref{intersection.boundary} does not hold and thus, by \eqref{small.dont.intersect.big.5} that  
\begin{align}\label{iteration.boundary.4}
B_{\aeps\theta\tilde \lambda_k \rho_k}(\eps z_k) \subset B_{\aeps\theta \tilde \lambda_l  \rho_l}(\eps z_l) \stackrel{\eqref{ball.j.subset.E_k-2}}{\subset} E_{l-1}.
\end{align}
Then, by \eqref{iteration.boundary.2} and \eqref{def.E_l} there exists an index $k_1 \leq l-2$ and $z_{k_1} \in J_{k_1}$ such that
\begin{align}\label{iteration.boundary.3}
B_{\aeps\theta\tilde \lambda_k \rho_k}(\eps z_k) \cap B_{\aeps\theta \lambda_{k_1} \rho_{k_1}}(\eps z_{k_1}) \neq \emptyset.
\end{align}
Moreover, by \eqref{no.intersection.comparable}, we necessarily have $k_1 \geq k + 2$. We thus recovered the second line in \eqref{chain.to.boundary}. 
Since $z_{k_1} \in J_{k_1}$, we use again \eqref{def.J_l} to infer that 
$$
B_{\aeps\theta\tilde \lambda_{k_1} \rho_{k_1}}(\eps z_{k_1}) \not\subset E_{k_1+1}.
$$
Therefore, if $k_1= l-2$, we argue as in \eqref{iteration.boundary.1} and conclude that  \eqref{intersection.boundary} is satisfied with $z_k$ substituted by $z_{k_1}$. 
This and \eqref{iteration.boundary.3} yield  \eqref{chain.to.boundary} with $m=1$. Clearly, the same holds if $k_1 < l-2$ but \eqref{intersection.boundary} nonetheless satisfied
by $z_{k_1}$. Let us now assume, instead, that $z_{k_1}$ does not satisfy the first line in \eqref{chain.to.boundary}: By \eqref{iteration.boundary.3} and 
\eqref{iteration.boundary.4} this implies that
\begin{align}
B_{\ \aeps \theta\tilde \lambda_{k_1}\rho_{k_1}}(\eps z_{k_1}) \subset B_{\aeps\theta \tilde \lambda_l  \rho_l}(\eps z_l)  \subset  E_{l-1}.
\end{align}

We may now argue as for \eqref{iteration.boundary.2} above and obtain the existence of a new index $k_2 \geq k_1 +2$ satisfying \eqref{iteration.boundary.3} with $k$
and $k_1$ substituted by $k_1$ and $k_2$ respectively. By repeating the same argument above we iterate and conclude \eqref{chain.to.boundary} for a general $m$.
We remark that, since at each step $n$ the index $k_n$ increases of at least 2 this procedure necessarily stops whenever $k_n= l-2$. 
In other words, we obtain \eqref{chain.to.boundary} after at most  $1 + \lceil \frac {k_{max}} 2 \rceil$ iterations. We thus established \eqref{chain.to.boundary}.

\smallskip

Equipped with \eqref{chain.to.boundary} we finally argue \eqref{small.dont.intersect.big.3}: Since for all $0 \leq n \leq m \leq 1 + \lceil \frac {k_{max}} 2 \rceil$,
$1 \leq \lambda_{k_n} \leq \Lambda$ and
$k_0 \leq \cdots \leq k_m \leq l-2$, we estimate
\begin{align}
	\eps |z_l - z_{k}| &\geq \eps |z_l - z_{k_m}| - \sum_{n=1}^{m} \eps |z_{k_n} - z_{k_{n-1}}| \\
	& \stackrel{\eqref{chain.to.boundary}}{\geq} \theta \tilde \lambda_l \aeps \rho_l - (1 +2 m)\Lambda \aeps \rho_{k_m}\\
	& \stackrel{\theta >1}{\geq} \tilde \lambda_l \aeps \rho_l + (\theta - 1)\aeps \rho_l -  (\km + 4)  \Lambda \aeps \rho_{k_m}.
\end{align}
We now use the fact that since $z_l \in \tilde I_l$ and $z_{k_m} \in J_{k_m} \subset \tilde I_{k_m}$, we have by \eqref{def.I_k} and the assumptions on the indices $k_n$
that $\frac {\rho_{l}}{\rho_{k_m}} \geq \eps^{-\delta}$. 
From this inequality it follows that 
\begin{align}
 \eps |z_l - z_{k}| &\geq \tilde \lambda_l \aeps \rho_l + \bigl( (\theta -1)\eps^{-\delta} -  (\km + 4) \Lambda \bigr)\aeps \rho_{k_m}
\end{align}
and for $\eps$ small enough we bound
\begin{align}
 \eps |z_l - z_{k}| &\geq \tilde \lambda_l \aeps \rho_l + 2\lambda_k \aeps \rho_{k_m},
\end{align}
where $\lambda_k$ is the factor associated to $z_k$. We now observe that if $k_m=k_0=k$, then the above inequality contradicts \eqref{small.dont.intersect.big.5}. 
If, otherwise $k=k_0 \neq k_m$, then by construction we have $k_0 \leq k_m -2$ and thus by \eqref{def.I_k} that $\rho_k \leq \rho_{k_m}$. This and the above inequality 
contradict \eqref{small.dont.intersect.big.5} also in this case. This proves claim \eqref{small.dont.intersect.big.3} and establishes \eqref{small.dont.intersect.big}. 
The proof of  Lemma \ref{l.geometric.v2} and Lemma \ref{l.geometry} are complete.}
\end{proof}
\section{Proof of Lemma \ref{reduction.operator}} \label{sec:testfunction}

\begin{proof}[Proof of Lemma \ref{reduction.operator}] For a $\theta > 1$ fixed, let  $H^\eps = H^\eps_b \cup H^\eps_g$ and the sets $\bar H^\eps_b$, $D^\eps_b$ be as introduced in Lemma \ref{l.geometry} and  Lemma \ref{l.geometric.v2}. Throughout this proof, we use the notation $\lesssim$ for $\leq C$ with the constant 
depending on $d$, $\beta$, $\theta$.

\smallskip

{\bf Step 1. } We recall that the set $D^\eps_b$ is related to the partitioning of
$H^\eps = H^\eps_b \cup H^\eps_g$ and is such that $H^\eps_b \subset \bar H^\eps_b \subset D^\eps_b$.  We construct $R_\eps v$ by distinguishing between the parts of
domain $D$ containing ``small'' holes (i.e. $H^\eps_g$) and the ones containing the clusters of holes (i.e. $H^\eps_b$). We set
\begin{align}\label{definition.Rv}
R_\eps v := \begin{cases}
v^\eps_b \ \ \ \text{ \ in $D^\eps_b$}\\
v^\eps_g \ \ \ \ \text{ in $D \backslash D^\eps_b$,}
\end{cases}
\end{align}
where the functions $v^\eps_b$ and $v^\eps_g$ satisfy
\begin{align}\label{Rv.bad}
\begin{cases}
v^\eps_{b} = 0 \ \text{ in $H^\eps_b$}, \ \ \ \ v^\eps_{b}= v \ \text{ in $D \backslash D^\eps_b$,}\\
\nabla \cdot v^\eps_b = 0 \ \text{ in $D$},\\
 v^\eps_b \in H^1_0(D) \ \ \text{for $\eps$ small enough and } v^\eps_{b} \to v \ \ \text{ in $H^1_0(D)$,}\\
\| v^\eps_{b} \|_{C^0} \lesssim \| v \|_{C^0(\bar D)}.
\end{cases}
\end{align}
and
\begin{align}\label{Rv.good}
\begin{cases}
  v^\eps_g= v \ \text{ in $D^\eps_b$}, \ \ \ \ v^\eps_g = 0 \ \text{in $H^\eps_g$,}\\
\text{$v^\eps_g$ satisfies properties \ref{pro.0.inside.holes} - \ref{pro.capacity} with $H^\eps$ substituted by $H^\eps_g$.}
\end{cases}
\end{align}

In particular, this means
\begin{align}
	R_\eps v = v_b^\eps + v_g^\eps - v. \label{R_v.sum}
\end{align}

\smallskip

Before constructing the functions $v^\eps_g$ and $v^\eps_b$, we argue that  $R_\eps v$ defined in \eqref{definition.Rv} satisfies all the properties \ref{pro.0.inside.holes} - \ref{pro.capacity}  enumerated in the lemma. Properties \ref{pro.0.inside.holes} and \ref{pro.divergence.free} are immediately satisfied. We turn to properties \ref{pro.convergence.H^1} and \ref{pro.convergence.L^p}. By \eqref{R_v.sum}, we rewrite
\begin{align}
 \| R^\eps v - v \|_{L^p(\Rd)} = \| v^\eps_g -v\|_{L^p(\Rd)} + \| v^\eps_b -v\|_{L^p(D^\eps_b)}.
 \end{align}
The first term on the right-hand side vanishes almost surely in the limit thanks to the second line of \eqref{Rv.good} (property \ref{pro.convergence.L^p} for $v^\eps_g$). We bound the second term by using H\"older's inequality and the last estimate in \eqref{Rv.bad}:
\begin{align}
 \| v^\eps_b -v\|_{L^p(D^\eps_b)}^p \leq \|v - v^\eps_b \|_{C^0(D)}|D^\eps_b| \lesssim \| v \|_{C^0(D)}|D^\eps_b|.
\end{align}
Thanks to \eqref{D_b}, also this last line almost surely vanishes in the limit $\eps \downarrow 0^+$. Thus, almost surely the whole norm $\| R^\eps v - v \|_{L^p(\Rd)} \to 0$ when $\eps \downarrow 0^+$. This yields property \ref{pro.convergence.L^p} for $R_\eps v$. 
To establish Property \ref{pro.convergence.H^1} we use a similar argument to bound the $L^2$-norm of $\nabla (R^\eps v - v)$, this time using that by \eqref{Rv.bad} the gradient $\nabla (v^\eps_b -v)$ converges strongly to zero in $L^2(\Rd)$. Properties \ref{pro.0.inside.holes} - \ref{pro.convergence.L^p} for $R^\eps v$ are hence established.

\smallskip

It remains to argue property \ref{pro.capacity}:
Let $u_\eps\in H^1_0(D_\eps)$ be such that $u_\eps \rightharpoonup u$ in $H^1(D)$ and $\nabla \cdot u_\eps =0$ in $D$. By \eqref{R_v.sum}, we have
\begin{align}
 \int \nabla R^\eps v \cdot \nabla u_\eps = \int \nabla v^\eps_g \cdot \nabla u_\eps +  \int \nabla (v^\eps_b -v) \cdot \nabla u_\eps.
\end{align}
By \eqref{Rv.bad} and the assumptions on $u_\eps$, the second integral on the right-hand side almost surely converges to zero in the limit $\eps \downarrow 0^+$. We treat the first integral term by observing that $H^1_0(D^\eps) \subset H^1_0(D\backslash H^\eps_g)$ and applying \eqref{Rv.good} (i.e. property \ref{pro.capacity} for $v^\eps_g$). This implies property \ref{pro.capacity} for $R_\eps v$ and concludes the proof of the lemma provided we construct $v^\eps_g$ and $v^\eps_b$ as above.

\medskip

{\bf Step 2. Construction of $v^\eps_b$ satisfying \eqref{Rv.bad}. \,}

\smallskip

To construct $v^\eps_b$ on $D^\eps_b$, we exploit the construction of the covering $\bar H^\eps_b$ of Lemma \ref{l.geometric.v2}, as sketched in Section \ref{sec:ideas}. The main advantage in working with $\bar H^\eps_b$ instead of $H^\eps_b$ is related to the geometric properties satisfied by  $\bar H^\eps_b$ which allow to define $v^\eps_b$ via a finite number of iterated Stokes problems on rescaled annuli. 

\smallskip

Throughout this step, we skip the upper index $\eps$ and write $v_b$ instead of $v^\eps_{b}$.  Let $J= \bigcup_{i=-3}^{\km} J_i$ be the sub-collection of the centres of the
balls generating $\bar H^\eps_b$ in the proof of Lemma \ref{l.geometric.v2}. For each $z_j \in J$, we write 
\begin{align}\label{notation.simpler}
R_j^\eps:= \lambda_j^\eps \rho_j, \ \ B_j := B_{\aeps R_j}(\eps z_j),\\
B_{\theta,j} := B_{\aeps \theta R_j}(\eps z_j), \ \  A_j:= B_{\theta,j} \backslash B_j,
\end{align}
with $\lambda_j^\eps \in [1, \Lambda]$ the factors defined in Lemma \ref{l.geometric.v2}.

\smallskip

As a first step, we consider the set $J_{\km}$ and define the function $v^{0}$ on $D$ as 
\begin{align}\label{first.extension}
\begin{cases}
v^{0} =v \ \ \ \text{in $D \backslash \bigcup_{z_j \in J_{\km}}B_{\theta, j}$}\\
v^{0} = 0 \ \ \ \text{ in $B_j$ for all $z_j \in J_{\km}$}\\
v^{0} = v^{0}_j \ \ \ \text{ in $A_j$ for all $z_j \in J_{\km}$,}
\end{cases}
\end{align}
where each $v^{0}_j$ solves
\begin{align}\label{reduction.1}
\begin{cases}
-\Delta v^{0}_j+ \nabla  p_j^{0} = -\Delta v \ \ \ &\text{in $A_j$}\\
\nabla \cdot v^{0}_j =0 \ \ \ &\text{in $A_j$}\\
 v^{0}_j  = 0 \ \ \ &\text{on $\partial B_j$}\\
  v^{0}_j= v \ \ \ &\text{on $\partial B_{\theta,j}$.}
\end{cases}
\end{align}
This is well-defined since $\dv v = 0$. In particular, each function $v^{0}_j -v$ solves the first problem in \eqref{P.annulus} in $A_i$, 
and we apply to it the estimates \eqref{estimate.stokes.annulus} with the choice $R=\theta$ and after a rescaling by $\aeps R_j$ and a translation of
$\eps z_j$. This yields
\begin{align}
&\| \nabla v^{0}_j \|_{L^2(A_j)}^2 \lesssim \biggl( \| \nabla v \|_{L^2(B_{\theta,j})}^2 + \frac{1}{\bigl(\aeps R_j\bigr)^2} \| v \|_{L^2(B_{\theta,j})}^2\biggr),\\
&\| v^{0}_j \|_{C^0(\overline{B_{\theta,j}})} \lesssim \| v \|_{C^0(\overline{B_{\theta,j}})}.
\end{align}
We now use the definition \eqref{notation.simpler} of $R_j$ to obtain
\begin{equation} \label{first.estimate}
\begin{aligned}
&\| \nabla v^{0}_j\|_{L^2(A_j)}^2 \lesssim \bigl( \| \nabla v \|_{L^2(B_{\theta,j} )}^2 +\eps^d \lambda_j \rho_j^{d-2} \| v \|_{L^\infty}^2 \bigr),\\
&\| v^{0}_j \|_{C^0(\overline{B_{\theta,j}})} \lesssim\| v \|_{C^0(\overline{B_{\theta,j}})}.
\end{aligned}
\end{equation}
Note that thanks to \eqref{similar.size.apart} of Lemma \ref{l.geometric.v2}, we have that $B_{\theta,i} \cap B_{\theta,j} = \emptyset$ for all $z_i\neq z_j \in J_{\km}$ and $\lambda_i \leq \Lambda$ for all $z_i \in J$. Thus,  this also implies by \eqref{first.extension} that
\begin{equation} \label{iteration.estimate.0}
\begin{aligned}
&\| \nabla v^{0} \|_{L^2(D)}^2 \lesssim \| \nabla v \|_{L^2(D)}^2 + \eps^d \sum_{z_j \in J_{\km}} \rho_j^{d-2}  \| v \|_{L^\infty(D)}^2,\\
&\| v^{0} \|_{C^0(D)} \lesssim \| v \|_{C^0({D})}.
\end{aligned}
\end{equation}
Furthermore, since $v^0 - v$ is supported only in the balls $B_{\theta,j}$,  the triangle inequality and \eqref{first.estimate} imply also that
\begin{align}\label{strong.convergence.0}
\| \nabla (v^{0} -v) \|_{L^2(D)}^2 \lesssim \sum_{z_j \in J_{\km}}\| \nabla v \|_{L^2(B_{\theta,j})}^2 +
\eps^d \sum_{z_j \in J_{\km}} \rho_j^{d-2}  \| v \|_{L^\infty(D)}^2.
\end{align}

We observe also that, by using again the fact that by Lemma \ref{l.geometric.v2} all the balls $B_j$ are disjoint, the function $v^{0}$ vanishes on
 \begin{align}\label{step.1}
\bigcup_{z_j \in J_{\km}} B_j \stackrel{\eqref{inclusion.step.by.step}}{\supseteq} \bigcup_{z_j \in \I_{\km}} B_{\aeps \rho_j}(\eps z_j).
\end{align}

\medskip

We now proceed iteratively and for $1 \leq i \leq \km + 3$ we consider the subsets $J_{\km- i} \subset J$.
For each $i$ in the range above, let $v^{i}$ be defined as in \eqref{first.extension} and \eqref{reduction.1}, with $v^{i-1}$ instead of $v$ and the domains $B_j$ and $A_j$ generated by the
elements $z_j \in J_{\km-i}$.  We now argue that at each step $i$ we have 
 \begin{equation}
 \label{iteration.estimate}
\begin{aligned}
&\| \nabla v^{i} \|_{L^2(D)}^2 \lesssim \| \nabla v \|_{L^2(D)}^2 + \eps^d \sum_{z_j \in \cup_{k=0}^{i}J_{\km-k} } \rho_j^{d-2}  \| v \|_{L^\infty(D)}^2,\\
&\| v^{i} \|_{C^0(D)} \lesssim \| v \|_{C^0({D})},
\end{aligned}
\end{equation}
and 
\begin{align}\label{vanishing.set}
v^{i} = 0 \ \ \ \ \text{ in }  \bigcup_{z_j \in \bigcup_{k=0}^{i} \I_{\km-k}} B_{\aeps \rho_j}(\eps z_j).
\end{align}
Moreover,
\begin{equation}
\label{strong.convergence}
\begin{aligned}
& v^{i} - v = 0  \ \ \ \ \text{ in } D \backslash \left(   \bigcup_{z_j \in \cup_{k=0}^i J_{\km-k}} B_{\theta,j} \right),\\
&\| \nabla (v^{i} -v) \|_{L^2(D)}^2 \lesssim  \sum_{z_j \in \cup_{k=0}^{i}J_{\km-k} } \!\!\!\!\! \!\!\!\Bigl( \| \nabla v \|_{L^2(B_{\theta,j})}^2 + \eps^d \rho_j^{d-2}  \| v \|_{L^\infty(D)}^2\Bigr).
\end{aligned}
\end{equation}
We prove the previous estimates by induction over $0 \leq i \leq \km + 3$.

\smallskip

It is easy to prove the estimates in \eqref{iteration.estimate} by induction:  For $i=0$, \eqref{iteration.estimate.0} is exactly \eqref{iteration.estimate}. 
We now observe that at each step $i$ we may argue as for $v^0$ and obtain \eqref{iteration.estimate.0} 
with $v^0$, $v$ and $J_{\km}$ substituted by $v^{i}$, $v^{i-1}$ and $J_{\km-i} $, respectively. 
Therefore, if we now assume \eqref{iteration.estimate} holds at step $i-1$, we only need to
combine the analogue of \eqref{iteration.estimate.0} for $v^{i}$ with \eqref{iteration.estimate} for $v^{i-1}$. 

\smallskip

We now consider \eqref{vanishing.set}: For $i=0$, this is implied immediately by \eqref{step.1}. Let us now assume that \eqref{vanishing.set} holds for $i-1$. By definition 
of $v^i$ (cf. \eqref{reduction.1}),  the function vanishes on 
$$
\bigcup_{z_j \in J_{\km-i}} B_j \stackrel{\eqref{inclusion.step.by.step}}{\supseteq} \bigcup_{z_j \in \I_{\km-i}} B_{\aeps \rho_j}(\eps z_j)
$$
and equals $v^{i-1}$ on $D \backslash \bigcup_{z_j \in J_{\km -i}} B_{\theta,j}$. By the induction hypothesis \eqref{vanishing.set} for  ${i-1}$, \eqref{vanishing.set} for $i$ follows provided 
\begin{align}
\left( \bigcup_{z_j \in J_{\km-i} }B_{\theta,j} \right) \cap \left( \bigcup_{z_j \in \cup_{k=0}^{i-1} \I_{\km-k}} B_{\aeps \rho_j}(\eps z_j) \right)= \emptyset.
\end{align}
By recalling the definitions of the balls $B_{\theta,j}$, this identity is a consequence of property \eqref{small.dont.intersect.big} of Lemma \ref{l.geometric.v2}. We established \eqref{vanishing.set} and 
\eqref{iteration.estimate} for each $0 \leq i \leq  \km +3$.

\smallskip

Finally, we turn to the claims in \eqref{strong.convergence}: For $i=0$, both lines of \eqref{strong.convergence} hold by construction and \eqref{strong.convergence.0}, respectively. If we now assume that \eqref{strong.convergence} is true for $i-1$, 
then $v^{i}$ is by construction equal to $v^{i-1}$ outside the set 
$$
\bigcup_{z_j \in J_{\km-i}} B_{\theta,j}.
$$
It now suffices to apply the induction hypothesis for $v^{i-1}$ to conclude the first statement in \eqref{strong.convergence}. In addition, by the triangle inequality
we estimate
\begin{align}
\|\nabla( v^{i} - v )\|_{L^2(D)}^2 &\leq \|\nabla (v^{i} - v^{i-1}) \|_{L^2(D)}^2 + \|\nabla( v^{i-i} - v )\|_{L^2(D)}^2.
\end{align}
We apply the induction hypothesis to the second term on the right-hand side above and get
\begin{align}\label{estimate.induction.1}
\|\nabla( v^{i} - v )\|_{L^2(D)}^2 
\leq \|\nabla (v^{i} - v^{i-1}) \|_{L^2(D)}^2 + \!\!\!\!\!\!\!\! \sum_{z_j \in \cup_{k=0}^{i-1}J_{\km-k} } \!\!\!\!\!\!\!\!\Bigr( \| \nabla v \|_{L^2(B_{\theta,j})}^2 + \eps^d \rho_j^{d-2}  \| v \|_{L^\infty(D)}^2\Bigl).
\end{align}
We now use the analogue of \eqref{first.estimate} with $v^0$ and $v$ substituted by $v^{i-1}$ and $v^i$ to infer that 
\begin{align}
 \|\nabla (v^{i} - v^{i-1}) \|_{L^2(D)}^2 \lesssim \sum_{z_j \in J_{\km-i}} \!\!\! \Bigl( \| \nabla v^{i-1} \|_{L^2(B_{\theta,j} )}^2 +\eps^d \lambda_j \rho_j^{d-2} \| v^{i-1} \|_{L^\infty(D)}^2 \Bigr),
 \end{align}
 and, by \eqref{iteration.estimate} for $v^{i-1}$, that
 \begin{align}
 \|\nabla (v^{i} - v^{i-1}) \|_{L^2(D)}^2 &\lesssim \sum_{z_j \in J_{\km-i}} \!\!\! \Bigl( \| \nabla v^{i-1} \|_{L^2(B_{\theta,j} )}^2 +\eps^d \lambda_j \rho_j^{d-2} \| v \|_{L^\infty(D)}^2 \Bigr)\\
 &\lesssim \sum_{z_j \in J_{\km-i}} \| \nabla (v^{i-1} -v) \|_{L^2(B_{\theta,j} )}^2 + \!\!\! \sum_{z_j \in J_{\km-i}} \!\!\! \Bigl( \| \nabla v \|_{L^2(B_{\theta,j} )}^2 +\eps^d \lambda_j \rho_j^{d-2} \| v \|_{L^\infty(D)}^2 \Bigr).
\end{align}
Since all $B_{\theta, j}$, $z_j \in J_{\km - i}$, are disjoint, this implies that
 \begin{align}
 \|\nabla (v^{i} - v^{i-1}) \|_{L^2(D)}^2 &\lesssim \| \nabla (v^{i-1} -v) \|_{L^2(D)}^2 + \!\!\! \sum_{z_j \in J_{\km-i}} \!\!\! \Bigl( \| \nabla v \|_{L^2(B_{\theta,j} )}^2 +\eps^d \lambda_j \rho_j^{d-2} \| v \|_{L^\infty(D)}^2 \Bigr).
\end{align}
We may apply  the induction hypothesis on $v^{i-1}$ again and combine the above estimate with \eqref{estimate.induction.1} to conclude \eqref{strong.convergence} for $v^i$. The proof of \eqref{strong.convergence} is complete.

\smallskip

Equipped with \eqref{iteration.estimate}, \eqref{vanishing.set} and \eqref{strong.convergence}, we finally set $v^\eps_b := v^{\km + 3}$ and show 
that this choice fulfils all the conditions in \eqref{Rv.bad}: The first and the second line in \eqref{Rv.bad} follow immediately by construction and the definition  
\eqref{D_b} of $D^\eps_b$. The second estimate in \eqref{iteration.estimate} with $i=\km + 3$ yields also the last inequality in \eqref{Rv.bad}. 
It thus only remain to show that, almost surely, $v^\eps_b \in H^1_0(D)$ for $\eps$ small enough and $v^\eps_b \to v $ in $H^1_0(D)$.

 To do this, we begin by showing that $\nabla (v^\eps_b -v) \to 0$ in $L^2(D)$: By   \eqref{strong.convergence} with $i= \km +3$ and the fact that  $v \in C^\infty_0(D)$, we indeed obtain
\begin{align}
\|\nabla(v^\eps_b -v) \|_{L^2(D)} &\lesssim \| v\|_{C^1(D)} \sum_{z_j \in J }\bigr( (\aeps \rho_j)^2 + 1 \bigl)\eps^d \rho_j^{d-2}.
\end{align}
We recall that the set $J$ depends on $\eps$, i.e. $J= J^\eps$. In addition, since $J \subset \I$ (cf. Lemma \ref{l.geometric.v2}) and $n^\eps = \Phi^\eps(D) \backslash I^\eps$, the limit in \eqref{good.set} of Lemma \ref{l.geometry} yields that almost surely
$\eps^d \#J^\eps \to 0$ when $\eps \downarrow 0^+$. This, together with  \eqref{bar.H^b}, \eqref{power.law} and the Strong Law of Large numbers (cf. Lemma \eqref{l.LLN.index.set} in the Appendix) implies that the right-hand side above almost surely vanishes in the limit
$\eps \downarrow 0^+$. Hence, we showed that  $\nabla (v^\eps_b -v) \to 0$ in $L^2(\Rd)$. By Poincar\'e's inequality, it now suffices to argue that almost surely and for $\eps$ small enough $v^\eps_b \in H^1_0(D)$ to infer that $v^\eps_b \to v$ in $H^1_0(D)$ and thus conclude the proof of \eqref{Rv.bad}.

\medskip

Let $K \Subset D$ be a compact set containing the support of $v$, and set $r= \operatorname{dist}(K, D) > 0$. We show that, almost surely, $v^\eps_b \in H^1_0(D)$ for all $\eps \leq \bar\eps$, with $\bar \eps= \bar\eps(r,\omega) >0$. To do so, we fix any realization $\omega \in \Omega$ (which is independent from $v$) for which we have \eqref{delta.max}, and resort to the construction of $v^\eps_b$ via the solutions $v^0, v^1 \cdots v^{\km +3}$ obtained by iterating \eqref{reduction.1}. We claim that for all $i=0, \cdots, \km +3$ we have 
\begin{align}\label{iteration.support}
\operatorname{supp}(v^i) =: K^\eps_i \subset D, \ \ \ \ \operatorname{dist}(K^\eps, D) \geq  r - 2(i+1)\theta \Lambda \eps^{2\delta d},
\end{align}
for all $\eps$ such that the right-hand side in the last inequality is positive. Since $v^\eps_b := v^{\km +3}$, we may choose $\bar \eps (r,\omega)$ such that $\eps^{2\delta d}\leq \frac{r}{4(\km + 4)\theta \Lambda}$ and use the above estimate to infer that $v^\eps_b$ is compactly supported in $D$ for all $\eps \leq \bar\eps(r,\omega)$.

\smallskip

We prove \eqref{iteration.support} iteratively and begin with $i=0$: By \eqref{reduction.1} and the assumption on the support of $v$, it follows that,
 if for $z_i \in J_{\km}$ the ball $B_{\theta, i}$ does not intersect the support $K$ of $v$, then $v^0 = v \equiv 0$ on $B_{\theta,i}$. 
This, together with property \eqref{similar.size.apart} of Lemma \ref{l.geometric.v2}, implies that
\begin{align}\label{support.1}
\operatorname{supp}(v^0) \subset K \bigcup_{z_i \in J_{\km}, \atop B_{\theta, i} \cap K \neq \emptyset}B_{\theta,i}.
\end{align}
By recalling that thanks to Lemma \ref{l.geometric.v2} each ball $ B_{\theta,j}$ has radius
$$
\theta \lambda_i\aeps \rho_i \leq \theta \Lambda \aeps\rho_i \stackrel{\eqref{delta.max}}{\leq}\theta \Lambda\eps^{2d\delta},
$$
we observe that \eqref{support.1} yields estimate \eqref{iteration.support} for $v^0$. Let us now assume \eqref{iteration.support} for $v^i$. Then, since $v^{i+1}$ solves \eqref{reduction.1} with boundary datum $v_i$, we may argue as above to infer that
\begin{align}
K^\eps_{i+1} \subset K^\eps_i \bigcup_{z_i \in J_{\km}, \atop B_{\theta, i} \cap K_i^\eps \neq \emptyset}B_{\theta,i}
\end{align}
and thus that
\begin{align}
\operatorname{dist}(K^\eps_{i+1}, D) \geq \operatorname{dist}(K^\eps_{i}, D) - 2\theta \Lambda \eps^{2d\delta} \stackrel{\eqref{iteration.support}}{\geq} r - 2(i+1)\theta \Lambda \eps^{2d\delta}.
\end{align}
This concludes the iterated estimate \eqref{iteration.support}, which completes the proof of this step.

\medskip

{\bf Step 3. Construction of $v^\eps_g$ satisfying \eqref{Rv.good}.} We now turn to the remaining set $D\backslash D^\eps_b$ and construct the vector field $v^\eps_g$
 in a way similar to \cite{AllaireARMA1990a}[Subsection 2.3.2] and \cite{Desvillettes2008}. 

\smallskip

For every $z_i \in n^\eps$, we write
\begin{align}\label{abbreviations}
a_{\eps,i}:= \aeps \rho_i, \ \ \ \ \ d_i:= \min \biggl\{ \operatorname{dist}(\eps z_i, D^\eps_b), \frac 1 2 \min_{z_j \in n^\eps, \atop z_j \neq z_i} \bigl( \eps | z_i - z_j| \bigr), \eps \biggr\}
\end{align}
and 
\begin{align}\label{notation.Allaire}
T_i = B_{a_{\eps,i}} (\eps z_i), \ \ B_i:= B_{\frac {d_i}{ 2}} (\eps z_i), \ \ B_{2,i}:= B_{d_i}(\eps z_i), \ \ C_i:= B_i \backslash T_i, \ \ D_i:= B_{2,i} \backslash B_i.
\end{align}
We remark that, since $z_i \in n^\eps$, Lemma \ref{l.geometry} implies that for $\delta > 0$ 
\begin{align}\label{bound.aeps}
a_{\eps,i} \leq \eps^{1+2 \delta}, \ \ \ \ d_i \geq \eps^{1+\delta},
\end{align}
and that all the balls $B_{2,i}$ are pairwise disjoint. 

\smallskip

For each $z_i \in n^\eps$, we define the function $v^\eps_g$ in $B_{2,i}$ in the following way:
\begin{align}\label{problem.Ci}
\begin{cases}
v^\eps_{g}= 0 \ \ \ &\text{ in $T_i$}\\
v^{\eps}_g= v -\tilde v^\eps_i \ \ \ &\text{ in $C_i$,}
\end{cases}
\end{align} 
where $\tilde v^\eps_i$ solves
\begin{align}\label{eq.Ci}
\begin{cases}
-\Delta \tilde v^\eps_i + \nabla \pi^\eps_i = 0 \ \ \ &\text{ in $\Rd \backslash T_i$}\\
\nabla \cdot \tilde v^\eps_i = 0 \ \ \ &\text{in $\Rd \backslash B_1$}\\
\tilde v^\eps_i = v  \ \ \ &\text{on $\partial T_i$}\\
\tilde v^\eps_i \to 0 \ \ \ &\text{ for $|x| \to +\infty$.}
\end{cases}
\end{align}
Finally, we require that on $D_i$, $v^\eps_g$ solves
\begin{align}\label{eq.Di}
 \begin{cases}
 -\Delta v^\eps_g + \nabla q^\eps_g = \Delta v \ \ &\text{in $D_i$}\\
 \nabla \cdot v^\eps_g = 0 \ \ &\text{in $D_i$}\\
 v^\eps_g= v  \ \ &\text{on $\partial B_{2,i}$}\\
 v^{\eps}_g= v - \tilde v^\eps_i   \ \ &\text{on $\partial B_i$,}
 \end{cases}
\end{align}
and we then extend $v^\eps_g$ by $v$ on $\Rd \backslash \bigcup_{z_i \in n^\eps} B_{2,i}$. By Lemma \ref{l.geometry} and the definition \eqref{abbreviations} of $d_i$, we have that $D^\eps_b \subset \Rd \backslash \bigcup_{z_i \in n^\eps} B_{2,i}$. Therefore, this definition of $v^\eps_g$ satisfies the first line of \eqref{Rv.good} and property \ref{pro.0.inside.holes} with $H^\eps$ substituted by $H^\eps_g$. It is immediate that by construction $\nabla \cdot v^\eps_g=0$ in $D$, i.e. $v^\eps_g$ satisfies also  property \ref{pro.divergence.free}.

\smallskip

We observe that by uniqueness of the solution to \eqref{eq.Ci}, we may rescale the domains $C_i$ and rewrite
\begin{align}\label{Ci}
v^{\eps}_g = v - \phi_\infty^{\eps,i} \bigl(\frac{\cdot - \eps z_i }{a_{\eps,i}}\bigr) \ \ \ \text{ in $C_i$,}
\end{align}
with $\phi^{\eps,i}_\infty$ solving the second system in \eqref{P.annulus} in the annulus $\Rd \backslash B_1$ and with boundary datum $\psi(x) =v \bigl(a_{i,\eps}x - \eps z_i \bigr)$. Similarly, by uniqueness of the solutions to \eqref{eq.Di} we may rescale the domains $D_i$ and write 
\begin{align}\label{Di}
v^\eps_g=  v - \phi_2^{\eps,i}( \frac{\cdot - \eps z_i }{d_i }) \ \ \ \text{ in $D_i$,}
\end{align}
with $\phi_2^{\eps,i}$ solving the first system in \eqref{P.annulus} in the annulus $B_2 \backslash B_1$ and with boundary datum $\psi(x) =\phi_\infty^{\eps,i} \bigl(\frac{d_i (x - \eps z_i )}{a_{\eps,i}}\bigr)$.

\smallskip

We now turn to properties \ref{pro.convergence.H^1} and \ref{pro.convergence.L^p} for $v^\eps_g$: We write
\begin{align}\label{split.norms.a}
 \| v^\eps_g - v \|_{L^p(\Rd)}^p& =\sum_{z_i \in n^\eps} \| v^\eps_g -v \|_{L^p(B_{2,i})}^p,\\ 
 \| \nabla( v^\eps_g -v) \|_{L^2(\Rd)}^2& = \sum_{z_i \in n^\eps}  \|\nabla( v^\eps_g -v )\|_{L^2(B_{2,i})}^2,
\end{align}
and, since $B_{2,i}= D_i \cup C_i \cup T_i$, we may further split each norm on the right hand side into the contributions on each set $D_i$, $C_i$ and $T_i$. We begin by focussing on the domains $D_i$: By \eqref{Di}, we  apply \eqref{estimate.stokes.annulus} to $\phi_2^{\eps,i}$ and infer that
\begin{align}\label{estimates.Di.2}
\|\nabla (v^\eps_g -v) \|_{L^2(D_i)}^2 &\lesssim  \|\nabla \tilde v^\eps_i \|_{L^2(D_i)}^2 + d_i^{-2} \| \tilde v^\eps_i\|_{L^2(D_i)}^2, \\
\| v^\eps_g - v \|_{C^0(D_i)} &\lesssim \|\tilde v^\eps_i \|_{C^0(\partial B_{2,i})}.
\end{align}
By using \eqref{Ci} and changing variables, we rewrite the second line above as
\begin{align}
\| v^\eps_g - v \|_{C^0(B_{2,i})} &\lesssim \| \phi_\infty^{\eps,i} \|_{C^0(\partial B_{d_i a_{i,\eps}^{-1}})}, 
\end{align}
and use \eqref{infty.behaviour} on $\phi_\infty^{\eps,i}$ to infer 
\begin{align}\label{estimates.Di}
\| v^\eps_g - v \|_{C^0(B_i)} &\lesssim \| v \|_{C^0} \Bigl(\frac{a_{i,\eps}}{d_i}\Bigr)^{d-2} \stackrel{\eqref{bound.aeps}}{\lesssim} \|v \|_{C^0}\eps^{\delta(d-2)}.
\end{align}
In particular,
\begin{align}
\label{estimates.Di.p}
	\|v_g^\eps - v\|_{L^p(D_i)}^p \lesssim a_{i,\eps}^d  \|v \|_{C^0}\eps^{\delta(d-2)} \lesssim \|v \|_{C^0} \eps^{d+\delta(d-2)}.
\end{align}
We now turn to the first inequality in \eqref{estimates.Di.2}, use \eqref{Ci} on the right-hand side, and change variables to estimate
\begin{equation} \label{estimates.Di.1}
\begin{aligned}
 \|\nabla (v^\eps_g -v) \|_{L^2(D_i)}^2 
 &\lesssim a_{\eps,i}^{d-2} \|\nabla \phi_\infty^{\eps,i}\|_{L^2(B_{d_i a_{i,\eps}^{-1}} \backslash B_{\frac 1 2 d_i a_{i,\eps}^{-1}})}^2 + a_{\eps,i}^{d}d_i^{-2} \| \phi_\infty^{\eps,i} \|_{L^2(B_{d_i a_{i,\eps}^{-1}} \backslash B_{\frac 1 2 d_i a_{i,\eps}^{-1}})}^2\\
 &\stackrel{\eqref{infty.behaviour.1}}{\lesssim}\|v \|^2_{C^1} a_{\eps,i}^{d-2} \Bigl( \frac{a_{\eps,i}}{d_i} \Bigr)^{d-2}
 \stackrel{\eqref{bound.aeps}}{\lesssim}\|v \|^2_{C^1} \eps^{d + \delta(d-2)} \rho_i^{d-2}.
\end{aligned}
\end{equation}

\smallskip

 We consider the sets $C_i$: We use the definition \eqref{Ci} for $v^\eps_g$ on $C_i$ and a change of variables to rewrite
 \begin{align}
\|\nabla (v^\eps_g -v) \|_{L^2(C_i)}^2 & = a_{\eps,i}^{d-2}\|\nabla  \phi^{\eps,i}_\infty \|_{L^2(B_{\frac 1 2 d_i a_{\eps,i}^{-1}} \backslash B_1)}^2. 
\end{align} 
Hence, using \eqref{stokes.infty} for $\phi^{\eps,i}_\infty$, we obtain
\begin{align}\label{estimate.Ci}
\|\nabla (v^\eps_g -v) \|_{L^2(C_i)}^2&\lesssim \|\nabla v \|_{L^2(B_{2 a_{\eps,i}}(\eps z_i)  \backslash T_i)}^2 + a_{\eps,i}^{-2} \| v\|_{L^2(B_{2 a_{\eps,i}}(\eps z_i)  \backslash T_i)}^2 \\
&\lesssim a_{\eps,i}^{d-2} \| v \|_{C^1}^2 = \eps^d \rho_i^{d-2} \|v\|_{C^1}^2.
\end{align}

Similarly,  by \eqref{Ci} and a change of variables, for each $2 \leq p < +\infty$ we have
\begin{align}
  \|  v^{\eps,i}_g - v\|_{L^p(C_i)}^p &= a_{\eps,i}^d \| \phi^{\eps,i}_\infty  \|_{L^p(B_{d_i a_{\eps,i}^{-1}}\backslash B_1)}^p,
\end{align}
and, thanks to the pointwise estimate \eqref{infty.behaviour} for $\phi^{\eps,i}_\infty$, we have that for all $p > \frac{d}{d-2}$
\begin{align}\label{Lp.Ci}
 \| v^\eps_g - v \|_{L^p(C_i)}^p \lesssim \|v\|_{C^0}^p a_{\eps,i}^d \stackrel{\eqref{bound.aeps}}{\lesssim}\| v \|_{C^0}^p\eps^{2+4 \delta} \eps^d\rho_i^{d-2}.
\end{align}

\smallskip

We finally turn to $T_i$, on which we easily bound
\begin{align}\label{estimates.Ti}
\|\nabla( v^\eps_g - v) \|_{L^2(T_i)}^2= \|\nabla v \|_{L^2(T_i)}^2 \leq \| v \|_{C^1}^2 a_{\eps, i}^{d}
\stackrel{\eqref{bound.aeps}}{\lesssim}\| v \|_{C^1}^2\eps^{2(1+\delta)} \eps^d \rho_i^{d-2},\\
\|v^\eps_g -v \|_{L^p(T_i)}^p = \| v \|^p_{L^p(T_i)}  \stackrel{\eqref{bound.aeps}}{\lesssim} \|v\|_{C^0}^p \eps^{2(1+2\delta)} \rho^{d-2}.
\end{align}

\smallskip

By collecting all the estimates in \eqref{estimates.Di.p}, \eqref{estimates.Di.1}, \eqref{estimate.Ci}, \eqref{Lp.Ci} and  \eqref{estimates.Ti} we get
\begin{align}\label{L2.estimates.loc}
\| \nabla v^\eps_g - v\|^2_{L^2(B_{2,i})}  \lesssim \| v \|^2_{C^1} \eps^d\rho_i^{d-2},
\end{align}
and for all $p > \frac{d}{d-2}$
\begin{align}
\| v^\eps_g - v \|_{L^p(B_{2,i})}^p \lesssim \| v \|_{C^\infty} \eps^d \bigl(\eps^{2} \rho_i^{d-2} + \eps^{\delta p (d-2)}\bigr).
\end{align}
We insert these estimates in \eqref{split.norms.a} and apply \eqref{power.law} and the Strong Law of Large Numbers on the right-hand sides to conclude that almost surely
\begin{align}
\| \nabla v^\eps_g \|_{L^2(D)} \lesssim 1
\end{align}
and that $v^\eps_g \to v$ in $L^p(D)$ for $p > \frac{d}{d-2}$. Since $v, v^\eps_g$ are supported in the bounded domain $D$ for $\eps$ small enough, we conclude properties \ref{pro.convergence.H^1}
and \ref{pro.convergence.L^p} for $v^\eps_g$.
 
\smallskip

We finally turn to property \ref{pro.capacity}.  We use an argument very similar to the one for Lemma 3.1 of \cite{GHV1}. For any $N \in \N$ fixed and all $z_i \in n^\eps$, let us define
\begin{align}
n^\eps_N := \Bigl\{ z_i \in n^\eps \, \colon \, d_i \geq \frac \eps {N} \Bigr\},
\end{align}
where $Q \subset \Rd$ is a unit cube.
Moreover, let $\rr^N:=\{ \rho^N_i \}_{z_i \in n^\eps}$ be the truncated environment given by $\rho_i^N:= \rho_i \wedge N$ and let $H^{\eps,N}_g$ be the set of holes generated by $n^\eps_N$ with $\rr^N$. 
Let $v^{\eps,N}_g$ be the analogues of $v^\eps_g$ for $H^{\eps,N}_g$. We begin by showing that  $v^{\eps,N}_g$ satisfy property \ref{pro.capacity} on
$H^{\eps,N}_g$ with 
\begin{align}\label{measure.truncated}
\mu^N = C_d \langle (\rho^N)^{d-2}\rangle \langle \#(N_{\frac{2}{N}}(Q)) \rangle,
\end{align}
where $Q$ is a unit ball and $N_{\frac{2}{N}}$ is defined in Subsection \eqref{sec:notation}.

Before showing this, we argue how to conclude also property \ref{pro.capacity} for $v^\eps_g$: Let $u_\eps\in H^1_0(D_\eps)$ such that $u_\eps \rightharpoonup u$ in
$H^1(D)$. For each $N \in \N$ fixed we bound
\begin{align}
 \limsup_{\eps \downarrow 0^+}&\biggl| \int \nabla v^\eps_g \cdot \nabla u_\eps  - \biggl(\int \nabla v \cdot \nabla u + \int v \cdot \mu u\biggr)\biggr|\\
 & \leq
 \limsup_{\eps \downarrow 0^+} \biggl|\int \nabla  v^{\eps,N}_g \cdot \nabla u_\eps  - \biggl(\int \nabla v \cdot \nabla u + \int v \cdot \mu u\biggr)\biggr|
  +  \limsup_{\eps \downarrow 0^+}\biggl| \int \nabla (v^\eps_g- v^{\eps,N}_g) \cdot \nabla u_\eps\biggr|.
\end{align}
 Since $H^{\eps,N}_g \subset H^\eps_g$, property \ref{pro.capacity} for $v^{\eps,N}_g$ yields
\begin{align}\label{estimates.H3}
 \limsup_{\eps \downarrow 0^+}\biggl|& \int \nabla v^\eps_g \cdot \nabla u_\eps  - \biggl(\int \nabla v \cdot \nabla u + \int v \cdot \mu u\biggr)\biggr|\\
 & \leq
 \biggl|\int v \cdot (\mu - \mu^N)  u\biggr| +  \limsup_{\eps \downarrow 0^+}\biggl| \int \nabla (v^\eps_g- v^{\eps,N}_g) \cdot \nabla u_\eps\biggr|.
\end{align}
We now appeal to the explicit construction of the functions $v^\eps_g, v^{\eps,N}_g$ to observe that 
\begin{align}
\supp( v^\eps_g - v^{\eps,N}_g) &\subset \bigcup_{z_i \in n^\eps_N,\atop \rho_i \geq N} B_{2,i} \cup \bigcup_{z_i \in n^\eps \backslash n^\eps_N} B_{2,i},\\
 v^\eps_g - v^{\eps,N}_g &= v^\eps_g \ \ \ \ \text{ in $ \bigcup_{z_i \in n^\eps \backslash n^\eps_N} B_{2,i}$.}
\end{align}
Therefore, 
\begin{align}
\|\nabla(v^\eps_g- v^{\eps,N}_g)\|_{L^2(D)}^2 
 &\lesssim  \sum_{z_i \in n^\eps_N,\atop \rho_i \geq N}
 \| \nabla (v^\eps_g- v^{\eps,N}_g)\|_{L^2(B_{2,i})}^2 +  \sum_{z_i \in n^\eps \backslash n^\eps_N} \| \nabla v^\eps_g\|_{L^2(B_{2,i})}^2.
\end{align}
We smuggle in the norms on the right-hand side the function $v$ and appeal to \eqref{L2.estimates.loc} for $v^\eps_g$ (and the analogue for $v^{\eps,N}_g$) to get that 
\begin{align}\label{estimate.difference.1}
 \|\nabla(v^\eps_g- v^{\eps,N}_g)\|_{L^2(D)}^2  \lesssim \|v\|_{C^1(D)}
 \eps^d\biggl(\sum_{z_i \in n^\eps}\rho_i^{d-2} \1_{\rho_i \geq N} + \sum_{z_i \in n^\eps \backslash n^\eps_N}(1+ \rho_i^{d-2}) \biggr) .
\end{align}
Assumption \eqref{power.law} and the Strong Law of the Large Numbers yield that almost surely
\begin{align}
\sum_{z_i \in n^\eps}\rho_i^{d-2} \1_{\rho_i \geq N} \to \langle \rho \1_{\rho \geq N} \rangle.
\end{align}
Moreover, by \eqref{good.set} and \eqref{small.distance.bad} of Lemma \ref{l.geometry}, and \eqref{convergence.average} of Lemma \ref{SLLN.general.pp}, we have that almost surely 
\begin{align}\label{small.error.n.eps}
\lim_{N \uparrow +\infty} \lim_{\eps \downarrow 0^+}\eps^d \#(n^\eps \backslash n^\eps_N) = 0.
\end{align}
This yields by Lemma \ref{l.LLN.index.set} that  
\begin{align}
\lim_{N \uparrow +\infty} \lim_{\eps \downarrow 0^+}  \|\nabla(v^\eps_g- v^{\eps,N}_g)\|_{L^2(D)} = 0.
\end{align}
Since $\nabla u_\eps$ is uniformly bounded in $L^2(D)$, we can insert this in \eqref{estimates.H3} to conclude
\begin{align}
 \limsup_{\eps \downarrow 0^+}\biggl|& \int \nabla v^\eps_g \cdot \nabla u_\eps - \biggl(\int \nabla v \cdot \nabla u + \int v \cdot \mu u\biggr)\biggr|
  \lesssim \limsup_{N \uparrow +\infty} \biggl|\int v \cdot (\mu - \mu^N) u\biggr|.
 \end{align}
By using again assumption \eqref{power.law} and  \eqref{small.error.n.eps} we infer that the right-hand side above vanishes almost surely and conclude property \ref{pro.capacity} for $v^\eps_g$ with $\mu$ as in Theorem \ref{t.main}. 

\smallskip

We now turn to property \ref{pro.capacity} for $v^{\eps,N}_g$. When no ambiguity occurs, we drop the upper index $N$. For every $u_\eps$ as above, we split the integral
\begin{align}
\int \nabla v^{\eps}_g \cdot \nabla u_\eps = \int \nabla v \cdot \nabla u_\eps - \int \nabla (v^{\eps}_g- v) \cdot \nabla u_\eps.
\end{align}
The first term converges to $\int \nabla v \cdot \nabla u$ by the assumption on the sequence $u_\eps$. To conclude property \ref{pro.capacity} it thus remains to argue that
\begin{align}\label{strange.term.1}
 \int \nabla (v^{\eps}_g- v) \cdot \nabla u_\eps \to \int v \cdot \mu^N  u.
\end{align}

\smallskip

To prove this, we recall the construction of $v^\eps_g$, and we split the integral into 
\begin{align}
\int \nabla (v^{\eps}_g- v) \cdot \nabla u_\eps &= \sum_{z_i \in n^\eps} \int_{C_i}\nabla( v^{\eps}_g - v) \cdot \nabla u_\eps + \sum_{z_i \in n^\eps} \int_{D_i}\nabla( v^{\eps}_g - v) \cdot \nabla u_\eps.
\end{align}
Note that the integral on each $T_i$ vanishes by the assumption $u_\eps \in H^1_0(D^\eps)$. We first focus on the second sum on the right-hand side above and use Cauchy-Schwarz and \eqref{estimates.Di.1} to bound
\begin{align}
\sum_{z_i \in n^\eps} \int_{D_i}\nabla( v^{\eps}_g - v) \cdot \nabla u_\eps \lesssim \| \nabla u_\eps \|_{L^2(D)} \Bigl(\eps^{d +\delta(d-2)} \sum_{z_i \in n^\eps}\rho_i^{d-2}\Bigr)^{\frac 1 2} \|v\|_{C^\infty}.
\end{align}
By the assumption on the weak convergence for the sequence $\nabla u_\eps$ and the Strong Law of Large Numbers, the right-hand side almost surely vanishes in the limit $\eps \downarrow 0^+$. Thus,
\begin{align}\label{p5.1}
\int \nabla (v^{\eps}_g- v) \cdot \nabla u_\eps &=  \sum_{z_i \in n^\eps} \int_{C_i}\nabla( v^{\eps}_g - v) \cdot \nabla u_\eps + o(1).
\end{align}

\smallskip

We turn to the remaining term above: For each $z_i \in n^\eps$, let $(\tilde \phi^{\eps,i}_\infty, \tilde \pi^{\eps,i}_\infty)$ solve the Stokes problem \eqref{P.annulus} in the exterior domain $\Rd \backslash B_1$ and with constant boundary datum $v(\eps z_i)$.  We define
\begin{align}\label{constant.corrector}
\bar \phi_\infty = \tilde \phi_\infty( \frac{\cdot - \eps z_i }{a_{\eps,i}} ), \ \ \ \bar \pi_\infty := a_{\eps,i}^{-1} \tilde \pi_\infty(\frac{\cdot - \eps z_i }{ a_{\eps,i}}),
\end{align} 
and smuggle these functions in each one of the integrals over $C_i$. This yields
\begin{align}\label{p5.2}
\sum_{z_i \in n^\eps} \int_{C_i}\nabla( v^{\eps}_g - v) \cdot \nabla u_\eps = \sum_{z_i \in n^\eps} \int_{C_i}\nabla( v_g^\eps -v - \bar \phi_\infty^{\eps,i}) \cdot \nabla u_\eps +  \sum_{z_i \in n^\eps} \int_{C_i}\nabla(\bar \phi_\infty^{\eps,i}) \cdot \nabla u_\eps.
\end{align}

\smallskip

We claim that the first integral on the right-hand side vanishes in the limit $\eps \downarrow 0^+$: By \eqref{Ci} and \eqref{constant.corrector}, each difference $v_g^\eps -v - \bar \phi_\infty^{\eps,i}$  solves the second system in \eqref{P.annulus} in $\Rd \backslash T_i$ with boundary datum $\psi= v - v(\eps z_i)$. 
Therefore, by the first inequality in \eqref{stokes.infty},
 \begin{align}
\|\nabla (v_g^\eps -v - \bar \phi_\infty^{\eps,i}) \|_{L^2(C_i)}^2 & \lesssim \|\nabla v \|_{L^2(B_{2 a_{\eps,i}}(\eps z_i)  \backslash T_i)}^2 + a_{\eps,i}^{-2} \| v - v(\eps z_i)\|_{L^2(B_{2 a_{\eps,i}}(\eps z_i)  \backslash T_i)}^2.
\end{align} 

As the vector field $v$ is smooth, we use a Lipschitz estimate on the last term, and conclude that
\begin{align}
\|\nabla( v^{\eps}_g - v - \bar \phi_\infty^{\eps,i})\|_{L^2(C_i)}^2 \lesssim \| v\|^2_{C^1} a_{\eps,i}^{d}  \stackrel{\eqref{bound.aeps}}{\lesssim}  \| v\|^2_{C^1}\eps^{2 + 4\delta} \eps^d \rho_i^{d-2}.
\end{align}
By Cauchy-Schwarz inequality and this last estimate we find
\begin{align}
 \sum_{z_i \in n^\eps} \int_{C_i}\nabla( v^{\eps}_g - v - \bar \phi_\infty^{\eps,i}) \cdot \nabla u_\eps \leq \| \nabla u_\eps \|_{L^2} \Bigl(\eps^{2 + d} \sum_{z_i \in n^\eps } \rho_i^{d-2} \Bigr)^{\frac 1 2},
\end{align}
and use the the Strong Law of Large Numbers to conclude that almost surely the above right-hand side vanishes. This, together with \eqref{p5.2} and \eqref{p5.1}, yields
\begin{align}\label{p5.3}
\int \nabla (v^{\eps}_g- v) \cdot \nabla u_\eps =  \sum_{z_i \in n^\eps} \int_{C_i}\nabla \bar \phi_\infty^{\eps,i} \cdot \nabla u_\eps + o(1).
\end{align}

\smallskip

We now integrate the first integral on the right-hand side above by parts and, since $u_\eps$ vanishes in $T_i$, we obtain
\begin{align}
 \int_{C_i}\nabla \bar \phi_\infty^{\eps,i} \cdot \nabla u_\eps &= - \sum_{z_i \in n^\eps} \int_{C_i}\Delta \bar \phi_\infty^{\eps,i}  u_\eps 
  + \int_{\partial B_i} \partial_\nu \bar \phi_\infty^{\eps,i} u_\eps,
\end{align}
where $\nu$ denotes the outer unit normal.
By using \eqref{constant.corrector}, the equation \eqref{P.annulus} for $(\bar \phi_\infty^{\eps,i}, \bar \pi_\infty^{\eps,i})$ and the fact that $\nabla \cdot u_\eps=0$ in $D$, we obtain
\begin{align}
 \int_{C_i}\nabla \bar \phi_\infty^{\eps,i} \cdot \nabla u_\eps &=  \sum_{z_i \in n^\eps} \int_{\partial B_i}( \partial_nu \bar \phi_\infty^{\eps,i} - \bar \pi^{\eps,i} \nu ) \cdot u_\eps.
\end{align}
 By wrapping this up with \eqref{p5.3}, we conclude that to show \eqref{strange.term.1} it suffices to prove that
\begin{align}\label{strange.term.2}
 \sum_{z_i \in n^\eps} \int_{\partial B_i}(\partial_\nu \bar \phi_\infty^{\eps,i} - \bar \pi^{\eps,i} \nu ) \cdot u_\eps \to \int v \cdot \mu^N u.
\end{align}
We establish \eqref{strange.term.2} as in \cite{AllaireARMA1990a}: We remark, indeed, that by the uniqueness of the solutions in \eqref{P.annulus}, for each $z_i \in n^\eps$, we have
$$
\bar \phi_\infty^{\eps,i} = \sum_{k=1}^d v_k(\eps z_i) w^\eps_{k}, \ \ \ \bar \pi^{\eps,i} = \sum_{k=1}^d v_k(\eps z_i) q^\eps_k,
$$
with $(w^\eps_k, q^\eps_k)$ the analogues of the oscillating test functions constructed in \cite{AllaireARMA1990a}[Proposition 2.1.4]. We remark that the only difference is that in this setting, the scales $a_{\eps,i}$ (i.e. the size of the holes $T_i$) depend on the index $z_i$ and are not constant but bounded by $N$ (we recall that we are considering the truncated environment $\rr^N$).  Therefore, by arguing as in Lemma 2.3.7 of \cite{AllaireARMA1990a} we use Lemma 2.3.5 of \cite{AllaireARMA1990a} and linearity to rewrite
\begin{align}\label{strange.term.3}
 \sum_{z_i \in n^\eps} \int_{\partial B_i}( \partial_\nu \bar \phi_\infty^{\eps,i} - \bar \pi^{\eps,i} \nu )u_\eps  = (\mu^N_\eps, u_\eps)_{H^{-1}, H^1_0} + r_\eps,
\end{align}
with 
\begin{align}
\mu^N_\eps =  \frac{C_d}{|B_1|} \sum_{z_i \in n^\eps} v(\eps z_i) (\rho_i^N)^{d-2} \frac{(2\eps)^d}{d_i^d} \1_{B_i}, \ \ \ r_\eps \to 0 \ \ \ \text{ in $H^{-1}(D)$.} 
\end{align}
Since $v \in C^\infty_0(D)$ and the radii $\rho^N_i$ are uniformly bounded, we can also replace $\mu^N_\eps$ by
\begin{align}
\tilde \mu^N_\eps = \frac{C_d}{|B_1|} \sum_{z_i \in n^\eps}(\rho_i^N)^{d-2} \frac{(2\eps)^d}{d_i^d} \1_{B_i} v.
\end{align}
To establish  \eqref{strange.term.2}, it remains to argue as in \cite{GHV1}[Lemma 3.1, case (b)] (see from formula (4.75) on) and appeal to Lemma \ref{l.w.independent} in \cite{GHV1}. This yields property \ref{pro.capacity} for $v^\eps_g$ and thus completes the proof of this step and of the whole lemma.
\end{proof}

\section{Probabilistic results}\label{s.probability}
The aim of this section is to give some probabilistic results on the random set $H^\eps$, in terms of the size of the clusters
generated by overlapping balls of comparable size; these results are used in Section \ref{geometry} to obtain a good covering for $H^\eps$ 
and to estimate its size.

\smallskip

We  introduce the following notation: For $\alpha \geq 1$, let
\begin{align}
H^\eps_\alpha = \bigcup_{z_i \in \Phi^\eps(D)} B_{\aeps \alpha \rho_i}(\eps z_i).
\end{align}
For a step-size $\delta > 0$, we partition the (random) collection of points $\Phi^\eps(D)$ in terms of the order of magnitude of the associated radii: We define 
\begin{align}\label{def.I_k_delta}
&I^{\eps}_{k, \delta} := \{  z_i \in \Phi^\eps(D) \ \colon \ \eps^{1-\delta k } < \aeps \rho_i \leq \eps^{1-\delta(k+1)}\} \ \ \ \text{for $k \geq -2$},\\
&I^{\eps}_{-3,\delta} := \{  z_i \in \Phi^\eps(D) \ \colon \  \aeps \rho_i \leq \eps^{1+2\delta} \},
\end{align}
and for every $k \geq -2$ also 
\begin{align}\label{Psi.delta.k}
\Psi^{k, \eps}_\delta= I_k^{\eps} \cup I_{k-1}^{\eps} \subset \Phi^\eps(D).
\end{align}
Each collection $\Psi^{k,\eps}_\delta$ thus generates the set 
\begin{align}\label{comparable.components}
H^{\delta,\eps}_{k,\alpha} := \bigcup_{z_i \in \Psi^{k,\eps}_\delta} B_{\aeps \alpha \rho_i}(\eps z_i) \subset H^\eps_\alpha
\end{align}
which is made of balls having radii which differ by at most two orders $\delta$ of magnitude.

\smallskip

\begin{lem}\label{l.chains}
Let $\alpha \geq 1$ and $0< \delta < \frac{\beta}{2d}$ be fixed. Then, there exists $M(d,\beta),  \km(\beta,d) \in \N$ such that for almost every
$\omega \in \Omega$  and every $\eps \leq \eps_0(\omega)$
\begin{itemize}
 \item[(I)] For every $k > \km$ we have
 \begin{align}
	\label{existence.k_max}
	I_{\eps,\delta}^k = \emptyset;
\end{align}

\smallskip

 \item[(II)] For every $-2 \leq k \leq \km$, each connected component of $H^\eps_{k,\alpha}$ defined in \eqref{comparable.components} is made of at most $M$ balls.
\end{itemize}
\end{lem}

\smallskip

\begin{proof}[Proof of Lemma \ref{l.chains}]
We begin with (I) and observe that assumption \eqref{power.law} and Chebyshev's inequality imply that for a constant $C< +\infty$
\begin{align}
	\label{probability.rho}
	\langle \rho^{d-2+\beta} \rangle \leq C, \qquad \P(\rho \geq r) \leq C r^{-(d-2 + \beta)}.
\end{align}
In addition, as already argued in Section 4 (see \eqref{delta.max}),\eqref{power.law} and the Strong Law of Large Numbers (see Lemma \ref{SLLN.general.pp}) imply that
for almost every $\omega \in \Omega$ and all  $\eps$ sufficiently small
\begin{align}
	\label{max.radius}
	\max_{z_i \in \Phi^\eps(D) }\aeps \rho_i \leq 2 \eps^{\frac{d}{d-2} - \frac{d}{d-2 + \beta}} \langle \rho^{d-2+\beta} \rangle^{\frac{1}{d-2+\beta}}.
\end{align}
Hence, for the same choice of $\omega$ and $\eps$ we have $I^{k} = \emptyset$ whenever $k > \km$ with
\begin{align}\label{k.times.delta}
\\
\eps^{1 - \delta (\km +1)} < \eps^{\frac{d}{d-2} - \frac{d}{d-2 + \beta}}, 
\end{align}
namely if
\begin{align}\label{k_max.bound}
1 - \delta(\km +1) <  \frac{d}{d-2} - \frac{d}{d-2 + \beta}.
\end{align}
We may thus choose the minimal $k_{max}$ satisfying the inequality above and conclude the proof for (II).

\medskip

We now turn to (II) and fix $-2 \leq k \leq \km$: For any $m \in \N$ we consider the event
\begin{align}
A^{\alpha, m}_{\eps, \delta, k}:= \{ \omega \, \colon \, &\text{There exist $m$ intersecting balls in $H^{\delta,\eps}_{k,\alpha}$} \}.
\end{align}
Then, (II) is equivalent to show that there exists an integer $M = M(\beta, d)\geq 2$ such that
\begin{align}\label{probability.chains}
\P\biggl( \bigcap_{\eps_0> 0} \bigcup_{\eps \leq \eps_0} \bigcup_{k \geq -2 }A_{\eps,\delta, k}^{\alpha, M} \biggr)= 0.
\end{align}
Furthermore, we begin by arguing that it suffices to prove that 
\begin{align}
	\label{reduction.eps_l}
	\P\biggl( \bigcap_{l_0 \in \N} \bigcup_{l \geq l_0} \bigcup_{k  \geq -2}A_{2^{-l},3\delta, k}^{\bar \alpha, M} \biggr)= 0,
\end{align}
i.e. statement \eqref{probability.chains} for the sequence $\eps_l = 2^{-l}$ and $\alpha, \delta$ substituted by $\bar \alpha= 2^{\frac{2}{d-2}} \alpha$ and $3\delta$.

\smallskip
	
Suppose, indeed, that \eqref{reduction.eps_l} holds: For any $\eps >0$, let $l\in \N$ be such that $\eps_{l+1} \leq \eps \leq \eps_{l}$. Then for every two 
 $z_i, z_{j} \in \Psi^{k,\delta, \eps}$ with $\rho_i \geq \rho_j$, definition \eqref{def.I_k_delta} yields that
\begin{align}
\rho_i - \rho_j \leq \rho_j (\frac{\rho_i}{\rho_j} -1) \leq \rho_j( \eps_{l+1}^{- 2 \delta} -1)\leq \rho_j \eps_{l+1}^{-3\delta}.
\end{align}
This implies that if $\rho_j \in I_{\tilde k-1}^{\eps_{l+1}, 3\delta}$ for some $\tilde k \in \Z$, then $\rho_i \in I_{\tilde k}^{\eps_{l+1}, 3\delta}$. This is equivalent 
to
\begin{align}\label{inclusion.1}
\Psi^{\delta, \eps}_k \subset  \Psi^{,3\delta, \eps_{l+1}}_{\tilde k}.
\end{align}
Equipped with this inclusion, we now show that
\begin{align}\label{inclusion.2}
A_{\eps,\delta,k}^{\alpha,m}  \subset A_{\eps_{l+1},3 \delta, \tilde k}^{\bar\alpha,m}.
\end{align}
To do so, let us assume that $z_i , z_j \in \Psi^{\delta, \eps}_k$ satisfy
	\begin{align}
		\label{collision.0}
		B_{ \alpha \aeps \rho_j}(\eps z_j) \cap B_{\alpha \aeps \rho_i}(\eps z_i) \neq \emptyset.
	\end{align}
Then,
\begin{align}
	\eps |z_i - z_j| \leq \alpha \aeps (\rho_i + \rho_j)
\end{align}
which yields
\begin{align}
	 |z_i - z_j| \leq \alpha \eps^{\frac 2 {d-2}}  (\rho_i + \rho_j) \leq \alpha \eps_l^{\frac 2 {d-2}}  (\rho_i + \rho_j) 
	 = 2^{\frac 2 {d-2}} \alpha \eps_{l+1}^{\frac 2 {d-2}}  (\rho_i + \rho_j).
\end{align}
This is equivalent to
	\begin{align}
		B_{ \bar \alpha \eps_{l+1}^{\frac{d}{d-2}}\rho_j}(\eps_{l+1} z_j) \cap B_{\bar \alpha \eps_{l+1}^{\frac{d}{d-2}}\rho_i}(\eps_{l+1} z_i) \neq \emptyset.
	\end{align}
Since the previous argument holds for any choice of two elements in $\Psi^{k, \delta, \eps}$, this and \eqref{inclusion.1} imply \eqref{inclusion.2}. This last 
statement allows also to conclude that for every $m \in \Z$
\begin{align}\label{monotonicity.1}
\bigcup_{k \geq -2} A_{\eps,\delta,k}^{\alpha,m} \subset \bigcup_{k \geq -2} A_{\eps_{l+1},2 \delta,k}^{\bar\alpha,m}.
\end{align}
This establishes that \eqref{reduction.eps_l} implies \eqref{probability.chains}.
 
\medskip
 
To conclude the proof of (II), it only remains to show \eqref{reduction.eps_l}: We begin by deriving a basic estimate for the probability of having a certain number 
of close points in a Poisson point process. We recall indeed that the centres $\Phi^\eps(D)$ are distributed according to a Poisson point process in $1 \eps D$ with intensity
 $\lambda$. We also recall that, for a general set $A \subset \Rd$ we denote by $N(A)$ the random variable providing the number of points of the process which are in $A$. 

\smallskip
 
For $0 < \eta <1$, let
\begin{align}
 \mathcal{Q}_\eta:= \bigl\{ [-\frac \eta 2, \frac \eta 2 ]^d + y \, | \, y \in (\eta \Z)^d \bigr\},
\end{align}
i.e. the set of cubes of length $\eta$ centered at the points of the lattice $(\eta \Z)^d$. Let $S_\eta$ be the set containing the edges of the cube $[0, \eta 2]^d$, i.e.
$$
S_\eta:= \{ z= (z_1,\dots,z_d)\in \R^d \, \colon \,  z_k \in \{0,\frac{\eta}{2}\} \ \ \text{for all $k=1, \cdots, d$} \}.
$$
Then, for any $x \in \R^d$ there always exists $z \in S_\eta$ and $B_{\frac \eta 2}(x) \subset Q$ for some $Q \in \mathcal Q_\eta + z$. Thus, if $\eta$ is chosen such that
$\lambda \eta^d \leq 1$, we use this geometric consideration to estimate
\begin{equation}
\begin{aligned}
	\P( \, \exists \, x \in \frac{1}{\eps} D \, \colon \,  N(B_{\frac{\eta}{2}}(x)) \geq m ) 
	&\lesssim \P( \, \exists \, Q \in \mathcal Q_\eta, z \in S_\eta  \, \colon \, (Q + z) \cap \frac{1}{\eps}  D \neq \emptyset, \ N( Q + z ) \geq m), 
\end{aligned}
\end{equation}
and the distribution for $N(A)$ to conclude that
\begin{equation}
	\label{points.squares.Poisson}
\begin{aligned}
	\P( \, \exists \, x \in \frac{1}{\eps} D \, \colon \,  N(B_{\frac{\eta}{2}}(x)) \geq m ) \lesssim \eps^{-d} \eta^{-d} e^{-\lambda \eta^d} \sum_{k=m}^{\infty}  \frac{(\lambda \eta^{d})^k}{k!} 
	\lesssim (\eta \eps)^{-d} (\lambda \eta^d)^{m}.
\end{aligned}
\end{equation}

\smallskip

Equipped with \eqref{points.squares.Poisson}, we estimate each $P(A_{\eps,k}^{\alpha,m})$: Let us assume that $z_i , z_j \in \Psi^{k,\delta, \eps}$ are such that
	\begin{align}
		B_{ \alpha \aeps \rho_j}(\eps z_j) \cap B_{\alpha \aeps \rho_i}(\eps z_i) \neq \emptyset.
	\end{align}
Then,
\begin{align}
	\eps |z_i - z_j| \leq \alpha \aeps (\rho_i + \rho_j) \leq 2 \alpha \eps^{1-\delta(k+1)} 
\end{align}
and thus by setting
\begin{align}\label{def.kappa.k}
\kappa_k =  -\delta(k+1),
\end{align}
we have
\begin{align}\label{inclusion.squares}
 |z_i - z_j| \leq  2 \alpha \eps^{\kappa_k}, \ \ \  A_{\eps,k}^{\alpha,m} \subset \{ \, \exists \, x \in \frac{1}{\eps} D \, \colon \, \# (\Psi^{k,\delta,\eps} \cap B_{m \alpha \eps^{{\kappa_k}}}(x)) \geq m \, \}.
\end{align}
We now want to estimate the event in the right-hand side above by appealing to \eqref{points.squares.Poisson} for each $\eps$ and $k$ fixed and with $\eta=\eta^\eps_k$ given 
by
\begin{align}\label{def.eta.k}
 \eta_k^\eps :=   m \alpha \eps^{{\kappa_k}}.
\end{align}
We observe indeed that by definition \eqref{def.I_k_delta}, for every $\eps$ the processes $\Psi^{k,\delta,\eps}$ are Poisson processes
on $\frac 1 \eps D$ with intensity given by
\begin{align}\label{bound.intensities}
 \lambda_k^\eps = \lambda \P( \, \eps^{-\frac{2}{d-2} - \delta(k-1)} \leq \rho \leq \eps^{-\frac{2}{d-2} - \delta(k+1)} \, )
 \stackrel{\eqref{probability.rho}}{\lesssim} \eps^{ (d-2+\beta)\Bigl(\frac{2}{d-2} + \delta(k-1)\Bigr)}
\end{align}
for any $k \geq -1$, and 
\begin{align}\label{bound.intensities.2}
 \lambda_{-2}^\eps = \lambda \P(\, \rho \leq \eps^{-\frac{2}{d-2} - \delta(-1)}\, ) \leq \lambda
\end{align}
for $k=2$. 

\medskip

We first argue that, provided that for every $k$ and $\eps$ small enough, there exists $\mu_k >0$ such that
\begin{align}\label{condition.small.intensity}
 \lambda^\eps_k(\eta^\eps_k)^d \leq \eps^{\mu_k}, 
\end{align}
then we conclude the proof of \eqref{reduction.eps_l}. Indeed, by the previous inequality we may apply \eqref{points.squares.Poisson} to the right-hand side of
\eqref{inclusion.squares} and bound by \eqref{def.eta.k} and \eqref{condition.small.intensity}
\begin{align}
	\P( A_{\eps,k}^{\alpha,m}) \lesssim \eps^{m\mu_k-d(1+{\kappa_k})}.
\end{align}
By choosing $m =M$, $M$ sufficiently large, we thus get
\begin{align}
	\P( A_{\eps,k}^{\alpha,m}) \lesssim \eps^{\mu_k}.
\end{align}
Since by (I) we only have to consider finitely many values of $k= -3, \cdots, \km$, $M$ can be chosen independently of $k$. Therefore, recalling that $\eps_l= 2^{-l}$ in 
\eqref{reduction.eps_l}, we use the previous estimate and assumption \eqref{condition.small.intensity} to infer
\begin{align}
	\sum_{l \in \N} \P\biggl(  \bigcup_{k  \geq -2 }A_{\eps_l,\delta,k}^{\alpha, M} \biggr) < \infty.
\end{align}
I thus remains to apply Borel-Cantelli's lemma to obtain \eqref{reduction.eps_l} and thus \eqref{probability.chains} as well as (II).

\medskip

To conclude the proof of the lemma, it thus remains to show \eqref{condition.small.intensity}. To do so, we recall the definitions \eqref{def.eta.k} and \eqref{def.kappa.k} 
of $\eta_k$ and $\kappa_k$ and we also set for every $-1 \leq k \leq \km$
\begin{align}\label{def.gamma.k}
 \gamma_k :=(d-2+\beta)\Bigl(\frac{2}{d-2} + \delta(k-1)\Bigr).
\end{align}
By \eqref{bound.intensities}, this definitions allows us to bound for each $\eps$
\begin{align}\label{bound.intensities.3}
 \lambda_k^\eps \leq \eps^{\gamma_k}.
\end{align}

\smallskip

We first show \eqref{condition.small.intensity} for $k=-2$: In this case, by \eqref{def.eta.k}, \eqref{def.kappa.k} and \eqref{bound.intensities.2}, we have
\begin{align}
	\lambda^\eps_{-2}(\eta^\eps_{-2}) \lesssim \eps^{d \delta}
\end{align}
and we may thus simply choose $\mu_{-2}= d \delta >0$. We now turn to the case $k > -2$: Again by \eqref{def.eta.k} and, this time, by  \eqref{bound.intensities.3} we have
\begin{align}
 \lambda^\eps_k (\eta^\eps_k)^d \lesssim \eps^{\gamma_k + d \kappa_k}.
\end{align}
Therefore we need
\begin{align}
\mu_k= \gamma_k + d\kappa_k \stackrel{\eqref{def.gamma.k}, \eqref{def.kappa.k}}{=}   \frac{2(d-2+\beta)}{d-2} - (2-\beta)\delta(k-1) - 2d \delta >0.
\end{align}
{ Since we assumed that $\beta \leq 1$, we may use \eqref{k_max.bound} on the second term in the right-hand side above and, after a short calculation, obtain that
\begin{align}
	\mu_k \geq 2 - (2-\beta) - 2d \delta \geq \beta - 2 d \delta .
\end{align}
Thanks to our assumption $\delta < \frac{\beta}{2d}$, we thus conclude that $\mu_k > 0$.} This establishes \eqref{condition.small.intensity} and completes the proof of the lemma.
\end{proof}

\appendix

\section{Proof of Remark \ref{N-S}}
\label{sec:Proof.remark}
The proof of the homogenization result in this case is analogous to the case of the Stokes equations, provided we prove the convergence of the non-linear term $u_\eps \nabla \cdot u_\eps$.

We recall the weak formulation of \eqref{NS.epsilon}. We define the space $ V_\eps := \{w \in H^1_0(D_\eps) \colon \dv w =0\}$ equipped with the norm
$\|\nabla \cdot\|_{L^2}$. Then, we call $u_\eps \in V$ a weak solution to \eqref{NS.epsilon} if
\begin{align}
	\mu \int \nabla u_\eps \cdot \nabla \phi + \int u_\eps \cdot \nabla u_\eps \cdot \phi = \langle f, \phi \rangle
	 \qquad \forall \phi \in \tilde V_\eps := \{w \in H^1_0(D_\eps) \cap L^d \colon \dv w =0\},
\end{align}
where the space $\tilde V_\eps$ is chosen such that the nonlinear term makes sense. Furthermore, by Sobolev embedding we observe
$\tilde V_\eps = V_\eps$ for $d \leq 4$. The weak formulation of \eqref{NS.hom} is analogous.
Existence of solutions to \eqref{NS.hom} is well-known.
However, the solution is only known to be unique if $d \leq 4$ and
\begin{align}
	\label{high.viscosity}
	\|f\|_{V'} < C(d,D) .
\end{align}

If $d \leq 4$ testing with the solution $u$ yields the energy estimate
\begin{align}
	\label{NS.energy}
	\| \nabla u_i \|_{L^2}  \leq \|f\|_{V'}.
\end{align}

For more details on the stationary Navier-Stokes equations see for example \cite{Temam} and \cite{Galdi2}.

\medskip

The proof of the convergence $u_\eps \wto \uh$ in $H^1(D)$ in the case $d = 3$ is now straightforward provided \eqref{high.viscosity} holds.
Indeed, thanks to \eqref{NS.energy}, the sequence $u_\eps$ is bounded in $H^1$, and by the uniqueness of the solutions to \eqref{NS.hom}, it therefore suffices to prove that the weak limit $u^\ast$ of any subsequence of $u_\eps$ satisfies \eqref{NS.hom}. 
To this end, let $v \in C_0^\infty(D)$ with $\dv = 0$. Then, applying Lemma \ref{reduction.operator}, we know
\begin{align}
	 \int \nabla u_\eps \cdot \nabla (R_{\eps} v) & \to \int \nabla u^\ast \cdot \nabla  v + \mu  u^\ast \cdot v, \\
	\langle f, R_\eps v \rangle & \to \langle f,  v \rangle.
\end{align} 
Therefore, it remains to show
\begin{align}
	\int u_\eps \cdot \nabla u_\eps \cdot (w^\eps_k \phi) \to \int u^\ast \cdot \nabla u^\ast_k \phi.
\end{align}
However, since $2^\ast = 6 > 4$ both $u_\eps$ and $ R_\eps v$ converge strongly in $L^4$ and $\nabla u_\eps$ converges weakly in $L^2$.
Thus, the convergence above follows immediately.

In the case $d=4$ this argument just fails, since the embedding from $H^1$ to $L^4$ is not compact.
However, since by Lemma \ref{reduction.operator} also  $R_\eps v \to v$ strongly in $L^q$, for any $4 < q < \infty$, the argument works again.

\section{Estimates for the Stokes equations in annuli and in the exterior of balls} 
\label{sec:estimates.stokes}

In this section we summarize some standard results for the solutions to the Stokes equation in annular and exterior domains (see, e.g. \cite{ AllaireARMA1990a, Desvillettes2008, Galdi}).  

\begin{lem}
Let  $R > 1$,  denote $A_R:= B_R\backslash B_1$, and let $\psi  \in H^1(B_\theta) \cap C^0(\bar{B_\theta})$ satisfy $\int_{\partial B_1} \psi \cdot \nu = 0$.  Let 
$(\phi_R, \pi_R)$ and $(\phi_\infty, \pi_\infty)$ be the (weak) solutions of 
\begin{align}\label{P.annulus}
\begin{cases}
\Delta \phi_R - \nabla \pi_R = 0 \ \ &\text{in $A_R$}\\
\nabla \cdot \phi_R = 0 \ &\text{in $A_R$}\\
\phi_R= \psi \  &\text{on $\partial B_1$}\\
\phi_R= 0 \ &\text{on $\partial B_R$},
\end{cases}\ \ \
\begin{cases}
\Delta \phi_\infty - \nabla \pi_\infty = 0 \ &\text{in $\Rd \backslash B_1$}\\
\nabla \cdot \phi_\infty= 0 \ &\text{in $\Rd \backslash B_1$}\\
\phi_\infty= \psi \ &\text{on $\partial B_1$}\\
\phi \to 0 \ \ \ &\text{for $|x| \to +\infty$}.
\end{cases}
\end{align}
Then,
\begin{equation}\label{estimate.stokes.annulus}
\begin{aligned}
\|\pi_R \|_{L^2(A_R) / \R } + \| \nabla \phi_R \|_{L^2(A_R)} &\leq C_1 \bigl( \| \nabla \psi \|_{L^2(A_R)} + \| \psi \|_{L^2(A_R)} \bigr),\\
 \| \phi_R \|_{C^0(\bar A_R)}&\leq C_1 \| \psi \|_{C^0(\partial B_1)},
\end{aligned}
\end{equation}
with $C_1=C_1(d, R)$. Moreover,
\begin{equation}\label{stokes.infty}
\begin{aligned}
\| \pi_\infty \|_{L^2(\Rd \setminus B_1)} +   \| \nabla \phi_\infty \|_{L^2(\Rd \setminus B_1)}
& \leq C_2 (\| \nabla \psi \|_{L^2(A_2)} + \| \psi \|_{L^2(A_2)}),\\
\| \phi_\infty \|_{C^0} &\leq C_2 \| \psi \|_{C^0( \partial B_1)},
\end{aligned}
\end{equation}
with $C_2 = C_2(d)$. Furthermore, 
\begin{equation}\label{infty.behaviour}
 \begin{aligned}
|\phi_\infty(x)| \leq C_2 \| \psi \|_{C^{0}(\partial B_1)}|x|^{2-d},
\end{aligned}
\end{equation}
and, if $\nabla \cdot \psi = 0$ in $B_1$,\footnote{This assumption is not needed, but makes the proof slightly simpler.}
\begin{equation}\label{infty.behaviour.1}
 \begin{aligned}
|\nabla \phi_\infty(x)| \leq C_2 \|\psi\|_{H^1(B_2)} |x|^{1-d} \qquad \text{for all } |x| \geq 3.
\end{aligned}
\end{equation}
\end{lem}

\begin{proof}
The existence and uniqueness of solutions to both problems in \eqref{P.annulus} together with the first estimate in both \eqref{estimate.stokes.annulus} and \eqref{stokes.infty} is a standard result \cite{Galdi}[Section IV and V].
The second estimate in both \eqref{estimate.stokes.annulus} and \eqref{stokes.infty} can be found in \cite{Maremonti1999}[Theorem 5.1 and Theorem 6.1]. Estimate \eqref{infty.behaviour} can be found in \cite{Maremonti1999}[Theorem 6.1], too.

\medskip

To prove \eqref{infty.behaviour.1}, we extend $\phi_\infty$ by $\psi$ inside $B_1$ and $\pi_\infty$ by $0$ inside $B_1$.
Then, by \eqref{stokes.infty}
\begin{align}
\begin{cases}
	- \Delta \phi_\infty + \nabla \pi_\infty = f &\quad \text{in } \Rd\\
	\nabla \cdot \phi_\infty = 0 &\quad \text{in } \Rd
\end{cases}
\end{align}
for some $f \in \dot H^{-1}(\Rd)$,
with
\begin{align}
	\supp f &\subset \overline{B_1}, \\
	\|f\|_{\dot H^{-1}(\Rd)} &\lesssim \|\psi\|_{H^1(B_2)}.
\end{align}
Here, $\dot H^{-1}(\Rd)$ is the dual of the homogeneous Sobolev space
\begin{align}
	\dot H^1(\Rd) := \Bigl\{ v \in L^{\frac{2d}{d-2}}(\Rd) \colon \nabla v \in L^2(\Rd) \Bigr\}, \qquad \|\cdot\|_{\dot H^1(\Rd)} := \|\nabla \cdot\|_{L^2(\Rd)}.
\end{align}
Hence, with $U$ being the fundamental solution of the Stokes equations
we have
\begin{align}
	\phi_\infty(x) = (U \ast f)(x).
\end{align}
The fundamental solution satisfies
\begin{align}
	|D^\alpha U(x)| \lesssim C(d,|\alpha|) |x|^{2 - d - |\alpha|}.
\end{align}
Using the compact support of $f$, and letting $\eta \in C_c^\infty(B_2)$ be a cut-off function with $\eta = 1 $ in $B_1$, we deduce
for all $|x| > 3$
\begin{align}
	|\nabla \phi_\infty(x)| &= |\langle \eta \nabla U(x-\cdot),f\rangle_{H^1,\dot H^{-1}}| \\
	& \leq \|\eta \nabla U(x-\cdot)\|_{\dot H^1(\Rd)} \|f\|_{\dot H^{-1}(\Rd)} \\
	& \lesssim C_3 \|\psi\|_{H^1(B_2)} |x|^{1-d}.
\end{align}
This proves \eqref{infty.behaviour.1}.
\end{proof}
\section{Some results on Strong Law of Large Numbers}
\label{sec:SLLN}

 For the reader's convenience, we list below some of the results proven in \cite{GHV1}[Section 5] on Strong Law of Large Numbers for a general marked point process and
 which we use throughout this paper. We adapt these statements to our special case of $\Phi$ being a Poisson process with intensity $\lambda > 0$ (see also Section
 \ref{setting}).  

 \begin{lem}\label{SLLN.general.pp}
Let $(\Phi, \rr)$ be as in Section \ref{setting}.
Then, for every bounded set $B \subset \Rd$ {which is star-shaped with respect to the origin}, we have 
 \begin{align}\label{average.number}
 \lim_{\eps \downarrow 0^+} \eps^d N^\eps(B)= \lambda |B| \ \ \ \ \text{almost surely},
 \end{align}
and
\begin{align}\label{SLLN.convergence.pp}
\lim_{\eps \downarrow 0^+} \eps^{d}\sum_{ z_i \in \Phi^\eps(B)} \rho_i^{d-2} = \lambda \langle \rho^{d-2} \rangle |B| \ \ \ \ \text{almost surely.}
\end{align}
 
 \smallskip
 
Furthermore, for every $\delta < 0$ the process $\Phi_\delta$ obtained from $\Phi$ as in \eqref{thinned.process} satisfies the analogues of \eqref{SLLN.convergence.pp}, \eqref{average.number} and
\begin{align}\label{convergence.average}
\lim_{\delta \downarrow 0^+}\langle N_\delta(A) \rangle = \lambda |A|
\end{align}
for every bounded set $A \subset \Rd$.

\end{lem}

  \begin{lem}\label{l.LLN.index.set}
	In the same setting of Lemma \ref{SLLN.general.pp}, let $\{ I_\eps \}_{\eps > 0}$ be a family of collections of points such that $I_\eps \subset \Phi^\eps( B)$ and
       \begin{align}\label{small.index.set}
       \lim_{\eps \downarrow 0^+} \eps^d\# I_\eps = 0 \ \ \ \text{almost surely.}
       \end{align}
       Then,
	\begin{align}
		\lim_{\eps \downarrow 0^+}{ \eps^d} \sum_{z_i \in I_\eps} \rho_i^{d-2} \to 0 \quad \text{almost surely.}
	\end{align}
\end{lem}

  \begin{lem}\label{l.w.independent}
In the same setting of Lemma \ref{SLLN.general.pp}, let us assume that in addition the marks satisfy $\langle \rho^2(d-2) \rangle < +\infty$. 
For $z_i \in \Phi$ and $\eps >0$, let $r_{i,\eps} > 0$, and assume there exists a constant $C>0$  such that for all $z_i \in \Phi$ and $\eps >0$
\begin{align}
	r_{i,\eps} \leq C \eps.
\end{align}
Then, almost surely, we have
\begin{align*}
\lim_{\eps \downarrow 0^+ }\sum_{z_i \in \Phi^\eps(B)} \rho_i^{d-2} \frac{\eps^d}{r_{i,\eps}^d}\int_{B_{r_{i,\eps}}(\eps z_i)} \zeta(x) \, dx 
= |B_1| \lambda \langle \, \rho^{d-2} \, \rangle \int_{B} \zeta(x) \, dx,
\end{align*}
for every $\zeta \in C_0^{1}(B)$, where $B_1 \subset \Rd$ denotes the unit ball. 
 \end{lem}

\bigskip 
 
\section*{Acknowledgements}
The authors acknowledge support through the CRC 1060 (The Mathematics of Emergent Effects) that is funded through the German Science Foundation (DFG), and the Hausdorff
Center for Mathematics (HCM) at the University of Bonn.

\bibliographystyle{amsplain}
\bibliography{giunti.RS}
\end{document}